\documentclass[11pt,reqno]{amsart}
\usepackage[left=1.08in, right=1.08in, top=1in, bottom=1in, includefoot]{geometry}
\usepackage{lipsum}
\usepackage{anyfontsize}
\usepackage{graphicx}
\usepackage{mathrsfs} 
\usepackage{amssymb}   
\usepackage{amsthm}  
\usepackage{bbm}
\usepackage[shortlabels]{enumitem}
\usepackage{xcolor}
\usepackage{soul}
\usepackage{scalerel}
\usepackage{float}
\newtheorem{thm}{Theorem}[section]

\newtheorem{prop}[thm]{Proposition}
\newtheorem{cor}[thm]{Corollary}
\newtheorem{Def}[thm]{Definition}
\newtheorem{lemma}[thm]{Lemma}
\newtheorem{remark}[thm]{Remark}

\usepackage{pgfplots}
\pgfplotsset{compat=1.16}
\definecolor{aqua}{rgb}{0.0, 0.44, 1.0}
\usetikzlibrary{positioning}
\usetikzlibrary{arrows}
\usepackage{stmaryrd}
\usepackage{subcaption}
\usetikzlibrary{arrows.meta}

\newtheorem{assumption}[thm]{Assumption}

\usepackage{mathtools}
\usepackage{esint}

\newcommand{\ve}{{\varepsilon}}
\newcommand{\N}{{\mathbb{N}}}

\newcommand{\R}{{\mathbb{R}}}

\numberwithin{equation}{section}
\title{The shallow--water limit of linear water waves over a discontinuous bottom}

\author[M. O. Paulsen]{Martin Oen Paulsen} 
\address{Institut de Mathématiques de Bordeaux\\ Université de Bordeaux et CNRS UMR 5251\\ 351 Cours de la Libération \\ 33405 Talence Cedex \\ France}\email{pmartinoen@gmail.com}

\date{\today}
\begin{document}
	
	\begin{abstract}  
		
		We rigorously justify the zero--dispersion limit of linear water waves over a sharp bottom step, recovering the shallow-water equations as a transmission problem, and describe the localized dispersive effects generated near the discontinuity. To obtain estimates uniform in the shallow-water parameter, we study an abstract functional framework for the linear water waves equations in two-dimensional fluid domains with a Lipschitz curved bottom. The framework is based on fractional powers of the Dirichlet--Neumann operator, and we characterize low-regularity energy spaces using adapted Rellich identities. At higher regularity, the associated velocity potential develops corner singularities. 
		
		For the step geometry, we use an explicit conformal mapping to obtain a version of Grisvard's shift theorem, where we derive an explicit formula for the singular part of the potential. The dependence on the horizontal variable reveals a narrow transition scale generated by the step. 	The convergence toward the shallow-water transmission problem holds away from this transition region, while a localized dispersive correction remains near the step. We illustrate this discrepancy numerically and identify the resulting step-induced dispersive boundary layer. Building on these observations, we consider the strongly dispersive regime and derive a consistent model that explicitly describes this boundary layer for a small step.

	\end{abstract}

	\maketitle 
	\section{Introduction}
	\subsection{Shallow water modeling for rough bathymetries} The modeling of long waves propagating over an irregular seabed is a central topic in coastal engineering. Variations in depth can alter the wave field through processes such as shoaling, reflection, transmission, and scattering, which have important consequences for coastal hazards and marine structures. The shallow-water equations serve as the standard model for these phenomena, describing depth-averaged flow when the wavelength is large compared with the water depth in an irrotational, incompressible fluid. Assessing the validity of this approximation over irregular bathymetries requires relating it to the full water-wave problem from which it is derived.

	In this work, we focus on the interaction between dispersion and an irregular bathymetry for the linear water-wave equations. The effect of dispersion is measured  by the characteristic water depth $H$ and characteristic  wavelength $\lambda$ through the shallow-water parameter
	\begin{equation*}
		\mu=\frac{H^2}{\lambda^2}.
	\end{equation*}
	In nondimensional variables, small-amplitude waves in a two-dimensional incompressible and irrotational fluid governed by the Bernoulli equation are described by the linear water-waves system in the Zakharov--Craig--Sulem formulation  \cite{Zakharov68,Craig_Sulem_Sulem_92,Craig_Sulem_93} as: 
	\begin{equation}\label{Eq: Linear ww} 
		\left\{
		\begin{aligned}
			\partial_t \zeta - \frac{1}{\mu} \partial_y \psi^{\mathfrak{h}} & = 0 \\ 
			\partial_t \psi + \zeta & = 0 
		\end{aligned} 
		\right.
		\qquad \text{in} \quad (0, \infty) \times \Gamma^{\rm D}.
	\end{equation}
	Here $\zeta$ denotes the free surface elevation and $\psi^{\mathfrak{h}}$ is the velocity potential that is defined within the strip-like fluid domain $\Omega$ with a flat surface $\Gamma^{\rm D}$ and Lipschitz curve bottom boundary $\Gamma^{\rm N}$ (see Definition \ref{Def: Assumption 1} for a precise description of the domain). The potential is the harmonic extension of its trace on the top surface, denoted by $\psi$, and in non-dimensional variables solves
	\begin{equation}\label{Eq: phi h}
		\left\{\begin{aligned}
			\Delta^{\mu} \psi^{\mathfrak{h}} &= 0 &&\qquad  \text{in}   \quad \:\: \Omega
			\\ 
			\psi^{\mathfrak{h}}  & =  \psi  &&\qquad  \text{on}  \quad \:\: \Gamma^{\rm D}
			\\ 
			\partial_{\rm n}^{\mu} \psi^{\mathfrak{h}} & = 0 &&\qquad  \text{on}  \quad \:\: \Gamma^{\rm N},
		\end{aligned}\right.
	\end{equation}
	with the twisted Laplacian defined by $\Delta^{\mu} = \mu \partial_x^2 + \partial_y^2$ and ${\rm n}$ denotes the outward unit normal, defined almost everywhere on the Lipschitz boundary, with $\partial_{\rm n}^{\mu}  = {\rm n} \cdot \left( \mu \partial_x, \partial_y\right)^{\rm T}$.

	In this configuration, the water-wave equations remain well-posed, even in the nonlinear case. This was proved in the seminal work of Alazard, Burq, and Zuily \cite{AlazardBurqZuily11} (see also \cite{AlazardBurqZuily_Wall_16}), where part of the argument relies on localized estimates for the harmonic extension away from the bottom. From the perspective of shallow-water modeling, however, this approach has the drawback that vertical localization creates a singular dependence on the shallow-water parameter. This will influence the existence time, and one way to avoid this issue is to use a global straightening of the fluid domain and to carefully track the dependence of estimates on the bottom parametrization. This is done in \cite{Alvarez-SamaniegoLannes08a} for sufficiently regular bottoms represented as graphs, where the authors establish estimates uniform in the shallow water parameter and use them to rigorously derive and justify several classical shallow-water asymptotic models from the full water-wave equations. A key step in the analysis is the reduction of the elliptic problem to a fixed reference strip, where the leading-order problem can be solved explicitly and the remainder controlled in Sobolev spaces; see also \cite{WWP} for a detailed presentation of this strategy. We also mention large-time well-posedness results in the case of sufficiently regular bathymetries with large topography variations: for the shallow-water equations we refer to \cite{BreschMetivier10}; for Boussinesq--Peregrine models read \cite{MesognonGireau17BP}; and for the nonlinear capillary water-wave equations, see \cite{MesognonGireau17WW}. However, the graph-based transformation used in these works does not apply here, since the straightening diffeomorphism used in \cite{WWP} introduces derivatives of the bottom parametrization into the transformed elliptic problem, and these derivatives then enter the estimates used in the shallow-water expansion.
	
	Conformal coordinates offer an alternative way to treat rough bathymetry in two dimensions. The fluid domain is mapped onto a flat strip, where the conformal map remains analytic along the free surface even when the bottom contains corners and allows for shallow-water expansion to be performed. This approach has a long history in the study of water waves over variable bathymetries. Fitz-Gerald \cite{FitzGerald76} used conformal coordinates for the linear water-wave problem, while Hamilton \cite{Hamilton77} used them for the formal derivation of long-wave models over rapidly varying depths. The method was also further developed by Nachbin, Papanicolaou, and Fokas \cite{NachbinPapanicolaou92, Nachbin03, FokasNachbin12} to formally derive long-wave and Boussinesq-type models for rapidly varying bathymetries.
	
	For polygonal bathymetries, including a vertical step, Cathala \cite{Cathala2013Polygonal, Cathala2013Thesis} introduced the shallow-water parameter explicitly into the conformal transformation and derived a leading-order expansion of the Dirichlet--Neumann operator together with a remainder estimate. Although the conformal map is smooth along the free surface, the estimates presented are not uniform with respect to $\mu$. Along the free surface, the conformal map transitions between the two constant-depth regimes over a region of size $\mathcal{O}(\sqrt{\mu})$, so that each differentiation produces a factor of order $\mathcal{O}(1/\sqrt{\mu})$. This loss of uniformity is consistent with the classical time-harmonic scattering theory, which shows that the local adjustment near a step is carried by evanescent modes, which decay exponentially away from the transition \cite{Newman65, MeiBlack69, EvansLinton94}. One may therefore expect a similar localization for general unsteady waves, although establishing this rigorously is delicate due to the nonlocal character of the water-wave problem.

	A complementary approach to handle geometries with sharp transitions is to formulate the reduced equations directly as initial boundary value or transmission problems. This point of view was introduced in the context of wave--structure interactions by Lannes in \cite{Lannes17} and subsequently developed for general one-dimensional hyperbolic transmission problems \cite{IguchiLannes21}, including the nonlinear shallow-water equations over a discontinuous topography. This framework was further developed in the context of an oscillating water column to study wave-energy conversion \cite{BocchiHeVergaraHermosilla21, BocchiHeVergaraHermosilla23}. In the more dispersive regime, the model includes a localized boundary layer near the interface. This phenomenon was exhibited for a Boussinesq system with transmission conditions in \cite{BreschLannesMetivier2021}, where the layer is concentrated on a scale $\mathcal{O}(\sqrt{\mu})$ and its derivatives exhibit a singular dependence on the dispersive parameter. Related dispersive boundary layers have since appeared in several initial-boundary and transmission problems for Boussinesq-type models; see \cite{LannesWeynans2020, BeckLannes22, BeckLannesWeynans25, RigalBonnetonLannes2026, BeckCrouseillesMartaud2026, BeckContentinMartaud2026,Bocchi26}. In these works, however, the reduced shallow-water or dispersive model is taken as the starting point. The rigorous derivation of these models from the full water-wave equations, including the emergence of dispersive boundary layers, remains open.
	
	\subsection{Description of the results and organization of the paper.} In this paper, we address the question of the rigorous derivation of asymptotic models from the linear water-wave equations in the case of a discontinuous bottom. The first step is to establish estimates for the linear water-wave equations that are uniform with respect to the shallow-water parameter.  In Section \ref{Sec: WW}, this is achieved by recasting the problem in an abstract functional framework defined through fractional powers of the Dirichlet--Neumann operator in $L^2(\R)$, which is defined by
	\begin{equation*}
		\mathcal G_\mu \psi
		=
		\frac{1}{\mu}
		\partial_y \psi^{\mathfrak h}_{|_{\Gamma^{\rm D}}},
	\end{equation*}
	where $\psi^{\mathfrak h}$ solves \eqref{Eq: phi h}. The construction is developed for fluid domains with a Lipschitz bottom boundary, with future applications to more general geometries in mind. We refer to Theorem \ref{thm: Wp ww} for a precise statement of this preliminary result. 
	
	The next step is to characterize these abstract, nonlocal spaces in terms of standard Sobolev spaces. For a fixed value of the shallow-water parameter, it is well-known that the Dirichlet--Neumann operator is equivalent to the Fourier multiplier $|{\rm D}|$ in high frequencies, up to smoother terms, even for rough bathymetries \cite{AlazardBurqZuily11}. The estimates obtained in this way, however, do not keep track of the dependence on $\mu$ and the behavior at low frequency as in \cite{WWP} for the regular bottom case. Our first result in this direction is given in Subsection \ref{Sec: char X1}, where we use a $\mu-$dependent Rellich identity, as in \cite{LannesMing26}, to prove in Proposition \ref{Prop: H1 characterization} that $ {\mathcal G}_{\mu} \psi $ and its square root is in $ L^2(\R)$ if and only if  $\partial_x \psi \in L^2(\R)$. As a consequence, this gives a well-posedness result for the linear water-wave equations in standard Sobolev spaces, with a precise control over the dependence on the shallow-water parameter.

	The situation becomes more delicate at the next level of regularity. In this case, one additionally requires $\mathcal G_\mu\psi\in\dot H^{1/2}_{\mu}(\R)$ and would like to identify this condition with $\psi\in\dot H^{3/2}_{\mu}(\R)$. The underlying difficulty is that one typically uses that $\nabla^\mu\psi^{\mathfrak h}\in H^1(\Omega)$, which in general fails in the presence of corners \cite{Grisvard85, Grisvard92, Dauge88, BorsukKondratiev06, KozlovMazyaRossmann97}. We therefore simplify the geometry in Section \ref{Sec: elliptic est for step} to a fluid domain with a sharp step and first study how the corresponding corner singularity depends on the shallow-water parameter. Adapting the classical works on corner singularities, we decompose the solutions to the elliptic problem \eqref{Eq: phi h}  and prove an $ H^2-$estimate that keeps precise track of the dependence on $\mu$ and on the size of the step. The proof relies on the explicit Schwarz--Christoffel map mentioned above, where the decomposition is determined by the local behavior of the conformal map near the corner. We therefore expect the same strategy to extend to more general corner geometries, a question that we leave for future work.

	Section \ref{Sec: derivation of sw} is devoted to the derivation of the shallow-water equations for a sharp step at $x=0$. In view of the preceding discussion, the limiting system is formulated as a transmission problem, with the shallow-water equations posed on the two constant-depth regions 
	\begin{equation}\label{Eq: Linear shallow water transmission} 
		\left\{\begin{aligned} \partial_t \zeta + \partial_x (h \overline{v})& = 0 \\ \partial_t \overline{v} + \partial_x \zeta & = 0, 
		\end{aligned}\right. 
		\qquad \text{in} \quad (0, \infty) \times \R \backslash\{0\}
	\end{equation}
	where $\overline{v}$ denotes the average horizontal velocity field and $h$ is the height. This system requires two transmission conditions at the position of the step \cite{IguchiLannes21}. The natural candidates are the continuity of $\zeta$ and of the horizontal discharge $q = h \overline{v}$, since these quantities remain regular for the water-wave equations (see Subsection \ref{Sec: discharge} for details on the regularity of $q$). The system is therefore completed with initial data and the following transmission conditions
	\begin{equation}\label{Eq: trans cond}
		\zeta(t, x = 0^-)= \zeta(t, x = 0^+)
		\quad \text{and} \quad
		q(t, x = 0^-)= q(t, x = 0^+).
	\end{equation}
	The result is given in Theorem \ref{Thm: Cv Step case}. The main difficulty in proving convergence, together with a quantitative rate in $\mu$, is to control the remainder in the shallow-water expansion of the Dirichlet--Neumann operator. This remainder involves higher-order derivatives of the harmonic extension and limits the convergence result to $H^1(\R_-)\oplus H^1(\R_+)$. In conformal coordinates, this loss of precision with respect to $\mu$ is concentrated in the transition region associated with the step, making the localized dispersive correction the main obstacle to obtaining a uniform estimate in a finer norm.

	It is possible to consider the shallow-water system at higher piecewise regularity and thereby seek convergence in a stronger norm. For instance, if we let $\llbracket \cdot \rrbracket$ denote the jump across the step, then assuming in addition the first-order compatibility conditions
	\begin{equation*}
		\llbracket h\partial_x \zeta  \rrbracket
		=
		\llbracket \partial_xq \rrbracket
		=
		0,
	\end{equation*}
	one obtains well-posedness in $H^2(\R_-)\oplus H^2(\R_+)$ \cite{IguchiLannes21}. This higher regularity, however, reveals a structural obstruction to convergence in a stronger global norm. The first transmission condition implies that $\partial_x\zeta$ can be discontinuous at $x=0$. By contrast, for smooth initial data, the water-wave solution is smooth. The pointwise error for the derivative can therefore remain large near the interface, even when it is small in the $H^1(\R_-)\oplus H^1(\R_+)-$norm.
	
	In Section \ref{Sec: num}, we illustrate numerically how dispersion smooths this jump and creates a thin transition region in which the two models differ. Then, in the final section, we provide an analytical description of this localized error in the vanishing-step regime. By decoupling the shallow-water parameter $\mu$ from the size of the step, we derive in Proposition \ref{Prop: ww + bl} a new system, consistent with the full linear water-wave equations, that recovers the latter as the step height tends to zero. The resulting model allows for strong dispersive effects and provides an explicit description, derived from first principles, of the boundary-layer correction near the step.

	\subsection{Notations} We provide some notations that will be used consistently throughout the paper. 
	
	\begin{itemize}
		\item  We let $C>0$ be a  constant  independent from the small parameters $\mu, \gamma \in (0,\tfrac{1}{2})$ and that may change from line to line.\\
		\item  As a shorthand, we use the notation $a \lesssim b$ to mean $a \leq C\, b$.\\
		\item  Let $x$ and $y$ denote the horizontal and vertical coordinates. Let $\partial_x$ and $\partial_y$ be the corresponding partial derivatives where we also write $\nabla =(\partial_x, \partial_y)^{\rm T}$. Also, we define the scaled gradient $\nabla^{\mu} = (\sqrt{\mu}\partial_x, \partial_y)^{\rm T}$\\
		\item We say $f$ is a  Schwartz function $\mathscr{S}(\mathbb{R})$, if $f \in C^{\infty}(\mathbb{R})$ and satisfies for all $\alpha, \beta \in \mathbb{N}$,
		\begin{equation*}
			\sup \limits_{x} |x^{\alpha} \partial_x^{\beta} f | < \infty.
		\end{equation*}
		\item For any tempered distribution $f \in \mathscr{S}'(\R) $ its Fourier transform will be defined by $\mathcal{F}(f)(\xi) = \hat{f} (\xi) = \int_{\R} e^{-  i x \xi } f(x) \: \mathrm{d}x$ and its inverse Fourier transform is defined by $\mathcal{F}^{-1}(f)(x) = \check{f} (x) = \frac{1}{2\pi } \int_{\R} e^{ i x \xi } f(\xi) \: \mathrm{d}\xi$.\\
		\item Let $m:\R \rightarrow \mathbb{R}$ be a smooth function. Then we will use the notation $m(\rm{D})$ for a multiplier defined in frequency by $\widehat{m(\mathrm{D}) f}(\xi) = m(\xi) \hat{f}(\xi)$.\\
		\item For any $s \in \mathbb{R}$ we call the multiplier $ \widehat{|{\rm D}|^{s}f} (\xi)= |\xi|^s \hat{f}(\xi)$ the Riesz potential of order $-s$. \\
		\item For any $s \in \mathbb{R}$ we call the multiplier $\langle {\rm D} \rangle^s = (1+{\rm D}^2)^{\frac{s}{2}}$ the Bessel potential of order $-s$. Moreover, the Sobolev space $H^s(\mathbb{R})$  is equivalent to the weighted $L^2-$space; $\|f\|_{H^s} = \|\langle {\rm D} \rangle^sf\|_{L^2}$. We also find it  convenient to define ${\mathfrak{P}}({\rm D})$ which is a multiplier associated to the symbol:
		\begin{equation*}\label{scaled J_mu}
			\mathcal{F}({\mathfrak{P}}({\rm D}))(\xi) = \frac{|\xi|}{\langle\sqrt{\mu} |\xi| \rangle^{1/2}},
		\end{equation*}
		and use it to define the semi-norm $|f|_{\dot{H}^{1/2}_{\mu}(\R)} = |{\mathfrak{P}}({\rm D})f|_{L^2(\R)}$.\\
		\item For $s\geq 0$ we also define $H^s(\R_{\pm}) = H^{s}(\R_-) \oplus H^s (\R_+)$. Moreover, for  $s>1/2$ and $f \in H^{s}(\R_{\pm})$ we define $\llbracket f \rrbracket  = f(0+) - f(0-)$.\\
		\item For all $s\in {\mathbb R}$, $s\geq 0$, we denote by $H^s(\Omega)$ the standard Sobolev space on $\Omega$.\\
		\item For all $s\geq 0$, we define the Beppo-Levi space $\dot{H}^{s+1}(\Omega)$ as
		$$
		\dot{H}^{s+1}(\Omega) =\{\phi\in L^1_{\rm loc}(\Omega) \: : \:  \nabla \phi\in H^s(\Omega)^2\},
		$$
		with associated semi-norm $\Vert \phi\Vert_{\dot{H}^{s+1}(\Omega)} = \Vert \nabla \phi\Vert_{H^s(\Omega)}$.		\\ 

		\item If $A$ and $B$ are two operators, then we denote the commutator between them by $[A, B] = AB - BA$.	
		
	\end{itemize}

	\section{The linear water waves equations with rough bathymetry}\label{Sec: WW}
	In this section, we study the dependence of solutions to the linear water-wave equations on the shallow-water parameter for a rough bottom in an abstract functional framework.  To be precise, the domain is defined by:
	\begin{assumption}\label{Def: Assumption 1}
		Let the fluid domain be denoted by $\Omega$ with $\Gamma = \partial\Omega$. Assume the domain is strip-like, bounded above by $\Gamma^{\rm D} = \R \times \{0\}$ and below by a Lipschitz curve $\Gamma^{\rm N}$. We further assume that the two boundary components are uniformly separated.
	\end{assumption}
		\begin{figure}[h]
		\centering
		\begin{tikzpicture}
			 
			\coordinate (TopL) at (-5,0);
			\coordinate (TopR) at (5,0);
			 
			\coordinate (B0)  at (-5.00,-2.45);
			\coordinate (B1)  at (-3.95,-3.35);
			\coordinate (B2)  at (-3.05,-2.55);
			\coordinate (B3)  at (-2.45,-2.40);
			\coordinate (B4)  at (-2.35,-3.45);
			\coordinate (B5)  at (-1.70,-3.40);
			\coordinate (B6)  at (-1.25,-3.00);
			\coordinate (B7)  at (-0.35,-3.50);
			\coordinate (B8)  at ( 0.55,-3.65);
			\coordinate (B9)  at ( 1.05,-2.95);
			 
			\coordinate (B10) at ( 0.30,-2.55);
			\coordinate (B11) at ( 0.85,-2.15);
			
			\coordinate (B12) at ( 2.05,-2.30);
			\coordinate (B13) at ( 2.45,-3.25);
			\coordinate (B14) at ( 3.05,-2.75);
			\coordinate (B15) at ( 3.65,-2.60);
			\coordinate (B16) at ( 4.00,-3.35);
			\coordinate (B17) at ( 4.30,-3.20);
			\coordinate (B18) at ( 4.80,-2.45);
			\coordinate (B19) at ( 5.00,-2.35);
			 
			\fill[aqua!20]
			(TopL)
			-- (TopR)
			-- (B19)
			-- (B18)
			-- (B17)
			-- (B16)
			-- (B15)
			-- (B14)
			-- (B13)
			-- (B12)
			-- (B11)
			-- (B10)
			-- (B9)
			-- (B8)
			-- (B7)
			-- (B6)
			-- (B5)
			-- (B4)
			-- (B3)
			-- (B2)
			-- (B1)
			-- (B0)
			-- cycle;

			\fill[gray!15]
			(B0)
			-- (B1)
			-- (B2)
			-- (B3)
			-- (B4)
			-- (B5)
			-- (B6)
			-- (B7)
			-- (B8)
			-- (B9)
			-- (B10)
			-- (B11)
			-- (B12)
			-- (B13)
			-- (B14)
			-- (B15)
			-- (B16)
			-- (B17)
			-- (B18)
			-- (B19)
			-- (5,-4.25)
			-- (-5,-4.25)
			-- cycle;
			
			\draw[brown,
			very thick,
			line join=miter
			]
			(B0)
			-- (B1)
			-- (B2)
			-- (B3)
			-- (B4)
			-- (B5)
			-- (B6)
			-- (B7)
			-- (B8)
			-- (B9)
			-- (B10)
			-- (B11)
			-- (B12)
			-- (B13)
			-- (B14)
			-- (B15)
			-- (B16)
			-- (B17)
			-- (B18)
			-- (B19);
			
			\draw[blue,very thick]
			(TopL) -- (TopR);
			
			\draw[thick]
			(-5.2,0) -- (-5,0); 
			
			\draw[->,>=Stealth,thick]
			(5,0) -- (5.5,0)
			node[right] {$x$};
			 
			\draw[->,>=Stealth,thick]
			(0,0) -- (0,1)
			node[above] {$y$};
			 
			\node[blue] at (-2.35,0.38)
			{$\Gamma^{\rm D}$};
			
			\node at (-0.10,-1.55)
			{\Large $\Omega$};
			
			\node[brown] at (2.05,-3.55)
			{\Large $\Gamma^{\rm N}$};
			
		\end{tikzpicture}
		
		\caption{The strip-like fluid domain $\Omega$ with a rough lower boundary $\Gamma^{\rm N}$.}
		\label{Fig. Rough bottom}
	\end{figure}
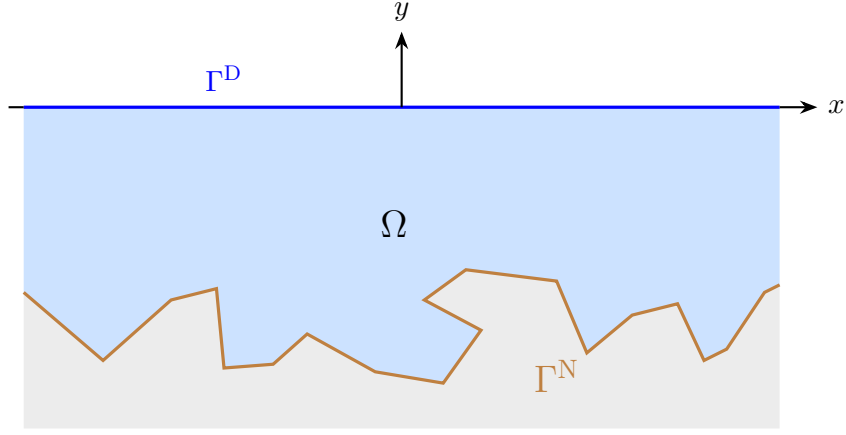

	As noted in the introduction, one of the challenges in the derivation of shallow water models, even at the level of the harmonic extension, is due to the corners in the domain. Solutions are first defined in the scaled energy space, denoted by ${\mathbb X}_{\mu}$, which is classically given by
	\begin{equation}\label{defX}
		{\mathbb X}_{\mu}
		=
		L^2(\R)\times \dot{H}_{\mu}^{1/2}(\R),
	\end{equation}
	endowed with the semi-scalar product
	$$
	\qquad \langle {\bf U}_1 , {\bf U}_2 \rangle_{\mathbb X_{\mu}}=
	\frac{1}{2}\int_{\R} \zeta_1\, \zeta_2
	+ \frac{1}{2} \int_{\R}  \psi_1 \,{\mathcal G}_{\mu}\psi_2,
	$$
	for all  ${\bf U}_1, {\bf U}_2 \in {\mathbb X_{\mu}}$. We write $\Vert \cdot \Vert_{\mathbb X_{\mu}}$ for the associated semi-norm, which is naturally characterized in terms of the Dirichlet--Neumann operator. Once we have established the well-posedness in the energy space, we will define an abstract function space for which we can deduce higher-order regularity. We then study this space in an abstract setting and provide more explicit characterizations under additional assumptions on the geometry of the domain.

	\subsection{Well-posedness in the energy space} Since the energy space is defined for the potential at the variational level, corner singularities do not arise. In fact, the characterization is an extension of Proposition 2.31 in \cite{WWP}, which considers the case where the bottom is the graph of a Lipschitz function. We present the proof here to clarify the precise dependence on the shallow water parameter.
	\begin{lemma}\label{Lemma: var est}
		Let $\Omega$ and $\Gamma$ be as in Assumption \ref{Def: Assumption 1}. Then for any $\psi \in \dot{H}^{1/2}_{\mu}(\R)$ there is a unique variational solution $\psi^{\mathfrak{h}} \in \dot{H}^1(\Omega)$ to \eqref{Eq: phi h}. 	Moreover, there is a universal constant $C>0$ such that
		\begin{equation}\label{Eq: Phi equiv psi}
			\frac{\sqrt{\mu}}{C} \vert \psi \vert_{\dot{H}^{1/2}_{\mu}(\R)} \leq 
			\|\nabla^{\mu} \psi^{\mathfrak{h}}\|_{L^2(\Omega)} 
			\leq  
			\sqrt{\mu} C \vert \psi \vert_{\dot{H}^{1/2}_{\mu}(\R)}
		\end{equation}
		and
		\begin{equation}\label{Eq: Gmu equiv psi}
			\frac{1}{C}\vert \psi \vert_{\dot{H}^{1/2}_{\mu}(\R)}^2\leq 
			\int_{\R}  \psi \, {\mathcal G}_{\mu}\psi 
			\leq C 
			\vert \psi \vert_{\dot{H}^{1/2}_{\mu}(\R)}^2.
		\end{equation}
	\end{lemma}
	\begin{proof}
		The existence and uniqueness of a variational solution $\nabla^{\mu}\psi^{\mathfrak{h}} \in L^2(\Omega)$ to \eqref{Eq: phi h} is classical. See, for instance, \cite{AlazardBurqZuily11} in the case of Lipschitz curve bathymetry or \cite{LannesMing26} for the case of floating structures. The main point is to give equivalence estimates that are sharp with respect to the shallow-water parameter. It is enough to prove \eqref{Eq: Phi equiv psi} since we have by Stokes Theorem that
		\begin{equation*}
			\int_{\Gamma^{\rm D}} \psi^{\mathfrak{h}}\partial_y\psi^{\mathfrak{h}} = \|\nabla^{\mu}\psi^{\mathfrak{h}}\|_{L^2(\Omega)}^2,
		\end{equation*}
		and would then imply \eqref{Eq: Gmu equiv psi}. For the proof of the lower bound in \eqref{Eq: Phi equiv psi}, we will argue as in the proof of Proposition 3.12 \cite{WWP} that there holds,
		\begin{equation}\label{Eq: Trace est H1/2}
			\sqrt{\mu}\vert \psi \vert_{\dot{H}^{1/2}_{\mu}(\R)} \lesssim \| \varphi \, \nabla^{\mu} \psi^{\mathfrak{h}} \|_{L^2(\Omega)},
		\end{equation}
		for some smooth cut-off function $\varphi = \varphi(y)$ equal to one at the surface and supported away from the bottom. Therefore, the result also holds for a Lipschitz bathymetry.  For the upper bound, we simply apply Lemma 2.34 \cite{WWP} which provides a lifting $\psi^{\dagger}: \R\times[-1, 0]\rightarrow \R$ that satisfies 
		\begin{equation}\label{Eq: Lifting Lannes}
			\|  \nabla^{\mu} \psi^{\dagger} \|_{L^2\left(\R \times \left(-1,0\right)\right)}
			\lesssim \sqrt{\mu} \vert \psi \vert_{\dot{H}^{1/2}_{\mu}(\R)}.
		\end{equation}
		Taking its restriction on $\Omega$ and defining $u \coloneq \psi^{\mathfrak{h}}- {\psi^{\dagger}}_{|_{\Omega}}$ gives
		\begin{align*}
			\int_{\Omega} |\nabla^{\mu} u |^2 = - \int_{\Omega} \nabla^{\mu}  \psi^{\dagger} \cdot \nabla^{\mu}  u.
		\end{align*}
		So that the triangle inequality and $\| \nabla^{\mu} u \|_{L^2 (\Omega)} \lesssim \sqrt{\mu} |\psi|_{\dot{H}^{1/2}_{\mu}(\R)}$ gives the result. 
	\end{proof}
	\begin{remark}\label{Remark: Localization issue}
		The control of the main unknown, $\psi$, only depends on the potential in the neighborhood of the surface. Therefore, it is possible to have localized higher-order elliptic estimates for a rough bathymetry \cite{AlazardBurqZuily11}. However, the localization is not small with respect to the shallow-water parameter due to commutators between $\partial_y$ and $\varphi$, which are of order one. This is why one uses global transformations, and is why the problem with rough bathymetries is challenging, even in the linear case.
	\end{remark}
	Moreover, we have for all ${\bf U}\in {\mathbb X}_{\mu}$ that  $\Vert {\bf U} \Vert_{\mathbb X_{\mu}}^2={\mathcal H}({\bf U})$, with ${\mathcal H}$ the Hamiltonian defined by 
	\begin{equation*}
		{\mathcal H}({\bf U})=\frac{1}{2} \int_{\R}\zeta^2+\frac{1}{2} \int_{\R} \psi \, {\mathcal G}_{\mu}\psi.
	\end{equation*}
	We may now use Lemma \ref{Lemma: var est} to deduce a simple uniform estimate on the solution of the linear water waves equations in the energy space whose proof is classical. 
	\begin{cor}\label{Cor: Wp lin ww}
		Let $\Omega$ and $\Gamma$ be as in Assumption \ref{Def: Assumption 1}, $\mathbb{X}_{\mu}$  defined by \eqref{defX}. Then for any $\mathbf{U}^{\rm in} = (\zeta^{\rm in}, \psi^{\rm in})^{\rm T}  \in \mathbb{X}_{\mu}$, there is a unique solution  
		\begin{equation*}
			\mathbf{U} = (\zeta, \psi)^{\rm T} \in C(\R_+;\mathbb{X}_{\mu}) 
		\end{equation*} 
		to \eqref{Eq: Linear ww} with initial condition ${\bf U}(0)={\bf U}^{\rm in}$ and the solution conserves the energy
		\begin{equation}\label{Eq: Energy est}
			\hspace{-1.8cm} \forall \: t \geq 0 \qquad   \Vert {\bf U}(t)\Vert_{\mathbb X_{\mu}}=\Vert {\bf U}^{\rm in}\Vert_{\mathbb X_{\mu}}.
		\end{equation}
	\end{cor}  	 
	\subsection{Higher-order regularity in an abstract function space}
	To study higher-order regularity of solutions to \eqref{Eq: Linear ww}, it is convenient to introduce a scale of functional spaces adapted to the problem.
	We know by Proposition A.14 in \cite{WWP} that ${\mathcal G}_{\mu}$ admits a self-adjoint realization on $L^2(\R)$ with domain $H^1(\R)$. Since it is also positive, it admits a square root, denoted ${\mathcal G}_{\mu}^{1/2}$, which is such that $\vert {\mathcal G}_{\mu}^{1/2} \psi \vert_{L^2(\R)}^2=\langle \psi, {\mathcal G}_{\mu}\psi \rangle_{\dot{H}^{1/2}\times (\dot{H}^{1/2})'}$. Since this latter expression makes sense for $\psi\in \dot{H}^{1/2}(\R)$, we still use  the notation $\vert {\mathcal G}_{\mu}^{1/2} \psi \vert_{L^2(\R)}$ for such $\psi$ even though it is not necessarily in $L^2(\R)$. More generally, if $n$ is an odd integer, we write
	$$
	\forall n=2l+1,\qquad \vert {\mathcal G}_{\mu}^{n/2}\psi \vert_{L^2(\R)}^2= \langle {\mathcal G}_{\mu}^l\psi, {\mathcal G}_{\mu} ({\mathcal G}_{\mu}^l \psi) \rangle_{\dot{H}^{1/2}\times (\dot{H}^{1/2})'},
	$$
	which is well defined whenever ${\mathcal G}_{\mu}^l \psi \in \dot{H}^{1/2}_{\mu}(\R)$; by abuse of language, we then say that ${\mathcal G}_{\mu}^{n/2}\psi\in L^2(\R)$. We can now introduce the appropriate functional spaces.
	\begin{Def}\label{defhighreg}
		Let $\Omega$ and $\Gamma$ be as in Assumption \ref{Def: Assumption 1}, and $n\in {\mathbb N}^*$. \\
		{\bf i.} We define
		\begin{align*}
			\dot{\mathcal H}_{\mu}^{n/2}(\R) &=\{ f\in \dot{H}^{1/2}(\R), \quad \forall 1\leq j\leq n, \quad {\mathcal G}_{\mu}^{j/2}f \in L^2(\R)\},
			\\
			{\mathcal H}_{\mu}^{n/2}(\R) &= L^2(\R)\cap \dot{\mathcal H}_{\mu}^{n/2}(\R),
		\end{align*}
		respectively endowed with their canonical semi-norm and norm.\\
		\noindent
		{\bf ii.} The space ${\mathbb X}^n_{\mu}$ is defined as
		$$
		{\mathbb X}^n_{\mu}={\mathcal H}^{n/2}_{\mu}(\R) \times \dot{\mathcal H}^{(n+1)/2}_{\mu}(\R)
		$$
		and is endowed with the semi-norm
		$$
		\Vert {\bf U}\Vert_{{\mathbb X}^n_{\mu}}=\Vert {\bf U}\Vert_{\mathbb X_{\mu}}+\sum_{j=1}^n ( \vert {\mathcal G}_{\mu}^{j/2}\zeta \vert_{L^2(\R)} +\vert {\mathcal G}_{\mu}^{(j+1)/2}\psi \vert_{L^2(\R)}),
		$$
		where ${\bf U}=(\zeta, \psi)^{\rm T}$.
	\end{Def}
	
	As a consequence of Corollary \ref{Cor: Wp lin ww}, arguing as in \cite{LannesMing26, LannesPaulsen25}, one can deduce the higher-order regularity result in this abstract functional setting: 
	\begin{thm}\label{thm: Wp ww}
		Let $\Omega$, and $\Gamma$ be as in Assumption \ref{Def: Assumption 1}. Then for all ${\bf U}^{\rm in}\in {\mathbb X}^n_{\mu}$ and for any $T>0$, there exists a unique solution ${\bf U}\in \cap_{j=0}^n C^j([0,T];{\mathbb X}^{n-j}_{\mu})$
		to \eqref{Eq: Linear ww} with initial condition ${\bf U}(0)={\bf U}^{\rm in}$. Moreover, there exists $C>0$  independent of ${\bf U}^{\rm in}$ such that  
		$$ 
		\sum_{j=0}^n \Vert \partial_t^j {\bf U}\Vert_{L^{\infty}([0, T] : {\mathbb X}^{n-j}_{\mu})} \leq C \Vert {\bf U}^{\rm in}\Vert_{{\mathbb X}^n_{\mu}}.
		$$
	\end{thm}  
	Now, before we proceed with the derivation of the transmission problem with a step, we need to characterize ${\mathbb X}^1_{\mu}$ and comment on the higher-order spaces. We will do it for different geometries and apply Theorem \ref{thm: Wp ww} to give uniform estimates on the linear water waves equations. Moreover, we will deduce enough regularity to define the transmission conditions and identify the condition on the initial data in the case of a discontinuous bottom. 
	\subsection{On the characterization of ${\mathbb X}^1_{\mu}$ for a Lipschitz domain}\label{Sec: char X1}
	We will study the characterization under assumption \ref{Def: Assumption 1} using a $\mu-$dependent Rellich identity showing, as in \cite{LannesMing26}, that
	\begin{equation*}
		{\mathbb X}^1_{\mu} = H^{1/2}_{\mu}(\R) \times \dot{H}^{1}(\R).
	\end{equation*}
	The key point here is the dependencies with respect to the shallow water parameter, where the first step in the proof is to relate the normal and tangential derivatives on $\Gamma^{\rm D}$. This is classically done via Rellich identities, which hold for general Lipschitz domains.  
	\begin{lemma}
		Let $\Omega \subset \R^2$  be a Lipschitz domain with boundary $\Gamma$.
		Let $\mathbf{w} = (w_1, w_2)^{\rm T} \in H^{2}(\Omega)$, $\alpha = (\alpha_1, \alpha_2)^{\rm T} \in W^{1,\infty}(\overline{\Omega})$ and ${\rm n}$ be the exterior normal vector. Then the scaled Rellich identity reads
		\begin{align}\label{Eq: Rellich 1}
			\int_{\Gamma}
			\left( 
			{\rm I}^{\mu}{\rm n} \cdot \alpha |\mathbf{w}|^2 
			-
			2 \left( \alpha \cdot \mathbf{w} \right) 
			\left({\rm I}^{\mu}{\rm n} \cdot \mathbf{w}\right)
			\right)
			= & 
			\int_{\Omega} \left( \left(\nabla^{\mu} \cdot \alpha\right) |\mathbf{w}|^2 
			-
			2 \mathbf{w} \cdot \left(\nabla^{\mu} \alpha\right)^{\rm T}\mathbf{w} 
			- 
			2 \left(\alpha \cdot \mathbf{w}\right) \nabla^{\mu} \cdot \mathbf{w} \right)
			\\ \notag
			&
			+ 
			\mathrm{Rot}^{\mu}\left(\alpha, \mathbf{w}\right),
		\end{align}	
		where the rotational component reads
		\begin{align*}
			\mathrm{Rot}^{\mu}\left(\alpha, \mathbf{w}\right) 
			= 
			2 \alpha \cdot\left( \frac{1}{2}\nabla^{\mu} |\mathbf{w}|^2 -  \left(\mathbf{w}\cdot \nabla^{\mu} \right)\mathbf{w}\right).
		\end{align*}

	\end{lemma}

	\begin{remark}\label{Remark: rot=0}
		The rotational component is zero for $\mathbf{w} = \nabla^{\mu} \psi^{\mathfrak{h}}$.
	\end{remark}

	\begin{proof}
		Equation \eqref{Eq: Rellich 1} is a classical Rellich identity introduced in \cite{Brown94, Brown_Lanzini_Capogna08} and easily deduced by considering the following vector field:
		\begin{equation*}
			\mathbf{v} = |\mathbf{w}|^2 \alpha - 2\left(\alpha \cdot \mathbf{w}\right) \mathbf{w}. 
		\end{equation*}
		Then for ${\rm I}^{\mu} = {\rm diag}(\sqrt{\mu}, 1)$ we have that \eqref{Eq: Rellich 1} is equivalent to 
		\begin{equation*}
			\int_{\Omega} \nabla^{\mu} \cdot \mathbf{v} = \int_{\Gamma}  {\rm I}^{\mu} {\rm n} \cdot \mathbf{v}.
		\end{equation*}

	\end{proof} 
	The characterization of $\dot{\mathcal{H}}_{\mu}^1(\R)$ is now a direct application:
	\begin{prop}\label{Prop: H1 characterization}
		Let $\Omega$ and $\Gamma$ be as in Assumption \ref{Def: Assumption 1}. Then 
		$$\dot{\mathcal{H}}_{\mu}^1(\R) = {\dot{H}^1(\R)}.$$
		In addition, for  $\psi \in \dot{\mathcal{H}}_{\mu}^1(\R)$ then there is a universal constant $C>0$ such that
		\begin{equation}\label{Eq: Rellich dx to dn}
			|\partial_x \psi|_{L^2(\R)}   \leq C \left(  |\psi|_{\dot{H}^{1/2}_{\mu}} +  \sqrt{\mu} |{\mathcal G}_{\mu} \psi|_{L^2(\R)} \right).
		\end{equation}
		On the other hand, for $\psi \in \dot{{H}}^1(\R)$ then there is a universal constant $C>0$ such that
		\begin{equation}\label{Eq: Rellich dn to dx}
			\sqrt{\mu} |{\mathcal G}_{\mu} \psi|_{L^2(\R)} \leq 
			C \left(  |\psi|_{\dot{H}^{1/2}_{\mu}} +  |\partial_x \psi|_{L^2(\R)}  \right)
		\end{equation}
	\end{prop}
	\begin{remark}\label{Remark: Dn op frequency}
		The estimates above are consistent with the flat bathymetry case. In particular, this is easily seen by comparing the symbol $\widehat{\partial_x} = i \xi$ with the Dirichlet-Neumann operator for a flat bottom, $\Gamma^{\rm N} = \R \times \{ y=-1\}$, whose symbol reads
		\begin{equation*}
			\mathcal{F}\left({{\mathcal G}_{\mu}}\right)(\xi) = \frac{1}{\sqrt{\mu}} |{\xi}| \tanh(\sqrt{\mu}|{\xi}|).
		\end{equation*}
	\end{remark}

	\begin{proof}
		This is a simple adaptation of the proof in \cite{LannesMing26}. We give the proof of the estimate in two separate steps.
		\\ 
		
		\noindent
		{\bf Step 1.} \textit{Proof of \eqref{Eq: Rellich dx to dn}} Let $\alpha \in W^{1,\infty}(\overline{\Omega})$ be such that $\alpha_1 = 0$ in $\Omega$ and $\alpha_2(x,y) = \alpha_2 (y)$. Moreover, suppose $\alpha_2$ is supported away from $\Gamma^{\rm N}$ and such that $\alpha_2(0) = 1$. Then choosing the vector field to be
		$$\mathbf{w} =
		\nabla^{\mu}
		\psi^{\mathfrak{h}},$$ 
		together with \eqref{Eq: Rellich 1} and Remark \ref{Remark: rot=0} implies
		\begin{align*}
			\mu \int_{\Gamma^{\rm D}} \left(\partial_x \psi^{\mathfrak{h}}\right)^2 
			=
			\int_{\Gamma^{\rm D}} \left(\partial_y \psi^{\mathfrak{h}}\right)^2 
			+
			\int_{\Omega}\left(\left(\nabla^{\mu} \cdot \alpha\right) \left|\mathbf{w}\right|^2 
			-
			2 \mathbf{w} \cdot \left(\nabla^{\mu} \alpha\right) \cdot \mathbf{w}
			\right),
		\end{align*}
		where we used that ${\rm n} =  \mathbf{e}_y$ on $\Gamma^{\rm D}$. To conclude, we need to carefully decompose the contribution from the last term on the right-hand side. In particular, the integrand reads
		\begin{align*}
			\left(\nabla^{\mu} \cdot \alpha\right) \left|\mathbf{w}\right|^2 
			-
			2 \mathbf{w} \cdot \left(\nabla^{\mu} \alpha\right) \cdot \mathbf{w}
			& = 
			\left(\partial_y \alpha_2 \right)\left(\mu \left(\partial_x \psi^{\mathfrak{h}}\right)^2 - \left(\partial_y \psi^{\mathfrak{h}}\right)^2\right),
		\end{align*}
		%
		%
		%
		and we find that 
		\begin{align*}
			\mu \int_{\Gamma^{\rm D}} \left(\partial_x \psi^{\mathfrak{h}}\right)^2 
			\leq 	\int_{\Gamma^{\rm D}} \left(\partial_y \psi^{\mathfrak{h}}\right)^2 + 
			C \int_{\Omega}\left|\nabla^{\mu} \psi^{\mathfrak{h}}\right|^2,
		\end{align*}
		for some universal constant $C>0$ and the conclusion follows from \eqref{Eq: Phi equiv psi}. 
		\\

		\noindent
		{\bf Step 2.} \textit{Proof of \eqref{Eq: Rellich dn to dx}} The proof is the same as above, where we put $\alpha_2 (0) = -1$ to find that
		\begin{align*}
			\int_{\Gamma^{\rm D}} \left(\partial_y \psi^{\mathfrak{h}}\right)^2 
			=
			\mu \int_{\Gamma^{\rm D}} \left(\partial_x \psi^{\mathfrak{h}}\right)^2
			+
			\int_{\Omega}\left(\left(\nabla^{\mu} \cdot \alpha\right) \left|\mathbf{w}\right|^2 
			-
			2 \mathbf{w} \cdot \left(\nabla^{\mu} \alpha\right) \cdot \mathbf{w}
			\right),
		\end{align*}
		and implies the result
		\begin{align*}
			\int_{\Gamma^{\rm D}} \left(\partial_y \psi^{\mathfrak{h}}\right)^2 
			\leq 
			\mu \int_{\Gamma^{\rm D}} \left(\partial_x \psi^{\mathfrak{h}}\right)^2
			+
			C \int_{\Omega}\left|\nabla^{\mu} \psi^{\mathfrak{h}}\right|^2.
		\end{align*}
	\end{proof}

	With this result at hand, we can make the definitions
	\begin{align*}
		|f|_{ H^{1/2}_{\mu}(\R)} & = |f|_{L^2(\R)} + |f|_{\dot{H}^{1/2}_{\mu}(\R)}\\
		|f|_{ \dot{H}^{1}_{\mu}(\R)} & = \frac{1}{\sqrt{\mu}} \left(|f|_{\dot{H}^{1/2}_{\mu}(\R)} + |\partial_x f |_{L^2(\R)}\right),
	\end{align*}
	and as a direct application of Theorem \ref{thm: Wp ww} we obtain the following result.
	\begin{cor}\label{Cor: characterization H1} 
		Let $\Omega$, and $\Gamma$ be as in Assumption \ref{Def: Assumption 1}. Then for all ${\bf U}^{\rm in}\in H^{1/2}_{\mu}(\R) \times \dot{H}^{1}_{\mu}(\R)$ and for any $T>0$, there exists a unique solution ${\bf U}\in  C\left([0,T]; H^{1/2}_{\mu}(\R) \times \dot{H}^1_{\mu}(\R)\right)$ to \eqref{Eq: Linear ww} with initial condition ${\bf U}(0)={\bf U}^{\rm in}$. Moreover, there exists $C>0$  independent of ${\bf U}^{\rm in}$ such that  
		$$ 
		| \zeta |_{L^{\infty}([0, T] ; H^{1/2}_{\mu}(\R))}
		+
		| \psi |_{L^{\infty}([0, T] ; \dot{H}^{1}_{\mu}(\R))} 
		\leq 
		C 
		\left(|\zeta^{\rm in}|_{ H^{1/2}_{\mu}(\R)}
		+
		| \psi^{\rm in} |_{\dot{H}^{1}_{\mu}(\R)} \right).
		$$   
	\end{cor}
	\begin{remark}\label{Remark: application of cor}
		If we further suppose $(\zeta^{\rm in}, \psi^{\rm in} )\in \mathcal{H}^{1}_{\mu}(\R) \times \dot{\mathcal{H}}^{3/2}_{\mu}(\R)$, then we also have $(\zeta, \psi) \in   C\left([0,T]; H^{1}_{\mu}(\R) \times \dot{H}^1_{\mu}(\R)\right)$, so that both $\zeta$ and $q$ are continuous in time and space.
	\end{remark}
	
	To remove the singularity in $\mu$ appearing in $H^{1/2}_{\mu}(\R) \times \dot{H}^1_{\mu}(\R)$ one typically has to assume more regularity on the data. In fact, the singular dependence in $\dot{H}^1_{\mu}(\R)$ can be avoided in the smooth bottom case by treating the Dirichlet--Neumann operator as a second-order differential operator \cite{WWP}. This is obvious in the flat-bottom case (see Remark \ref{Remark: Dn op frequency}) where one can use a Taylor expansion in low frequencies to get a uniform bound with respect to $\mu$ at a price of regularity. For a smooth domain, this is classically done by straightening diffeomorphisms and elliptic estimates on the harmonic extension of $\psi$ \cite{WWP}. However, for a general Lipschitz domain this is not clear as the potential can have singularities at the bottom, even though the Dirichlet--Neumann operator remains regular. We will next illustrate this difficulty by considering a $\dot{H}^2(\Omega)-$estimate in the case of a discontinuous bottom.  
	
	\section{An elliptic estimate for a discontinuous bathymetry}\label{Sec: elliptic est for step} The step geometry is a special case for which we have a conformal map to the unit strip. We give the definition: 
	\begin{assumption}\label{Def: Assumption 3}
		Let $\gamma \in (0,1)$ and let $\Omega$ denote the fluid domain with boundary $\Gamma$. The domain has an inlet of size one for $x<0$ and a sharp step discontinuity at $x=0$. The lower boundary $\Gamma^{\rm N}$ is defined by the piecewise constant function
		\begin{equation*}
			h(x) =
			\begin{cases}
				1 & \text{for } x < 0, \\ 
				1-\gamma & \text{for } x > 0.
			\end{cases}
		\end{equation*}
		The fluid domain is therefore
		\begin{equation*}
			\Omega = \{(x,y) \in \mathbb{R}^2 : -h(x) < y < 0\}.
		\end{equation*}
		We denote by $\Gamma^{\rm D}$ the flat upper boundary $y=0$, and by $\Gamma^{\rm N}$ the lower boundary. The interface at $x=0$ is a vertical segment generating a corner of angle $3\pi/2$.

	\end{assumption}
	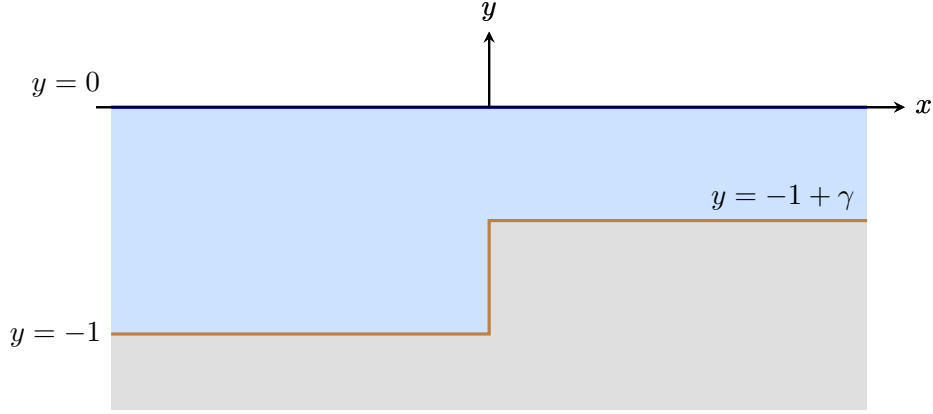
\begin{figure}[h] 
		\centering
		\begin{tikzpicture}  
			 
			\coordinate (L_inf) at (-5, -3);    
			\coordinate (L_center) at (0, -3);  
			\coordinate (R_center) at (0, -1.5); 
			\coordinate (R_inf) at (5, -1.5);   
			\coordinate (Top_L) at (-5, 0);     
			\coordinate (Top_R) at (5, 0);      
			\coordinate (Bot_L) at (-5, -4);    
			\coordinate (Bot_R) at (5, -4);     
			 
			\draw[->, >=stealth, thick] (-5.2, 0) -- (5.5, 0) node[right] {$x$};
			\draw[->, >=stealth, thick] (0, 0) -- (0, 1) node[above] {$y$};
			
			\fill[aqua!20] 
			(-5,0) -- (0,0) -- (0,-3) -- (-5,-3) -- cycle;
			 
			\fill[aqua!20] 
			(0,0) -- (5,0) -- (5,-1.5) -- (0,-1.5) -- cycle;
			 
			\fill[gray!100, nearly transparent]
			(-5,-3) -- (0,-3) -- (0,-4) -- (-5,-4) -- cycle;
			
			\fill[gray!100, nearly transparent]
			(0,-1.5) -- (5,-1.5) -- (5,-4) -- (0,-4) -- cycle;
			 
			\draw[brown, very thick] 
			(L_inf) -- (L_center) -- (R_center) -- (R_inf);
			 
			\draw[blue, very thick] (Top_L) -- (Top_R);
			 
			\draw[->, >=stealth, thick] (-5.2, 0) -- (5.5, 0) node[right] {$x$};
			\draw[->, >=stealth, thick] (0, 0) -- (0, 1) node[above] {$y$};
			 
			\node[anchor=east] at (-5,-3) {$y=-1$};
			\node[anchor=west] at (2.8,-1.2) {$y=-1+\gamma$};
			\node[anchor=south] at (-5.6,0) {$y=0$};
			
		\end{tikzpicture} 
		\caption{The fluid domain $\Omega$ in the case of a discontinuous bottom.}
		\label{Fig. Step}
	\end{figure}
	Before we turn to the main results of this section, we recall the conformal mapping in this case. Since we are working with a twisted Laplace operator $\Delta^{\mu}$, we first find it convenient to transform it back to the usual Laplace operator by introducing the scaling:
	\begin{equation}\label{Eq: Scaling w}
		\psi^{\mathfrak{h}}_{\sigma^{-1}} = \psi^{\mathfrak{h}}\left(x,\frac{y}{\sqrt{\mu}}\right) = \psi^{\mathfrak{h}}\circ \sigma^{-1}(x,y),
	\end{equation}
	which is defined on 
	\begin{equation*}
		\Omega_{\mu} =  \{(x, y) \in \R^2 \: : \: -\sqrt{\mu}\, h(x)< y <0   \},
	\end{equation*}
	where $\Gamma^{\rm D}$ and $\Gamma_{\mu}^{\rm N}$ are the boundaries on the top and bottom. Then, $\psi^{\mathfrak{h}}_{\sigma^{-1}}$ satisfies
	\begin{equation}\label{Eq: harmonic ext with step scaled}
		\left\{\begin{aligned}
			\Delta  \psi^{\mathfrak{h}}_{\sigma^{-1}} &= 0 &&\qquad  \text{in}   \quad \:\: \Omega_{\mu}
			\\ 
			\psi^{\mathfrak{h}}_{\sigma^{-1}} & =  \psi  &&\qquad  \text{on}  \quad \:\: \Gamma^{\rm D} 
			\\ 
			\partial_{\rm n}  \psi^{\mathfrak{h}}_{\sigma^{-1}} & = 0 &&\qquad  \text{on}  \quad \:\: \Gamma^{\rm N}_{\mu}.
		\end{aligned}\right.
	\end{equation} 
	We may now give the conformal mapping in this case, which is classically derived using the Schwarz--Christoffel formula; see, for instance, \cite{Mei89, DriscollTrefethen02}  and \cite{Cathala2013Polygonal, Cathala2013Thesis} with dependence in $\mu$. 
	
	\begin{figure}[h!]
		\centering 
		\begin{minipage}{.45\textwidth}
			\centering 
			\begin{tikzpicture}
				\begin{axis}[
					x=1cm, y=1.0cm,
					axis lines=middle,
					xlabel={$x$}, ylabel={$y$},
					x label style={anchor=north},
					y label style={anchor=east},
					xmin=-2.5, xmax=3.3, ymin=0, ymax=1,
					ticks=none,
					clip=false
					] 
					\fill[aqua!20] (axis cs:-2.5,0) -- (axis cs:3,0) -- (axis cs:3,-1.5) -- (axis cs:0,-1.5) -- (axis cs:0,-3) -- (axis cs:-2.5,-3) -- cycle;
					\fill[gray!100, nearly transparent] (axis cs:-2.5,-3) -- (axis cs:0,-3) -- (axis cs:0,-1.5) -- (axis cs:3,-1.5) -- (axis cs:3,-3.5) -- (axis cs:-2.5,-3.5) -- cycle;
					
					\draw[brown, very thick] (axis cs:-2.5,-3) -- (axis cs:0,-3);
					\draw[brown, very thick] (axis cs:0,-1.5) -- (axis cs:3,-1.5);
					\draw[red, very thick] (axis cs:0,-3) -- (axis cs:0,-1.5);
					\draw[blue, very thick] (axis cs:-2.5,0) -- (axis cs:3,0);
					
					\node[anchor=south east] at (axis cs:-2.5, -0.3) {\small $y = 0$};
					\node[anchor=east] at (axis cs:-2.5, -3) {\small $y =-\sqrt{\mu}$};
					\node[anchor=west] at (axis cs:-1, -1.2) {\small $y = \sqrt{\mu}(-1+\gamma)$};
					\node[anchor=west] at (axis cs:-0.6, -3.2) {\small $x=0$};
				\end{axis}
			\end{tikzpicture}
		\end{minipage}%
		\hspace{0.08cm}
		\hfill 
		\begin{minipage}{.08\textwidth}
			\centering
			\begin{tikzpicture}
				\node[label=above:{\large $F = U + iV$}] at (-0.1,-1.) {}; 
				\draw[<-, >=stealth, thick] (-0.4,-1.2) to [out=30, in=150] (0.4,-1.2);
			\end{tikzpicture}
		\end{minipage}%
		\hfill 
		\begin{minipage}{.45\textwidth}
			\centering
			\begin{tikzpicture}
				\begin{axis}[
					x=1.0cm, y=1.0cm, 
					axis lines=middle,
					xlabel={$\alpha$}, ylabel={$\beta$},
					x label style={anchor=north},
					y label style={anchor=east},
					xmin=-2.5, xmax=3.3, ymin=0, ymax=1,
					ticks=none,
					clip=false
					] 
					\fill[aqua!20] (axis cs:-2.5,0) rectangle (axis cs:3,-2.8);
					\fill[gray!100, nearly transparent] (axis cs:-2.5,-2.8) rectangle (axis cs:3,-3.5);
					
					\draw[brown, very thick] (axis cs:-2.5,-2.8) -- (axis cs:0,-2.8);
					\draw[brown, very thick] (axis cs:1.2,-2.8) -- (axis cs:3,-2.8);
					\draw[blue, very thick] (axis cs:-2.5,0) -- (axis cs:3,0);
					
					\draw[red, very thick] (axis cs:0,-2.8) -- (axis cs:1.2,-2.8);
					\node[anchor=north, red] at (axis cs:0, -2.8) {\small $0$};
					\node[anchor=north, red] at (axis cs:1.2, -2.8) {\small $\delta$};
					
					\node[anchor=east] at (axis cs:-2.5, -2.83) {\small $\beta=-\sqrt{\mu}$};
					\node[anchor=south east] at (axis cs:-2.5, -0.35) {\small $\beta=0$};
				\end{axis}
			\end{tikzpicture}
		\end{minipage}
		\caption{A conformal transformation $F$ as defined in equation \eqref{Eq: Conf. map Step}. The red line illustrates the height of the step.} 
		\label{Fig: Conf Step}
	\end{figure}
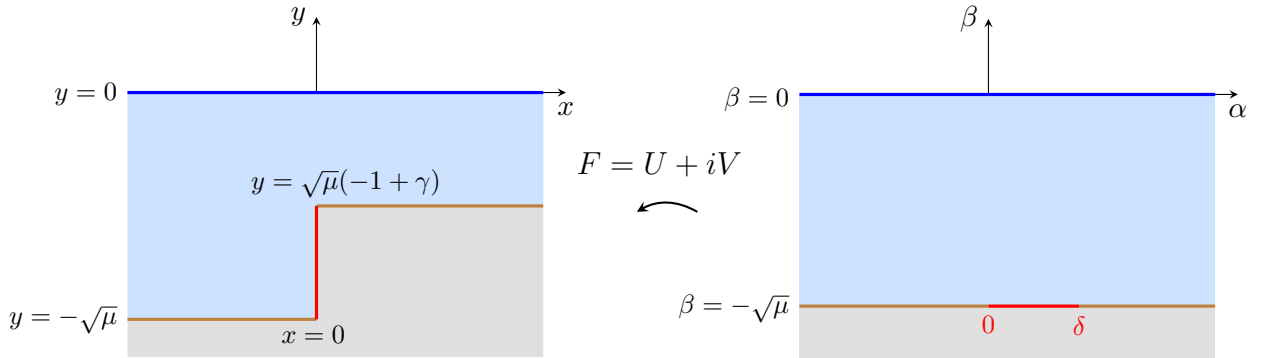

	\begin{lemma}\label{lemma: conf map}
		Let $\gamma, \mu \in (0,1)$ and define the $\mu-$dependent strip
		$$\mathrm{S}_{\mu} = \{(\alpha, \beta) \in \mathbb{C} \: : \: \beta \in (-\sqrt{\mu}, 0)\}.$$
		Then the conformal map to $\Omega_{\mu}$ is given by
		\begin{equation}\label{Eq: Conf. map Step}
			F :
			\begin{cases}
				\: \: \: \:   {\rm S}_{\mu} & \rightarrow \: \: \: \Omega_{\mu}
				\\
				(\alpha, \beta)& \mapsto \: \:\: U(\alpha, \beta) + i V(\alpha, \beta),
			\end{cases}
		\end{equation}
		where $U$ and $V$ are the real and imaginary parts of 
		\begin{align*}
			F(\zeta)  
			= & 
			\frac{2\sqrt{\mu}(1-\gamma)}{\pi}  
			\left(
			\tanh^{-1} \left(  W(f(\zeta)) \right) 
			- 
			(1-\gamma)^{-1}\tanh^{-1} \left((1-\gamma)^{-1}W(f(\zeta)) \right) \right) 
			\\ 
			&  
			-
			i \,  \sqrt{\mu},
		\end{align*}
		for  $\zeta = \alpha + i \beta$ and
		\begin{align*} 
			\delta & = -  \frac{2\sqrt{\mu}}{\pi}\ln{(} 1-\gamma {)}, \quad 
			f(\zeta)& =  -\exp\left( \frac{\pi}{\sqrt{\mu}}\zeta\right), 
			\quad 
			W(f(\zeta)) & = \left( \frac{f(\zeta) - 1}{f(\zeta)  - f(\delta-i \sqrt{\mu})} \right)^{1/2}.
		\end{align*}
		Moreover, we have the following expression for the determinant of the Jacobian associated to the map:
		\begin{equation}\label{Eq: det Jac conf map} 
			|F'(\zeta)|^2
			=
			(1-\gamma)
			\left|
			\frac{	
				\cosh\left(\frac{\pi(\zeta-\delta)}{2\sqrt{\mu}}\right)
			}{
				\cosh\left(\frac{\pi\zeta}{2\sqrt{\mu}}\right)
			}\right|.
		\end{equation}

	\end{lemma}
	\begin{remark}\label{Remark: change of variables}
		We will on several occasions use this map to transform the step geometry into a strip, and we find it convenient to specify the relationship between the derivatives in the two coordinate systems. In particular, we have that
		$U = U(\alpha,\beta)$ and $V = V(\alpha, \beta)$ satisfies
		\begin{equation}\label{eq: C-R}
			\begin{cases}
				U_{\alpha} = V_{\beta}
				\\ 
				U_{\beta} = - V_{\alpha}.
			\end{cases}
		\end{equation}
		Then, by the chain rule, we have the relations
		\begin{equation*}
			\hspace{1cm}
			\begin{cases}
				\partial_{\alpha} 
				=
				U_{\alpha} \partial_x + V_{\alpha} \partial_y 
				\\ 
				\partial_{\beta} 
				=
				U_{\beta} \partial_x + V_{\beta} \partial_y 
			\end{cases}
			\quad 
			\text{and implies}
			\quad
			\begin{cases}
				\partial_{x} 
				=
				\frac{1}{J}
				\big{(}U_{\alpha} \partial_\alpha + U_{\beta} \partial_\beta \big{)}
				\\ 
				\partial_{y} 
				=
				\frac{1}{J}
				\big{(}
				V_{\alpha} \partial_{\alpha} + V_{\beta} \partial_{\beta} 
				\big{)},
			\end{cases}
		\end{equation*}
		where $J$ denotes the determinant of the Jacobian  and, from the Cauchy-Riemann equations \eqref{eq: C-R}, is equal to
		\begin{equation*}
			J = U_{\alpha}V_{\beta} - U_{\beta}V_{\alpha} = U_{\alpha}^2 + V_{\alpha}^2 = 	|F'(\alpha + i \beta)|^2.
		\end{equation*}
	\end{remark}

	To ease the presentation, we give the details of the derivation in the present setting in Appendix \ref{App: prop conf map}, where we also give several identities that will be used in the proofs that follow.

	\subsection{A version of Grisvard's shift Theorem} The conformal map is singular at the bottom in the neighborhood of the re-entrant corner when translated back into the physical variables. This agrees with the classical theory of Grisvard \cite{Grisvard85, Grisvard92} (see also \cite{Kondratiev67, Grisvard92, Dauge88, CostabelDauge93, DaugeNBL90}, and the work of Ming and Wang  \cite{MingWang17} in the context of the Dirichlet--Neumann operator). We are interested in the dependence of the singular function with respect to the shallow water limit. In particular, we will prove that the potential can be decomposed into a regular and singular part, where the latter is defined by:
	\begin{Def}\label{Def: Sing function}
		Let $\gamma, \mu \in (0, \tfrac{1}{2})$,  $\Omega$ and $\Gamma$ satisfy Assumption	\ref{Def: Assumption 3}, and denote by
		\begin{equation*}
			{\rm p}_{\rm c}
			=
			\bigl(0,-(1-\gamma)\bigr)^{\rm T},
		\end{equation*}
		the re-entrant corner. Define the anisotropic polar coordinates
		$(\rho_\mu,\theta_\mu)$ centered at zero by
		\begin{equation*}
			\rho_\mu(x,y)
			=
			\left(
			x^2 + \mu y^2
			\right)^{1/2}
			\quad 
			\text{and} 
			\quad 
			\theta_\mu(x,y)
			=
			\arg\left(
			x+i \sqrt{\mu} y
			\right) 
			\quad 
			\text{for} 
			\quad
			0<\theta_\mu(x,y)<\frac{3\pi}{2}.
		\end{equation*} 
		Also, let $\chi\in C_c^\infty(\R)$ be identically one in a neighbourhood of $0$ and choose some $c_\ast>0$ sufficiently small, independently of
		$\gamma$ and $\mu$, and define
		\begin{equation*}
			\chi_\mu(x)
			=
			\chi\left(
			\frac{x}
			{c_\ast \sqrt{\mu}}
			\right).
		\end{equation*} 
		Then the singular function in the physical variables is defined by
		\begin{equation*}
			{\mathfrak s}_{\mu}(x,y - 1 + \gamma)
			=
			\kappa_{\gamma,\mu}\, \chi_\mu(x) \, 
			\rho_\mu(x,y)^{2/3}
			\cos\left(
			\frac{2}{3}\theta_\mu(x,y)
			\right),
		\end{equation*} 
		with
		\begin{equation}\label{Eq: sing coeff}
			\kappa_{\gamma,\mu}
			=
			\left(
			\frac{
				9\sqrt{\mu}\,\gamma(2-\gamma) }{
				4\pi(1-\gamma)^2
			}
			\right)^{1/3}.
		\end{equation}
		
	\end{Def} 
	The decomposition of the potential into regular and singular parts belongs to the classical regularity theory for elliptic problems in corner domains. The singular exponents are typically obtained by analyzing the corresponding sector problem, using the Mellin transform or spectral analysis, where the singular coefficient is generally characterized using a dual singular function. Here, we instead exploit the explicit Schwarz--Christoffel mapping from the step domain to a flat strip, where the transformed potential admits an explicit Fourier representation. Composing its Taylor expansion with the fractional-power expansion of the inverse mapping yields the singular term and an $H^2-$regular remainder which is treated using a Moser estimate. This explicit representation also allows us to track the dependence on $\mu$ and $\gamma$ and to derive a direct formula for the singular coefficient in terms of the Dirichlet data. In particular, the potential satisfies the following expansion:

	\begin{prop}
		Let $\gamma, \mu \in (0, \tfrac{1}{2})$, $\Omega$ and $\Gamma$ be as in Assumption \ref{Def: Assumption 3}. Then for any $\psi \in \dot{H}^{3/2}({\R})$ there is a regular function $\psi^{\mathfrak{h}}_{\rm reg} \in \dot{H}^2(\Omega)$ such that the variational solution $\psi^{\mathfrak{h}} \in \dot{H}^1(\Omega)$ to \eqref{Eq: phi h} admits the following decomposition
		\begin{equation*}
			\psi^{\mathfrak{h}} 
			=
			\psi^{\mathfrak{h}}_{\rm reg}
			+
			{\mathfrak c}(\psi)\, {\mathfrak s}_{\mu},
		\end{equation*}
		where the linear form ${\rm c}(\psi) \in \big(\dot{H}^{1}(\R)\big)'$ is defined by 
		\begin{equation}\label{Eq: linear form}
			{\mathfrak c}(\psi)  = \frac{1}{2\pi}
			\int_{\mathbb R}
			e^{i\delta\xi}
			\frac{i\xi}{\cosh\left(\sqrt{\mu}|\xi|\right)}
			\widehat{\psi_0}(\xi)
			\,{\rm d}\xi,
		\end{equation}
		with
		\begin{equation*}
			\delta
			=
			-\frac{2\sqrt{\mu}}{\pi}\log(1-\gamma) \qquad \text{and} \qquad \psi_0(\alpha) = \psi\circ U(\alpha,0),
		\end{equation*}
		for $\alpha \in \R$. Moreover, the solution satisfies
		\begin{align*}
			\left\| (\nabla^{\mu})^2\left( \psi^{\mathfrak{h}}  - {\mathfrak c}(\psi)\,
			{\mathfrak s}_{\mu}\right)
			\right\|_{L^2(\Omega)} 
			\leq 
			(\sqrt{\mu}
			+
			\gamma \mu^{1/4}) C|\partial_x \psi|_{H^{1/2}(\R)},
		\end{align*}
		for some universal constant $C>0$.  
	\end{prop}

	\begin{remark}
		The singular decomposition reveals a characteristic horizontal scale $x/\sqrt{\mu}$ through the cut-off function. The same scale also appears in the dynamics, where the discrepancy between the water-wave and shallow-water solutions becomes localized near the step.
	\end{remark}
	
	\begin{proof}
		We first derive the formula for $\psi^{\mathfrak{h}}$ by going in conformal coordinates. Then we study the asymptotics and relate it back to the original physical variables. \\

		\noindent
		{\bf Step 1.} \textit{The singularity decomposition in conformal coordinates.} We first rescale the potential by \eqref{Eq: Scaling w} and then write the harmonic extension in the conformal variables by
		$$
		F:S_{\mu}\rightarrow \Omega_{\mu},
		\qquad
		F(\alpha+i\beta)= U(\alpha,\beta)+i \,V(\alpha,\beta).
		$$
		Then we rescale it back using the scaling operation $\sigma(\alpha,\beta) = \alpha + i\, \sqrt{\mu}\, \beta$ and make the definition
		\begin{align*}
			\Psi^{\mathfrak h}(\alpha, \beta) 
			& = \psi^{\mathfrak h}\circ \sigma^{-1} \circ F \circ \sigma (\alpha, \beta)
			\\
			&\eqcolon \psi^{\mathfrak h}\circ F_{\mu} (\alpha, \beta).
		\end{align*}
		Also note that $F = \sigma^{-1} \circ F \circ \sigma$ on $\beta = 0$, so that the Dirichlet data $\psi$ in conformal coordinates is defined by
		$$
		\psi_0(\alpha) 
		=
		\psi\left(U(\alpha,0)\right).
		$$
		In particular, $\Psi^{\mathfrak h}$ satisfies the following elliptic problem
		\begin{equation*}
			\left\{\begin{aligned}
				\Delta_{\alpha,\beta}^{\mu}\Psi^{\mathfrak h}&=0&&\qquad  -1 <\beta<0
				\\ 
				\Psi^{\mathfrak h}(\alpha,0) & =\psi_0(\alpha) &&\qquad  \alpha\in\mathbb{R}
				\\ 
				\partial_{\beta}\Psi^{\mathfrak h}(\alpha,-1) & = 0 &&\qquad  \alpha\in\mathbb{R}.
			\end{aligned}\right.
		\end{equation*} 
		Using the Fourier transform in the $\alpha$ variable, the solution classically reads
		$$
		\Psi^{\mathfrak h}(\alpha,\beta)
		=
		\frac{1}{2\pi}
		\int_{\mathbb{R}}
		e^{i\alpha\xi}\, 
		\frac{
			\cosh\bigl((\beta + 1)\sqrt{\mu}|\xi|\bigr)
		}{
			\cosh\bigl(\sqrt{\mu}|\xi|\bigr)
		}\, 
		\widehat{\psi_0}(\xi)
		\, {\rm d}\xi,
		$$
		and by Plancherel's identity one finds that
		\begin{align}\label{Eq: est on Psi}
			\left\|\left(\nabla^{\mu}\right)^2\Psi^{\mathfrak h}\right\|_{L^2\left( {\rm S}\right)} \lesssim   \mu |\partial_{\alpha} \psi_0|_{\dot{H}^{1/2}_{\mu}(\R)}.
		\end{align}
		The singularity is only a feature of the degeneracy of the conformal map near the re-entrant corner. To identify it, we make a Taylor expansion to find that
		\begin{align}\notag
			\Psi^{\mathfrak h}(\alpha, \beta)
			& =
			\Psi^{\mathfrak h}(\delta,-1)
			+
			\nabla_{\alpha,\beta} \Psi^{\mathfrak h}(\delta,-1)
			\cdot 
			\begin{pmatrix}
				\alpha- \delta \\
				\beta + 1
			\end{pmatrix}
			+
			R_0(\alpha, \beta)
			\\ 
			& = 
			\Psi^{\mathfrak h}(\delta, -1 )
			+
			\partial_\alpha\Psi^{\mathfrak h}(\delta, - 1)
			(\alpha - \delta )
			+
			R_0(\alpha, \beta),\label{Eq: expansion conformal coordinates}
		\end{align}
		where we used the homogeneous Neumann condition. The singular term is identified with $(\alpha - \delta )$ next by expanding the conformal map in a small ball depending on $\gamma$ and $\mu$, while the rest $R_0$ is estimated in physical variables on the same ball.
		\\

		\noindent
		{\bf Step 2.} \textit{Expansion of singular part in physical variables.} First, let ${\rm p} = x+iy \in \Omega$ and ${\rm p}_{\rm c} = i(-1 +\gamma )$ denote the corner point in the physical domain. By definition we have that 
		$$H({\rm p}) \coloneq F^{-1}_{\mu}({\rm p}) - F^{-1}_{\mu}({\rm p}_{\rm c})$$
		maps to $\zeta - \zeta_{\rm c} \coloneq \alpha-\delta+ i(\beta +1)$. To justify the expansion of the inverse, we define the set
		$$ 
		\sigma B_{c_0 \gamma \sqrt{\mu}}({\rm p}_{\rm c}) \coloneq \{p \in \mathbb{C} \, : \, |\sigma( {\rm p}-{\rm p}_{\rm c})| < c_0 \gamma \sqrt{\mu}\},$$
		for some $c_0 \in (0,1)$, and let ${\rm p} \in \Omega \cap \sigma B_{c_0 \gamma \sqrt{\mu}}({\rm p}_{\rm c}) $. Then by Lemma \ref{Lemma: est on inverse} we have the following expansion 
		\begin{equation}\label{Eq: Inverse rel}
		 	\sigma \circ H({\rm p})
			= 
			\kappa 
			\left( \sigma({\rm p}-{\rm p}_{\rm c}) \right)^{2/3}
			+
			\kappa^2(\sigma({\rm p}-{\rm p}_{\rm c}))^{4/3}\,
			R,
		\end{equation}
		where the coefficient is defined by \eqref{Eq: sing coeff} and the rest satisfies
		\begin{equation*}
			\sup_{|\sigma( {\rm p}-{\rm p}_{\rm c})| < c_0 \gamma \sqrt{\mu}   }\left( 
			\left|R({\rm p})\right|
			+
			\left| R'({\rm p})\right|
			+
			\left| R''({\rm p})\right|
			\right)
			\lesssim \gamma^{-1}.
		\end{equation*}
		Since $(\alpha - \delta)$ corresponds to the real part of $H$, we find it natural to define the singular function by
		\begin{equation*}
			{\mathfrak s}
			({\rm p}) 
			=\kappa \,  
			\mathfrak{Re}\,
			\left( \sigma({\rm p}-{\rm p}_{\rm c}) \right)^{2/3},
		\end{equation*}
		which is equivalent to the one given in Definition \ref{Def: Sing function}. Moreover, we have that  
		$$\mathfrak{Re}H = {\mathfrak s} + \mathcal{R}_1,$$
		where the rest is controlled taking the real part of \eqref{Eq: Inverse rel}:
		\begin{align*}
			\left| \left(\nabla^{\mu} \right)^j  \mathcal{R}_1 \right| \leq \gamma^{-1/3} \mu^{1/3+j/2} \left|\sigma({\rm p}-{\rm p}_{\rm c}) \right|^{4/3-j},
		\end{align*}
		for $j = 0, 1, 2$. In particular, we have $\mathcal{R}_{1} \in \dot{H}^2(\Omega)$ and the estimate
		\begin{align*}
			\left\| \left(\nabla^{\mu}\right)^j \mathcal{R}_1 \right \|_{L^2\left( \Omega \cap \sigma B_{c_0\gamma \sqrt{\mu}}({\rm p}_{\rm c}) \right)}^2
			& \lesssim 
			\gamma^{-2/3}\mu^{2/3 + j}
			\int_{\{\left|\sigma(x,y)\right| \lesssim \gamma \sqrt{\mu}\}} 
			|\sigma(x,y)|^{8/3-2j} \, {\rm d} x {\rm d} y
			\\
			& \lesssim 
			\gamma^{-2/3}\mu^{2/3 - 1/2+ j}
			\int_{\{0<r \lesssim \gamma \sqrt{\mu}\}} 
			r^{8/3-2j}r \,  {\rm d}r
			\\
			& \lesssim
			\mu^{5/2},
		\end{align*}
		where we used the support and applied a change of variable.  \\
		
		\noindent
		{\bf Step 3.} \textit{Localized $H^2-$estimate on the rest.} For the estimate in the physical variables on $R_0$, defined in \eqref{Eq: expansion conformal coordinates}, we first note that having $p \in \Omega \cap \sigma B_{c_0\gamma \sqrt{\mu}}({\rm p}_{\rm c})$ implies $\zeta \in {\rm S}\cap \sigma B_{c_0\gamma \sqrt{\mu}}({\zeta}_{\rm c})$ using \eqref{Eq: Inverse rel}. Also, by abuse of notation we write $R_0 = R_0(\zeta)$ and use the Fundamental Theorem of Calculus to find that
		\begin{align*}
			|\nabla^{\mu} R_0(\zeta )| & \lesssim  \frac{1}{\sqrt{\mu}} |\sigma(\zeta - \zeta_{\rm c})| \sup \limits_{t \in (0,1)} \left|  \left(\nabla^{\mu}\right)^2 \Psi^{\mathfrak h}(\zeta_{\rm c} + t (\zeta - \zeta_{\rm c}))\right|.
		\end{align*}
		Now since the components of $\left(\nabla^{\mu}\right)^2\Psi^{\mathfrak h}(\sigma^{-1}\zeta)$ are harmonic, we may use the localized $L^{2}-L^{\infty}$ Moser estimate, see for instance Theorem 8.17 in \cite{GilbargTrudinger01}, which gives 
		\begin{align}\notag
			\sup \limits_{{\rm S}\cap B_{\sqrt{\mu}}({\zeta}_{\rm c})}\left|\left(\nabla^{\mu}\right)^2\Psi^{\mathfrak h}(\sigma^{-1}\zeta)\right| 
			& \leq \frac{C}{\sqrt{\mu}}   \left\| \left(\nabla^{\mu}\right)^2\Psi^{\mathfrak h}\circ \sigma^{-1}\right\|_{L^2\left( {\rm S}\cap B_{2\sqrt{\mu}}({\zeta}_{\rm c})\right)} 
			\\ 
			& \leq
			\mu^{-1/4}C \left\| \left(\nabla^{\mu}\right)^2\Psi^{\mathfrak h}\right\|_{L^2\left( {\rm S}\right)}.\label{Eq: Moser est}
		\end{align}
		Next, we employ \eqref{Eq: est on Psi} and find that
		\begin{align*}
			|\nabla^{\mu} R_0(\zeta )|  \lesssim  \frac{1}{\sqrt{\mu}} |\sigma(\zeta - \zeta_{\rm c})|
			c_1 \qquad \text{and} \qquad \left|\left(\nabla^{\mu}\right)^2 R_0(\zeta )\right|  \lesssim c_1,
		\end{align*}
		with $c_1  = \left(\mu^{3/4}
		|\partial_\alpha \psi_0|_{\dot{H}^{1/2}_{\mu}(\R)}
		\right)$. In other words, if we let $\mathcal{R}_0 = R_0 \circ F_{\mu}^{-1}$ then we can make a change of variable, using Lemma \ref{Lemma: est on change of variables} where there is a singularity at $\zeta_0 = - i$ that satisfies $|\sigma(\zeta - \zeta_0)|^{-1} \leq \gamma^{-1}\mu^{-1/2}$ on $\zeta \in {\rm S}\cap \sigma B_{c_0\gamma \sqrt{\mu}}({\zeta}_{\rm c})$. In particular, we obtain the following bound on the rest
		\begin{align*}
			\left\| \left(\nabla^{\mu}\right)^2 \mathcal{R}_0\right \|_{L^2\left( \Omega \cap \sigma B_{c_0\gamma \sqrt{\mu}}({\rm p}_{\rm c}) \right)}^2
			\lesssim & 
			c_1^2 \sqrt{\mu}
			\int_{\{ |\sigma (\alpha, \beta)| \lesssim \gamma \sqrt{\mu} \}} |\sigma(\alpha, \beta)|^{-1} \, {\rm d}\alpha {\rm d}\beta
			\\
			\lesssim & \gamma \mu^2 |\partial_\alpha \psi_0|_{\dot{H}^{1/2}_{\mu}(\R)}^2.
		\end{align*}

		\noindent
		{\bf Step 4.} \textit{Elliptic estimate.} To conclude, we will use the previous steps and split the global estimate into two parts. The first part is on ${\rm p} \in \Omega \cap \sigma B_{c_0\gamma \sqrt{\mu}}$ with $c_0>0$ such that
		\begin{equation*}
			\psi^{\mathfrak{h}}
			=
			\Psi^{\mathfrak h}(\delta, -1 )
			+
			\mathfrak{c}(\psi)
			\mathfrak{s}
			+
			\mathcal{R}_0 
			+
			\mathfrak{c}(\psi)
			\mathcal{R}_1,
		\end{equation*}
		where $	\mathfrak{c}(\psi)$ is defined by \eqref{Eq: linear form}. Then if we let $\chi\in C_c^\infty(\R)$, equal to one in a neighbourhood of $0$, and for $c_{\ast} >0$ such that 
		\begin{equation*}
			\chi_\mu(x)
			=
			\chi\left(
			\tfrac{x}{c_{\ast}\sqrt{\mu}}
			\right),
		\end{equation*}
		it is identically one on $\Omega \cap \sigma B_{c_0\gamma \sqrt{\mu}}$. Then the singular term $\mathfrak{s}_{\mu} = \chi_\mu \mathfrak{s}$ is in $H^{\infty}$ away from the corner, and we have that
		\begin{align*}
			\left\| (\nabla^{\mu})^2\left( \psi^{\mathfrak{h}}  -  {\mathfrak c}(\psi)\,
			\chi_{\mu}{\mathfrak s}\right)
			\right\|_{L^2(\Omega)}
			= &
			\left\| (\nabla^{\mu})^2\mathcal{R}_0 
			\right\|_{L^2(\Omega\cap \sigma B_{c_{\ast} \gamma \sqrt{\mu}}({\rm p}_{\rm c}))} 
			+
			\left\| (\nabla^{\mu})^2\mathcal{R}_1
			\right\|_{L^2(\Omega\cap \sigma B_{c_{\ast} \gamma \sqrt{\mu}}({\rm p}_{\rm c}))}
			\\
			& 
			+
			\left\| (\nabla^{\mu})^2\left( 
			{\mathfrak c}(\psi)\,
			\chi_{\mu} {\mathfrak s}\right)
			\right\|_{L^2(\Omega\cap (\sigma B_{c_{\ast} \gamma \sqrt{\mu}}({\rm p}_{\rm c})^{\rm c})}
			\\ &
			+
			\left\| (\nabla^{\mu})^2 \psi^{\mathfrak{h}}  
			\right\|_{L^2(\Omega\cap (\sigma B_{c_{\ast} \gamma \sqrt{\mu}}({\rm p}_{\rm c})^{\rm c})}
			\\
			\eqcolon& \,  
			{\rm I}_1 + {\rm I}_2 + {\rm I}_3 + {\rm I}_4.
		\end{align*}
		The estimate on ${\rm I}_1$ is given in Step $3$, while for ${\rm I}_2$ we use the information provided in Step 2, together with Cauchy-Schwarz's inequality to get
		\begin{align*}
			{\rm I}_2 & \lesssim \mu^{5/4} \left( \int_{\R} \frac{1}{\cosh^2(\sqrt{\mu}|\xi|)} {\rm d} \xi \right)^{1/2} |\partial_{\alpha} \psi_0|_{L^2(\R)}
			\\
			& \lesssim \mu  |\partial_{\alpha} \psi_0|_{L^2(\R)}.
		\end{align*}
		To control ${\rm I}_3$, we first argue as in Step 2, where we find that
		\begin{align*}
			\left| \left(\nabla^{\mu} \right)^j  {\mathfrak s} \right| \leq \gamma^{1/3} \mu^{1/6+j/2} \left|\sigma({\rm p}-{\rm p}_{\rm c}) \right|^{2/3-j},
		\end{align*}
		which gives
		\begin{align*}
			\int_{\{\gamma \sqrt{\mu} \lesssim |\sigma({\rm p} - {\rm p}_{\rm c})| \lesssim \sqrt{\mu}\}}\left| \left(\nabla^{\mu} \right)^j  {\mathfrak s} \right|^2 \lesssim \mu^{3/2}.
		\end{align*}
		Then since the cut-off function is only a restriction in the horizontal variable, we find that
		\begin{align*}
			{\rm I}_3 
			& \lesssim \mu^{3/4} |{\mathfrak c}(\psi)|
			\\ 
			& \lesssim \sqrt{\mu} |\partial_{\alpha} \psi_0|_{L^2(\R)}.
		\end{align*}
		Finally, for ${\rm I}_4$, we may again perform a change of variable where we use that $\gamma \sqrt{\mu} \lesssim |\sigma({\rm p} - {\rm p}_{\rm c})|$ implies $\gamma \sqrt{\mu} \lesssim |\sigma({\zeta} - {\zeta}_{\rm c})|$ together with Lemma \ref{Lemma: est on change of variables} to find that
		\begin{align*}
			{\rm I}_4^2
			 \lesssim 
			 \|\nabla^{\mu} \Psi\|_{H^1({\rm S})}^2
			+
			\gamma\sqrt{\mu}
			\int_{{\rm S}\cap\{|\sigma({\zeta} - {\zeta}_{0})|\lesssim \sqrt{\mu}\}} \frac{ |\nabla^{\mu} \Psi(\zeta)|^2}{|\sigma(\zeta - \zeta_0)|} 
			+   
			\sqrt{\mu} 
			 \int_{{\rm S}\cap\{\gamma \sqrt{\mu} \lesssim |\sigma({\zeta} - {\zeta}_{\rm c})|\lesssim \sqrt{\mu}\}} 
			\frac{ |(\nabla^{\mu})^2 \Psi(\zeta)|^2}{|\sigma (\zeta - \zeta_{\rm c})|} 
			.
		\end{align*}
		Then using again the Moser estimate \eqref{Eq: Moser est} and the bound on the harmonic extension \eqref{Eq: est on Psi} gives
		\begin{align*} 
			{\rm I}_4 \lesssim 
			\sqrt{\mu}
			|\psi_0|_{\dot{H}^{1/2}_{\mu}(\R)} + 
			\mu |\partial_{\alpha} \psi_0|_{\dot{H}^{1/2}_{\mu}(\R)}.
		\end{align*}

		Gathering all these estimates, we find that
		\begin{align*}
			\left\| (\nabla^{\mu})^2\left( \psi^{\mathfrak{h}}  -  {\mathfrak c}(\psi)\,
			\chi_{\mu}{\mathfrak s}\right)
			\right\|_{L^2(\Omega)} \lesssim 	\sqrt{\mu}
			|\psi_0|_{\dot{H}^{1/2}_{\mu}(\R)} +  \sqrt{\mu} |\partial_{\alpha} \psi_0|_{L^2(\R)} +
			\mu |\partial_{\alpha} \psi_0|_{\dot{H}^{1/2}_{\mu}(\R)}.
		\end{align*}  
		To conclude the proof, we use Lemma \ref{lemma: alpha to G} to deal with the first two terms: 
		\begin{align*}
			|\psi_0|_{\dot{H}^{1/2}_{\mu}(\R)} +  |\partial_{\alpha} \psi_0|_{L^2(\R)} \lesssim |\psi|_{\dot{H}^{1/2}_{\mu}(\R)} +  |\partial_{x} \psi|_{L^2(\R)}.
		\end{align*}
		On the other hand, for $\partial_{\alpha}\psi_0 \in \dot{H}^{1/2}_{\mu}(\R)$ we have the relation  $\partial_{\alpha}\psi_0 = U_0'\, (\partial_x \psi)\circ U_0$ with $U_0' \in L^{\infty}(\R)$ and with smooth derivatives that is singular with respect to $\mu$ as described in Lemma \ref{Lemma: est on U_0}. In particular, using this result together with Lemma 4.14 in \cite{WWP} we find that
		\begin{align*}
			|\partial_{\alpha}\psi_{0}|_{H^{1/2}(\R)}
			& \leq \left|U_0''\right|_{H^1(\R)} |\partial_x \psi|_{H^{1/2}(\R)}
			\\ 
			& \leq  (1+\gamma \mu^{-3/4})C|\partial_x \psi|_{H^{1/2}(\R)}.
		\end{align*}
		The most singular term with respect to the shallow water parameter was hidden in this computation. Adding it to the previous estimates completes the proof.
	\end{proof}
	
	As seen in the proof, the precision in $\mu$  is closely related to the regularity of the solution, where higher-order derivatives induce singularities from differentiating the conformal map. To address this underlying issue, we will work directly with the trace of the normal derivative, for which we need adapted weighted estimates specific to the shallow water expansion.

	\section{Derivation of the shallow water equations for a discontinuous bathymetry}\label{Sec: derivation of sw}
	
	In this section, we will prove the main result of the paper. Namely, we know from Theorem \ref{thm: Wp ww} that the linear water waves equations, given by \eqref{Eq: Linear ww}, admit regular solutions, and we will show that the solutions are directly linked to the shallow water equations with transmission conditions through the horizontal discharge. Then we study energy estimates for transmission problems to prove the convergence.

	\subsection{On the horizontal discharge and transmission conditions}\label{Sec: discharge} We first recall the precise definition of the horizontal discharge and its relation to the Dirichlet-Neumann operator. At the variational level, the corners at the bottom do not matter, and we can define it in the general configuration.
	\begin{lemma}\label{Lemma: q}
		Let $\Omega$ and $\Gamma$ be as in Assumption \ref{Def: Assumption 1}. Let $\psi \in \dot{H}^{1/2}_{\mu}(\Gamma^{\rm D})$ and  $\psi^{\mathfrak{h}} \in \dot{H}^1(\Omega)$ be the variational solution of  \eqref{Eq: phi h}. 
		Then the horizontal discharge  is defined by 
		\begin{equation}\label{Def: Discharge}
			\mathcal{Q}_{\mu} :
			\begin{cases}
				\dot{H}^{1/2}(\R) & \rightarrow L^2(\R)  
				\\ 
				\psi & \mapsto 
				\int_{-h}^0 \partial_x \psi^{\mathfrak{h}} \: \mathrm{d} y,
			\end{cases} 
		\end{equation}
		and satisfies
		\begin{align*}
			|\mathcal{Q}_{\mu} \psi|_{L^2(\R)}   \lesssim |\psi|_{\dot{H}^{1/2}_{\mu}(\Gamma^{\rm D})}. 
		\end{align*}
		Moreover, in the distributional sense, there holds		 
		\begin{equation*}
			{\mathcal G}_{\mu}\psi  = - \partial_{x}\mathcal{Q}_{\mu} \psi.
		\end{equation*}
	\end{lemma}
	
	\begin{remark}
		To ease notation, we will in many instances write $q = \mathcal{Q}_{\mu} \psi$.
	\end{remark}

	\begin{proof}
		
		Let $\psi \in \mathcal{D}(\Gamma^{\rm D})$, and use  Hölder's inequality and Lemma \ref{Lemma: var est} to find that
		\begin{align*}
			|q|_{L^2(\Gamma^{\rm D})} \lesssim |h|_{L^{\infty}(\Gamma^{\rm D})} \|\nabla \psi^{\mathfrak{h}} \|_{L^2(\Omega)} \lesssim |\psi|_{\dot{H}^{1/2}(\Gamma^{\rm D})}. 
		\end{align*}
		
		For the second part, the proof is the same as Proposition 3.35 in \cite{WWP} and is recalled here. In particular, we let $\varphi \in \mathcal{D}(\Gamma^{\rm D})$. Then use that $\psi^{\mathfrak{h}}$ is harmonic and Green's formula to obtain the identity
		\begin{align*}
			\int_{\Gamma^{\rm D}} {\mathcal G}_{\mu}\psi \: \varphi 
			=
			\frac{1}{\mu}
			\int_{\Omega} \nabla^{\mu}\psi^{\mathfrak{h}}  \cdot \nabla^{\mu} \varphi  
			= 
			\int_{\R} \int_{-h}^0 \partial_x \psi^{\mathfrak{h}}  \: \partial_x\varphi \: \mathrm{d}y\,  \mathrm{d}x,
		\end{align*}
		where $\varphi = \varphi(x)$ and we find that
		\begin{align}\label{Eq: weak fromulation}
			\int_{\Gamma^{\rm D}} {\mathcal G}_{\mu} \psi \: \varphi  
			=
			\int_{\Gamma^{\rm D}}  \left(\partial_x \varphi\right) \Big{(}\int_{-h}^0   \partial_x \psi^{\mathfrak{h}}  \:  \mathrm{d}y\Big{)} \: {\rm d}\Gamma^{\rm D}. 
		\end{align}
	\end{proof}
	This is the natural quantity to define the shallow water equations since the definition contains a transmission condition at zero in the case of step bathymetry. Formally, if the trace of $q$ is well-defined at zero, we see this from integrating the weak formulation \eqref{Eq: weak fromulation} that
	\begin{equation*} 
		{\mathcal G}_{\mu}\psi = -\partial_x q - \llbracket q \rrbracket\delta_0,
	\end{equation*}
	where $\delta_0$ is a Dirac mass concentrated at zero. Clearly, the jump term is zero if ${\mathcal G}_{\mu}\psi \in L^2(\R)$ and is justified for $\psi \in \dot{H}^1(\R)$. Indeed, under this assumption we have that $q$ is continuous even for a Lipschitz bathymetry. 
	\begin{lemma}
		Let $\Omega$ and $\Gamma$ be as in Assumption \ref{Def: Assumption 1}. Let $\psi \in \dot{H}^1(\R)$ then $q \in H^1(\R)$ and satisfies
		\begin{align*}
			\sqrt{\mu}|\partial_x q|_{L^2(\R)} \leq 
			C \left(  |\psi|_{\dot{H}^{1/2}_{\mu}} +  |\partial_x \psi|_{L^2(\R)}  \right).
		\end{align*} 
	\end{lemma}
	
	\begin{proof}
		
		By Proposition \ref{Prop: H1 characterization} we have that $\mathcal{G}_{\mu}\psi \in L^2(\R)$ and we conclude by Lemma \ref{Lemma: q} that $q \in H^1(\R)$.  The estimate follows by \eqref{Eq: Rellich dn to dx}:
		\begin{align*}
			\sqrt{\mu}|\partial_x q|_{L^2(\R)} \leq 
			C \left(  |\psi|_{\dot{H}^{1/2}_{\mu}} +  |\partial_x \psi|_{L^2(\R)}  \right).
		\end{align*} 
	\end{proof}
	Now, from Corollary \ref{Cor: characterization H1} and Remark \ref{Remark: application of cor} we have that $(\zeta, q)$ are continuous in time and space for regular data, which is why they serve as the natural variables for the transmission problem. In other words, we can use the solution of water waves to define $\mathcal{Q}_{\mu}\psi \eqcolon h \overline{v}$, which in turn satisfies the following equation:
	\begin{equation}\label{Eq: Shallow water with error}
		\left\{\begin{aligned}
			\partial_t \zeta + \partial_x (h \overline{v}) & = 0 
			\\ 
			\partial_t \overline{v} +  \partial_x \zeta & = r,
		\end{aligned}\right.
	\end{equation}
	where the source term reads
	$$r = \partial_t\left( \overline{v}  -  \partial_x \psi\right).$$ 
	To compare this system with the shallow water equations and deduce convergence in \linebreak $L^{\infty}_TH_x^1(\R\backslash \{0\})$ one necessarily needs to impose the transmission conditions. Since we are able to recover these from the solutions of the water waves equations, we may now prove the energy estimate before we turn to the proof of the main result.

	\subsection{Shallow water equations as a transmission problem} As alluded to in the previous section, to study the zero-dispersion limit for the water waves equations in the uniform norm, one needs to use the good variables that allow for transmission conditions. In turn, we may pose the shallow water equations with a source on two half-lines and use transmission conditions as a link:
	\begin{equation*}
		\llbracket \zeta \rrbracket = 0 \qquad \text{and} \qquad 	\llbracket q \rrbracket = 0.
	\end{equation*}
	In particular, one can treat the transmission condition at $x = 0$ as a boundary value and reduce the Cauchy problem to an initial boundary value problem. The reduced problem was studied rigorously in \cite{IguchiLannes21}, and we recall it here.  We will formulate the problem as a $4\times 4$ system where we first define the notation:
	$$\mathbf{u} 
	\coloneq
	\begin{pmatrix}
		\zeta\\
		q
	\end{pmatrix},
	\quad 
	\quad 
	\mathbf{B} 
	\coloneq
	\begin{pmatrix}
		0 & 1
		\\ 
		h & 0
	\end{pmatrix},
	\quad  
	\quad 
	\mathbf{F} 
	\coloneq
	\begin{pmatrix}
		0 \\
		R
	\end{pmatrix},
	$$
	and define
	\begin{equation*}
		\left\{\begin{aligned}
			\mathbf{u}^{-} & \coloneq \mathbf{u}, \quad  \quad \mathbf{B}^{-}  \coloneq \mathbf{B}, &&  \quad \text{and} \quad  \mathbf{F}^{-}  \coloneq \mathbf{F} &&& \qquad  \text{for} &&&& x<0
			\\
			\mathbf{u}^{+} &  \coloneq \mathbf{u}, \quad  \quad  \mathbf{B}^{+} \coloneq \mathbf{B}, && \quad \text{and} \quad \mathbf{F}^{+}  \coloneq \mathbf{F} &&&\qquad  \text{for} &&&&  x>0.
		\end{aligned}\right. 
	\end{equation*}
	Then by construction we have that 
	\begin{equation*}
		\left\{\begin{aligned}
			\partial_t \mathbf{u}^{\pm}  + \mathbf{B}^{\pm} \partial_x \mathbf{u}^{\pm} &  = \mathbf{F}^{\pm}  \qquad \text{for} \quad \pm x > 0
			\\
			{\mathbf{u}^{\pm}}_{|_{t=0}} & = \mathbf{0},
		\end{aligned}\right.
	\end{equation*}
	subject to the transmission conditions across $x=0$:
	\begin{equation*}
		\llbracket\mathbf{u}\rrbracket = {\mathbf{u}^{+}}{|_{x=0^+}} - {\mathbf{u}^{-}}{|_{x=0^-}}  = 0.
	\end{equation*}
	We can rewrite the transmission problem with a source as an initial boundary value problem by introducing the simple change of variable 
	\begin{equation*}
		\mathbf{U}(t,x) 
		= 
		\begin{pmatrix}
			\mathbf{u}^-(t,-x)
			\\ 
			\mathbf{u}^+(t,x)
		\end{pmatrix},
		\quad 
		\quad 
		\mathbb{B}(x)
		=
		\begin{pmatrix}
			-\mathbf{B}^{-}(-x) & \mathbf{0}_{2\times 2}\\
			\mathbf{0}_{2\times 2} & \mathbf{B}^{+}(x)
		\end{pmatrix},
		\quad  
		\quad 
		\mathbf{F}(t,x) 
		= 
		\begin{pmatrix}
			\mathbf{F}^-(t,-x)
			\\ 
			\mathbf{F}^+(t,x)
		\end{pmatrix},
	\end{equation*}
	for $x>0$. Then for all $T>0$ we have that the solution to the water waves equations satisfies the initial boundary value problem
	\begin{equation}\label{Eq: IBV}
		\left\{\begin{aligned}
			\partial_t \mathbf{U}  + \mathbb{B} \partial_x \mathbf{U} &  = \mathbf{F}   \qquad && \text{in} \quad (0,T)\times \R^+
			\\
			{\mathbf{U}}_{|_{t=0}} & = {\bf U}^{\rm in}
			\qquad && \text{on} \quad \R^+
			\\ 
			N{\mathbf{U}}_{|_{x=0}}  & = \mathbf{0}
			\qquad && \text{on} \quad (0,T),
		\end{aligned}\right.
	\end{equation}
	where $N$ is a $2 \times 4$ matrix defined in terms of the identity matrix and reads $N = \begin{pmatrix}
		- {\rm I}_{2\times 2} &  {\rm I}_{2\times 2}
	\end{pmatrix}$. Such systems are studied in \cite{IguchiLannes21}, even in the nonlinear case, and we have the \textit{a priori} estimate, which is a simple adaptation of Proposition 3.8 in this paper.
	\begin{prop}\label{Prop: transmission energy est}
		Let $t>0$ and suppose  ${\bf U} \in H^1 \left(\left[0, t\right); H^1(\R_+)\right)$ solves \eqref{Eq: IBV}. Then for any $t >0$ there is a constant $C>0$ independent from ${\bf U}$ such that
		\begin{align*}
			& \left|{\bf U}(t, \cdot)\right|_{H^1(\R_+)}  
			+
			\left|{\bf U}(t,0) \right|_{L^2\left((0, t)\right)}  	
			+
			\left|\partial_t{\bf U}(t,0) \right|_{L^2\left((0, t)\right)} 
			\\
			& \hspace{0.5cm}\leq C \left(|{\bf U}(0, \cdot)|_{H^1(\R_+)} +  \left|{\bf F}(t, \cdot )\right|_{L^2(\R_+)}+ \int_{0}^t\left(  \left|  {\bf F}(s, \cdot )\right|_{L^2(\R_+)} + \left|\partial_t {\bf F}(s, \cdot )\right|_{L^2(\R_+)}\right) \, {\rm d} s\right).
		\end{align*}
	\end{prop}
	\begin{remark}\label{Ref: on trace of source}
		Here special care is taken regarding the dependence with respect to the trace of the source term ${\bf F}$ and the implicit constants in time. In Proposition 3.8 of \cite{IguchiLannes21} the estimate also gives a control on 
		\begin{equation*}
			\left|\partial_x {\bf U}(t,0) \right|_{L^2\left((0, t)\right)}  	
		\end{equation*}
		but requires in turn ${\bf F}_{|_{x=0}} \in L^2((0,t))$. This would be delicate to control with respect to the shallow water parameter. The reason is that the water wave equations experience dispersive smoothing over the step, while shallow water equations will have a jump in the derivative of the surface. See Sections \ref{Sec: num} and \ref{Sec: bl} for a more detailed discussion on this point.  
	\end{remark}
	\begin{proof}
		The a priori estimate is based on using a Kreiss-symmetrizer to define an adapted energy; then we apply time derivatives from which we can deduce spatial regularity. We give the proof in three short steps.
		\\
		
		\noindent
		{\bf Step 1.} \textit{Construction of a Kreiss-symmetrizer}. To construct the symmetrizer, it is enough to verify Assumption 3.4 in \cite{IguchiLannes21}, which ensures the strict hyperbolicity of the system. Indeed, this is done in Section 6, example 6.2 of the same paper, and since $\mathbb{B}$ is a constant matrix one can apply  Lemma 3.9, in the same paper, to deduce that there exist a Kreiss-Symmetrizer $S \in \R^{4\times 4}$ with $S = S^{\rm T}$ and positive numbers $\alpha_0, \alpha_1, \beta_0, \beta_1$ such that
		\begin{align*}
			S\,\mathbb{B} = \left(S\,\mathbb{B} \right)^{\rm T},
		\end{align*}
		uniformly positive definite
		\begin{equation}\label{Eq: Equiv Kreiss}
			\qquad\qquad\alpha_0 |v|^2 \leq v^{\rm T} S\, v \leq \beta_0 |v|^2 \qquad v \in \R^4
		\end{equation}
		and 
		\begin{equation}\label{Eq: bc Kreiss}
			\qquad\qquad v^{\rm T} 	S\,\mathbb{B} v \leq - \alpha_1 |v|^2 + \beta_1 |Nv|^2 \qquad v \in \R^4.
		\end{equation}
		In other words, the natural energy functional is defined by
		\begin{align*}
			\mathcal{E}(t) = \frac{1}{2} \int_{\R_+} {\bf U}(t, x) \cdot  S \, {\bf U}(t, x) \, {\rm d}x.
		\end{align*}
		Here the choice of the energy is simpler than in \cite{IguchiLannes21}, where they also used weights to deal with a general variable-coefficient problem and used it to absorb lower-order terms. This is not necessary in this simple configuration and will allow us to obtain a sharper bound in time.  We will now give the $L^2-$estimate, then use it to deduce higher regularity.
		\\ 
		
		\noindent
		{\bf Step 2.} \textit{$L^2-$estimate.} For any ${\bf U} \in L^{\infty}\left((0, T); H^1(\R_+)\right)$ solving \eqref{Eq: IBV} there holds
		\begin{align*}
			\frac{\rm d}{{\rm d} t} \mathcal{E}(t) = \frac{1}{2}S \, \mathbb{B}{\bf U} \cdot {\bf U}_{|_{x=0}} + \int_{\R_+} S \, {\bf F} \cdot {\bf U}\, {\rm d}x.
		\end{align*}
		Integrating this inequality in time and imposing \eqref{Eq: Equiv Kreiss}--\eqref{Eq: bc Kreiss} implies 
		\begin{equation*}
			\frac{\rm d}{{\rm d} t} \mathcal{E}(t) +\alpha_1 \left|{\bf U}(t, 0)\right|^2
			\leq  \beta_0\left|{\bf F}(t, \cdot)\right|_{L^2(\R_+)}
			\left(\mathcal{E}(t)\right)^{1/2}.
		\end{equation*}
		Using the observation that $\left(\mathcal{E}(t)\right)^{1/2} \leq 	\left(\mathcal{E}(0)\right)^{1/2} + \tfrac{\beta_0}{2}\int_{0}^t\left|{\bf F}(s, \cdot)\right|_{L^2(\R_+)}\, {\rm d}s$, gives the final estimate:
		\begin{align*}
			\left|{\bf U}(t, \cdot)\right|_{L^2(\R_+)} + \left|{\bf U}(t,0) \right|_{L^2\left((0, T)\right)} \leq C \left(|{\bf U}^{\rm in}|_{L^2(\R_+)} + \int_{0}^t \left|{\bf F}(s, \cdot )\right|_{L^2(\R_+)} \, {\rm d} s\right),
		\end{align*}
		for some $C>0$. We will now use this bound to deduce higher regularity.
		\\ 
		
		\noindent
		{\bf Step 3.} \textit{$H^1-$estimate}. We first apply one time derivative to system \eqref{Eq: IBV} and deduce that
		\begin{align*}
			\left|\partial_t {\bf U}(t, \cdot)\right|_{L^2(\R_+)} + \left|\partial_t{\bf U}(t,0) \right|_{L^2\left((0, t)\right)} \leq C \left(|\left(\partial_t{\bf U}\right)(0, \cdot)|_{L^2(\R_+)} + \int_{0}^t \left|\partial_t {\bf F}(s, \cdot )\right|_{L^2(\R_+)} \, {\rm d} s\right).
		\end{align*}
		The spatial regularity is deduced from $\partial_x {\bf U} =  {\mathbb{B}}^{-1} \left(- \partial_t {\bf U} + {\bf F} \right)$  and gives 
		\begin{align*}
			\left|\partial_x  {\bf U}(t, \cdot)\right|_{L^2(\R_+)}   
			\leq C \left(|\left(\partial_x{\bf U}\right)(0, \cdot)|_{L^2(\R_+)} +  \left|{\bf F}(0, \cdot )\right|_{L^2(\R_+)}+ \int_{0}^t \left|\partial_t {\bf F}(s, \cdot )\right|_{L^2(\R_+)} \, {\rm d} s\right).
		\end{align*}
		Here we may also use the Fundamental Theorem of Calculus to see that 
		$$\left|{\bf F}(0, \cdot )\right|_{L^2(\R_+)} \leq  \left|{\bf F}(t, \cdot )\right|_{L^2(\R_+)} +  \int_0^t \left|\partial_t {\bf F}(s, \cdot )\right|_{L^2(\R_+)} \, {\rm d} s.$$
		Adding all these estimates completes the proof. \\ 
	\end{proof}

	\subsection{The zero-dispersion limit}\label{Sec: zero disp limit}
	We will now prove that the \lq\lq error\rq\rq\: in equation \eqref{Eq: Shallow water with error} is small with respect to $\mu$, which is a consistency result.  Then we will use it together with the well-posedness result and characterization of $\mathcal{H}_{\mu}^{n/2}(\R)$ to establish the convergence between the two models.

	\begin{thm}\label{Thm: Cv Step case}
		Let  $\mu, \gamma \in (0, \tfrac{1}{2})$, $\Omega$ and $\Gamma$ be as in Assumption \ref{Def: Assumption 3}. Let $(\zeta^{\mathrm{in}}, \psi^{\mathrm{in}}) \in H^{5}(\R)\times \dot{H}^{6}(\R)$ such that  both $\zeta^{\mathrm{in}}$ and $\psi^{\mathrm{in}}$ vanish on $\{|x|\leq c_0\}$ for some $c_0 \in (0, 1)$ independent of $\gamma$ and $\mu$.  Moreover, define  the initial flux by 
		$$q^{\rm in }  = {\mathcal{Q}_{\mu}} \psi^{\rm in} \in H^1(\R).$$ 
		Then there is a unique solution 
		$$(\zeta, \psi) \in C\big(\left[0,T\right] :  H^{1}(\R)\times \dot{H}^{1}(\R)\big),$$
		to the linear water waves equations \eqref{Eq: Linear ww}. Moreover, from the data $(\zeta^{\rm in}, q^{\rm in})$ there is a unique solution 
		$$(\zeta^{\rm sw}, q^{\rm sw}) \in C\left(\left[0, T\right] : H^1(\R_{\pm}) \times H^1(\R_{\pm}) \right),$$
		to the shallow water equations as a transmission problem \eqref{Eq: Linear shallow water transmission}--\eqref{Eq: trans cond} such that for any $t \in [0, T]$ there holds,
		\begin{equation*}
			|\zeta  - \zeta^{\rm sw}|_{L^{\infty}\left(\left[0,t\right]; H^1(\R_{\pm})\right)}
			+
			|q - q^{\rm sw}|_{L^{\infty}\left(\left[0,t\right]; H^1(\R_{\pm})\right)} \leq \big( \sqrt{\mu} +  \gamma \mu^{1/4} \big) (1+t) \, C\big(|\zeta^{\rm in}|_{{H}^{5}(\R)} + |\psi^{\rm in}|_{\dot{{H}}^6(\R)} \big).
		\end{equation*}

	\end{thm}
	\begin{remark} 
		In the final estimate of the statement, we do not control in $L^{\infty}((0,t))$ the one-sided traces at $x=0$ of the spatial derivatives of the two solutions. This is due to the lack of control on the trace of the source term; see Remark \ref{Ref: on trace of source}.  
	\end{remark} 
	\begin{proof}  
		The strategy of the proof is first to show that there exists a function $R: \R_+\times  \R \rightarrow \R$ satisfying 
		\begin{equation*}
			\sup\limits_{t \in [0, T]}\left( |R(t)|_{L^2\left( \R \right)} +  |\partial_t R(t)|_{L^2\left( \R \right)}\right)\leq  C \left(|\zeta^{\rm in}|_{{H}^{5}(\R)} + |\partial_x \psi^{\rm in}|_{H^{5}(\R)} \right),
		\end{equation*}
		such that
		\begin{equation*}
			\left\{\begin{aligned}
				\partial_t \zeta + \partial_x q  & = 0 
				\\ 
				\partial_t q +  h\partial_x \zeta & =\left( \sqrt{\mu} + \gamma \mu^{1/4} \right) R.
			\end{aligned}\right.
		\end{equation*}
		Then, in a second step, we will prove the convergence using the energy estimate given in Proposition \ref{Prop: transmission energy est}.  We give the proof of each point in two separate steps. 
		\\  
		
		\noindent
		{\bf Step 1.} \textit{Consistency.} To prove the first point of the statement, we need to prove that the rest satisfies:
		\begin{equation}\label{Eq: consistency}
			\max\limits_{j=1,2}\left\vert \partial_t^j \left( q - h\partial_x  \psi \right)  \right\vert_{L^2\left(\R \right)}
			\lesssim \left( \sqrt{\mu} + \gamma \mu^{1/4} \right) C
			.
		\end{equation}
		The idea is to make an elliptic estimate on ${\mathtt u}_j = \partial_t^j \left(\psi^{\mathfrak{h}} - \psi\right)$ which solves the following problem
		\begin{equation*}
			\left\{\begin{aligned}
				\Delta^{\mu} {\mathtt u}_j & = -\mu \partial_x^2 \partial_t^j \psi  &&\qquad  \text{in}   \quad \:\: \Omega
				\\ 
				{\mathtt u}_j   & =  0 &&\qquad  \text{on}  \quad \:\: \Gamma^{\rm D}
				\\ 
				\partial_{\rm n}^{\mu} {\mathtt u}_j & = -	\partial_{\rm n}^{\mu} \partial_t^j \psi &&\qquad  \text{on}  \quad \:\: \Gamma^{\rm N}.
			\end{aligned}\right.
		\end{equation*} 
		We will first provide a bound on the rest in terms of the solution in conformal coordinates.  Then we will analyze the dependence with respect to the conformal change of variables to bound the solution by the data. In particular, we first use the definition above and Cauchy-Schwarz to find that
		\begin{align*} 
			\left\vert \partial_t^j \left( q - h\partial_x  \psi \right)  \right\vert_{L^2\left(\R \right)} 
			= 
			\left\vert \int_{-h}^0 \partial_x {\mathtt u}_j\:\mathrm{d} y \right\vert_{L^2(\R)}
			\leq C
			\left\| \partial_x {\mathtt u}_j  \right\|_{L^2(\Omega)}.
		\end{align*}	
		We will now use that the conformal mapping is uniformly bounded in the energy space and then use it to make sharp estimates on the error. To do so, we recall the notation $\sigma^{-1} (x, y) = \left( x, \frac{y}{\sqrt{\mu}} \right)$ for any $(x,y) \in \R \times (-1, 0)$ and also recall the conformal map $F: \mathrm{S}_{\mu} \rightarrow \Omega_{\mu}$ from a strip in the upper--half plane, given in Lemma \ref{lemma: conf map}. Then using the relations in Remark \ref{Remark: change of variables} gives
		\begin{align*}
			{\rm I}_1 
			& = {\mu}^{-\tfrac{1}{4}}
			\left\| \partial_x \left({ {\mathtt u}_j}\circ {\sigma}^{-1} \right) \right\|_{L^2(\Omega_{\mu})}
			\\ 
			& 
			=
			{\mu}^{-\tfrac{1}{4}}
			\left\| \tfrac{1}{\sqrt{J}}\left(U_{\alpha} \partial_\alpha + U_{\beta} \partial_\beta \right) \left( {\mathtt u}_j \circ {\sigma}^{-1}\circ F\right) \right\|_{L^2({\rm S}_{\mu})} 
			\\ 
			&  
			\leq 
			{\mu}^{-\tfrac{1}{4}}
			\left\|  \nabla \left( {\mathtt u}_j \circ {\sigma}^{-1} \circ F\right) \right\|_{L^2({\rm S}_{\mu})}, 
		\end{align*}
		where we used $\tfrac{U_{\alpha}^2 + U_{\beta}^2}{J} = 1$ in the final estimate. Lastly, we let $F_{\mu} =  {\sigma}^{-1} \circ F  \circ\sigma $ and transform the function to the unit strip to see that
		\begin{align*}
			{\rm I}_1  
			\leq  
			\frac{1}{\sqrt{\mu}}
			\left\|  \nabla^{\mu} \left( {\mathtt u}_j \circ F_{\mu}\right) \right\|_{L^2({\rm S}_{1})}.
		\end{align*} 
		Then we have the formula
		\begin{align*}
			{\mathtt u}_j  \circ F_{\mu} (\alpha, \beta)= \left( \frac{\cosh\left((\beta+1)\sqrt{\mu }|{\rm D}|  \right)}{\cosh\left(\sqrt{\mu}|{\rm D}|\right)} - 1\right) \partial_t^j \psi_0(\alpha),
		\end{align*}
		where $F = \sigma^{-1}\circ F\circ\sigma$ on $\beta = 0$ and $\psi_{0}(\alpha) = \psi(U(\alpha,0))$. As a result of Plancherel's identity, Fubini, and a Taylor expansion we find 
		\begin{align*}
			\left\| \nabla^{\mu}_{\alpha, \beta} ({\mathtt u}_j \circ F_{\mu}) \right\|_{L^2({\rm S}_{1})}^2
			& = 
			\mu 
			\int_{\R} 
			\left| \mathcal{F}\left(\partial_{\alpha}\partial_t^j\psi_0(\cdot)\right)(\xi)  \right|^2
			\int_{-1}^0
			\left( \frac{\cosh\left((\beta+1)\sqrt{\mu }|{\xi}|\right)}{\cosh\left(\sqrt{\mu}|{\xi}|\right)} - 1\right)^2 \, {\rm d} \beta\, 
			{\rm d}  \xi 
			\\ 
			& \leq
			\mu^2 C
			\int_{\R} 
			\left| \partial_{\alpha}^2\left( \partial_t^j  \psi_0 (\alpha)  \right)\right|^2 
			{\rm d}\alpha.
		\end{align*}
		We will now bound the solution by the data in $\mathcal{H}_{\mu}^{5/2}(\R) \times \dot{\mathcal{H}}_{\mu}^3(\R)$, before concluding with an estimate for the data in the usual Sobolev space. In particular, using Lemma \ref{lemma: alpha to G} and  Theorem \ref{thm: Wp ww} gives the following bound in terms of the initial data:
		\begin{align*}
			\max\limits_{j=1,2}\left\vert \partial_t^j \left( q - h\partial_x  \psi \right)  \right\vert_{L^2\left(\R \right)} 
			& \lesssim 
			\sqrt{\mu}	\max\limits_{j=1,2} \left| \partial_{\alpha}^2 \partial_t^j \psi_0\right|_{L^2(\R)}
			\\
			& \lesssim 
			\sqrt{\mu}(1 + \gamma\mu^{-1/4})C\left( |\zeta|_{\dot{\mathcal{H}}_{\mu}^2(\R)} + |\psi|_{\dot{\mathcal{H}}_{\mu}^3(\R)} \right)
			\\ 
			& \lesssim 
			\sqrt{\mu}(1 + \gamma\mu^{-1/4})C\left( |\zeta^{\rm in}|_{\mathcal{H}_{\mu}^{5/2}(\R)} + |\psi^{\rm in}|_{\dot{\mathcal{H}}_{\mu}^3(\R)} \right).
		\end{align*}

		One could have supposed that the data remains bounded in the abstract functional setting. However, the physical meaning of such an assumption is not clear due to the nonlocal definition of these spaces. Therefore, we will provide a bound in the usual Sobolev space under vanishing conditions on the data using Lemma \ref{lemma: G to dx}:
		\begin{align*}
			\max\limits_{j=1,2}\left\vert \partial_t^j \left( q - h\partial_x  \psi \right)  \right\vert_{L^2\left(\R \right)}  
			& \lesssim 
			\sqrt{\mu}(1 + \gamma\mu^{-1/4})C\left( |\zeta^{\rm in}|_{{H}^{5}(\R)} + |\psi^{\rm in}|_{\dot{{H}}^6(\R)} \right).
		\end{align*}

		We may now use this step to conclude the proof. \\ 

		\noindent
		{\bf Step 2.} \textit{Convergence}. First, the well-posedness of the transmission problem is proved in \cite{IguchiLannes21}. Secondly, the proof of the convergence of the solutions between the two models is given; we define the variables:
		\begin{align*}
			\tilde{\zeta} = \zeta - \zeta^{\rm sw} \quad \text{and} \quad \tilde{q} = q - q^{\rm sw}.
		\end{align*}
		Then the difference satisfies 
		\begin{equation*}
			\left\{\begin{aligned}
				\partial_t  \tilde \zeta + \partial_x \tilde q & = 0 
				\\ 
				\partial_t \tilde q +  h\partial_x \tilde\zeta & = \sqrt{\mu}(1 + \gamma\mu^{-1/4}) R,
			\end{aligned}\right.
		\end{equation*}
		with transmission conditions 
		\begin{equation*}
			\llbracket\tilde \zeta\rrbracket =	\llbracket \tilde q \rrbracket = 0.
		\end{equation*}
		The system can therefore be reduced to the initial boundary value problem \eqref{Eq: IBV} and the conclusion now follows from Proposition \ref{Prop: transmission energy est}. 
	\end{proof}

	\section{Numerical simulations: dispersive smoothing}\label{Sec: num}

	In this section, we numerically illustrate the localized dispersive smoothing that develops near the step for solutions of the linear water-wave equations. We compare the water-wave solution and its spatial derivative with their shallow-water counterparts and examine the residual
	$$
	R=\partial_t\left(q-h\partial_x\psi\right),
	$$
	which measures the discrepancy between the two models. As explained in Remark \ref{Ref: on trace of source}, controlling the traces of the derivatives of the two solutions would require this residual to be small at $x=0$. The aim is therefore to determine whether the dynamics near the step require a refined approximation that accounts for localized dispersive effects. At the same time, the shallow-water model accurately captures wave propagation away from the step. 
	
	The numerical approximation of the linear water-wave equations is based on a conformal transformation combined with a Fourier spectral method, following the classical conformal–spectral approach of Trefethen \cite{Trefethen80} (see also \cite{DriscollTrefethen02} for the numerical computation of more general Schwarz--Christoffel mappings). The method is described in the next section.

	\subsection{A conformal Fourier-spectral method} 
	We now describe the numerical scheme used to approximate solutions of the linear water-wave equations. After truncating the conformal variable to a periodic interval, we approximate the pullbacks of the unknowns by trigonometric polynomials. The conformal representation of the Dirichlet--Neumann operator then allows its principal part to be evaluated mode by mode using Fourier multipliers. 	In particular, from the definition of the  map \eqref{Eq: Conf. map Step}, we have that $F = \sigma^{-1}\circ F\circ\sigma$ on $\beta = 0$ and can therefore denote the restriction of the conformal map to the upper boundary by
	\begin{equation*}
		U_0(\alpha)=U(\alpha,0).
	\end{equation*}
	and throughout the section we may express $f : \R \rightarrow \R$ in conformal coordinates by
	\begin{equation*}
		f_0(\alpha)=f\circ U_0(\alpha).
	\end{equation*}
	With this notation at hand, we may introduce the $L^2$--projection operator onto the space of trigonometric polynomials.
	\begin{Def}\label{Def: Conformal Fourier projection}
		Let $L>0$ and define $\R_{2L} = \R/(2L\mathbb Z)$ with $k_L=\frac{\pi}{L}$. Then we denote $\mathcal{T}_N$ the space of trigonometric polynomials of degree $N$:
		\begin{equation*}
			\mathcal T_N
			=
			\mathrm{span}
			\left\{
			e^{ink_L\alpha}
			:
			|n|\leq N
			\right\}.
		\end{equation*}
		For $f\in L^2(\R_{2L})$, we write
		\begin{equation*}
			\widehat f_n
			=
			\frac{1}{2L}
			\int_{-L}^{L}
			f(\alpha)e^{-ink_L\alpha}
			\,{\rm d}\alpha
		\end{equation*}
		and define the $L^2$-orthogonal projection onto $\mathcal T_N$ by
		\begin{equation*}
			P_Nf(\alpha)
			=
			\sum_{|n|\leq N}
			\widehat f_n e^{ink_L\alpha}.
		\end{equation*}
		For a function $f$ defined on the free surface, whose pullback
		\begin{equation*}
			f_0=f\circ U_0,
		\end{equation*}
		is identified with a function on $\R_{2L}$, we define its conformal Fourier projection by
		\begin{equation*} 
			\left(\widetilde P_Nf\right)\circ U_0
			=
			P_Nf_0.
		\end{equation*}

	\end{Def}
	Before setting up the full scheme, we first recall the definitions of the main operators involved and their projections onto $\mathcal{T}_{N}$.

	\subsubsection{Projected Dirichlet--Neumann operator and flux}

	The formula for the Dirichlet--Neumann operator obtained in \eqref{Eq: DN fourier} and that was used in the previous section reads
	\begin{equation*}
		({\mathcal G}_{\mu}\psi)\circ U_0
		=
		\frac{1}{U_0'}\, 
		{\rm g}_{\mu}({\rm D})\psi_0,
		\qquad
		{\rm g}_{\mu}(\xi)
		=
		\frac{|\xi|}{\sqrt{\mu}}
		\tanh\left(\sqrt{\mu}|\xi|\right).
	\end{equation*}
	We will now define its projection in a way that preserves the symmetry of the operator:
	\begin{Def}
		We define the projected Dirichlet--Neumann operator by
		\begin{equation*}
			{\mathcal G}_{\mu}^N
			=
			\widetilde P_N{\mathcal G}_{\mu}\widetilde P_N.
		\end{equation*}
		Equivalently, for $\psi_0=\psi\circ U_0 \in L^2(\R_{2L})$, we have
		\begin{equation*}
			\left({\mathcal G}_{\mu}^N\psi\right)\circ U_0
			=
			P_N
			\left[
			\frac{1}{U_0'}
			{\rm g}_{\mu}({\rm D}_{\alpha})
			P_N\psi_0
			\right].
		\end{equation*}
	\end{Def}
	\begin{remark}
		If $\psi_0^N\in\mathcal T_N$, then we simply write
		\begin{equation*}
			\left({\mathcal G}_{\mu}^N\psi^N\right)\circ U_0
			=
			P_N
			\left[
			\frac{1}{U_0'}
			{\rm g}_{\mu}({\rm D}_{\alpha})
			\psi_0^N
			\right].
		\end{equation*}
	\end{remark}
	Next, we treat the horizontal discharge. It is recovered from the relation
	\begin{equation*}
		\partial_xq 
		=
		-{\mathcal G}_{\mu}\psi.
	\end{equation*}
	Then using Remark \ref{Remark: change of variables} to find the relation $	(\partial_x f)\circ U_0	=	\frac{1}{U_0'}\partial_{\alpha} f_0$ we may define the discharge in conformal coordinates by
	\begin{align*}
		q_0
		=
		{\rm T}_{\mu}({\rm D})\partial_{\alpha}\psi_0
		\qquad \text{where}\qquad 
		{\rm T}_{\mu}(\xi)
		=
		\frac{\tanh\left(\sqrt{\mu}|\xi|\right)}{\sqrt{\mu}|\xi|}.
	\end{align*}
	\begin{Def}\label{Def: Projected horizontal discharge}
		Let $\psi \in \dot{H}^1(\R)$ then denote by $\mathcal Q_{\mu}$ the horizontal-discharge operator defined as
		\begin{equation*}
			\left(\mathcal Q_{\mu}\psi\right)\circ U_0
			=
			{\rm T}_{\mu}({\rm D})\partial_{\alpha} \psi_0.
		\end{equation*} 
		We define the projected horizontal-discharge operator by
		\begin{equation*}
			\mathcal Q_{\mu}^N
			= 
			\widetilde P_N \mathcal Q_{\mu}
		\end{equation*}

	\end{Def}
	We will now use these definitions to define the scheme for the linear water-wave systems to make comparisons with the shallow water equations.

	\subsubsection{Projection of the linear water-wave equations onto $\mathcal T_N$}
	We now define its Fourier--Galerkin approximation and the corresponding time discretization for the linear water-wave system \eqref{Eq: Linear ww}.
	\begin{Def}\label{Def: Fourier RK approximation}
		Let $\left(\zeta^{\rm in}, \psi^{\rm in}\right)^{\rm T} \in L^2(\R_{2L})\times \dot{H}^{1/2}_{\mu}(\R_{2L})$ and $ {\bf u}^N = \left(\zeta_0^N, \psi_0^N\right)^{\rm T} \in \mathcal T_N\times\mathcal T_N$ and define the operator
		\begin{equation*}
			{\bf A}^N
			:
			\mathcal T_N\times\mathcal T_N
			\longrightarrow
			\mathcal T_N\times\mathcal T_N,
			\qquad 
			{\bf A}^N {\bf u}^N
			=
			\begin{pmatrix} 
				P_N
				\left[
				\frac{1}{U_0'}
				{\rm g}_{\mu}({\rm D}_{\alpha})
				\psi_0^N
				\right] \\
				- \zeta_0^N
			\end{pmatrix}.
		\end{equation*}
		Then we say ${\bf u}^N$ is a Fourier--Galerkin approximation of the linear water-wave system \eqref{Eq: Linear ww} if and only if it is the solution of
		\begin{equation}\label{Eq: Numerical linear WW}
			\partial_t{\bf u}^N
			=
			{\bf A}^N{\bf u}^N,
		\end{equation}
		with initial condition
		\begin{equation*}
			{\bf u}^N(0)
			=
			\begin{pmatrix}
				P_N\left(\zeta^{\rm in}\circ U_0\right)
				\\
				P_N\left(\psi^{\rm in}\circ U_0\right)
			\end{pmatrix}.
		\end{equation*}

	\end{Def}
	Finally, the semi-discrete system is integrated in time using the classical explicit fourth-order Runge--Kutta method for the numerical illustrations.

	\subsection{Numerical experiment}
	
	We now compare the solutions of the linear water-wave equations obtained from the Fourier--Galerkin scheme with the corresponding solutions of the linear shallow-water equations. We consider two different classes of initial data. First, we take a right-going wave packet compactly supported away from the step. In this case, the initial data vanish in a neighborhood of $x=0$ and therefore satisfy the shallow-water transmission conditions.  We begin by describing the numerical setup. 	 
	\subsubsection{Set-up}  We first introduce the wave packet used in the experiment.
	
	\begin{Def}\label{Def: Wave packet}
		Let $A \in (0,  1-\gamma)$, $\ve \in (0, 1)$, and $x_\star \in \R$. Then we define the cut-off function by
		$$
		\chi(s)
		=
		\begin{cases}
			\displaystyle
			\exp\left(1-\frac{1}{1-s^2}\right),
			&
			|s|<1, \\
			0,
			&
			|s|\geq1,
		\end{cases}
		$$
		and the wave packet centered at $x_{\star}$ reads
		\begin{equation}\label{Eq: Localized wave packet}
			\mathfrak w(x)
			=
			A\cos(x-x_{\star})
			\chi\left(\varepsilon(x-x_\star)\right).
		\end{equation}
	\end{Def}

	\begin{figure}[h!]
		\centering
		\includegraphics[width=\textwidth]{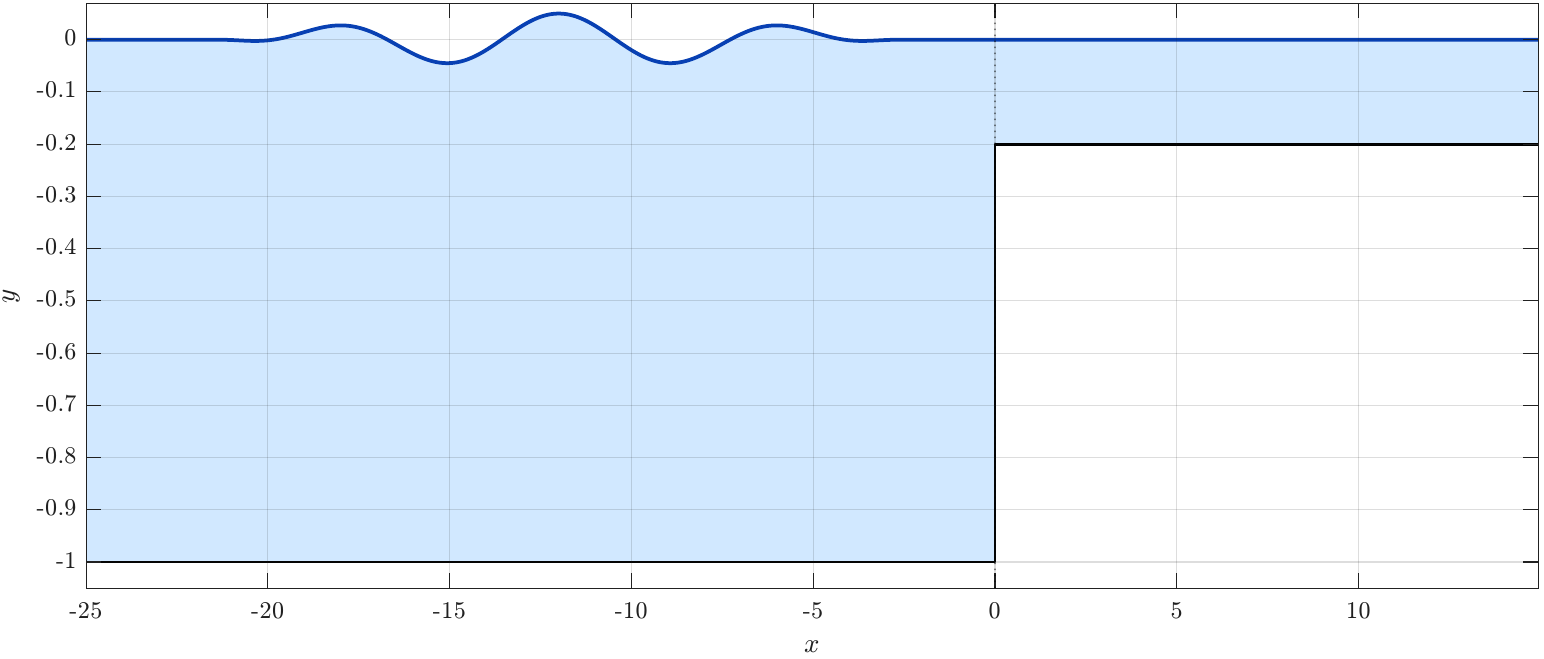}
		\caption{The wave packet \eqref{Eq: Localized wave packet} used as the initial surface elevation in the first numerical experiment, with
			$A=0.05$, $\gamma= 0.8$, $x_{\star}=-12$, $\varepsilon=0.1$.}
		\label{Fig: Initial data}
	\end{figure}
	
	We denote the nondimensional wave speeds on the two sides of the step by
	\begin{equation*}
		c_-
		=
		1,
		\qquad
		c_+
		=
		\sqrt{ (1-\gamma)}.
	\end{equation*}
	The initial surface elevation and discharge are chosen as
	\begin{equation}\label{Eq: Initial wave packet}
		\zeta^{\rm in}
		=
		\mathfrak w,
		\qquad
		q^{\rm in}
		=
		c_-
		\mathfrak w.
	\end{equation}
	The relation $q^{\rm in}=c_-\zeta^{\rm in}$ ensures that the initial packet is right-going. For the linear water-wave system, we choose the initial surface potential $\psi^{\rm in}$ so that
	\begin{equation}\label{Eq: Initial potential from discharge}
		\mathcal Q_{\mu}\psi^{\rm in} 
		=
		q^{\rm in}.
	\end{equation}
	At the discrete level, \eqref{Eq: Initial potential from discharge} is inverted mode by mode using the multiplier ${\rm T}_{\mu}$.

	\subsubsection{A wave packet localized away from the step}
	
	Since the packet is supported away from the step, we assume that
	\begin{equation*}
		x_{\star}+\ve^{-1}<0,
	\end{equation*}
	which implies $\zeta^{\rm in}$ and $q^{\rm in}$ vanish in a neighborhood of $x=0$, and the transmission conditions are automatically satisfied at the initial time. To be precise, we choose a wave packet with the following parameters:
	\begin{table}[H]
		\centering
		\small
		\renewcommand{\arraystretch}{1.12}
		\setlength{\tabcolsep}{12pt}
		\begin{tabular}{c c c c c}
			\hline
			$\gamma$& $A$ & $\quad x_{\star}$ & $\varepsilon$
			\\
			\hline
			$0.8$ & $0.05$ & $-12$ & $0.1$
			\\
			\hline
		\end{tabular} 
	\end{table}

	For these initial data, the corresponding shallow-water solutions are given by:
	\begin{lemma} 
		Define the reflection and transmission coefficients by
		\begin{equation*}
			{\rm R}_{\gamma}
			=
			\frac{c_--c_+}{c_-+c_+},
			\qquad
			{\rm T}_{\gamma}
			=
			\frac{2c_-}{c_-+c_+}.
		\end{equation*}
		Then the solution of \eqref{Eq: Linear shallow water transmission}--\eqref{Eq: trans cond} with initial data
		\begin{equation*}
			\zeta^{\rm in}
			=
			\mathfrak w,
			\qquad
			q^{\rm in}
			=
			c_-\mathfrak w
		\end{equation*}
		is given by
		\begin{equation}\label{Eq: Explicit shallow water solution}
			\zeta^{\rm sw}(t,x)
			=
			\begin{cases}
				\mathfrak w(x-c_-t)
				+
				{\rm R}_{\gamma}
				\mathfrak w(-x-c_-t)
				&
				\text{if }
				x<0
				\\ 
				{\rm T}_{\gamma}
				\mathfrak w
				\left(
				\displaystyle
				\frac{c_-}{c_+}x-c_-t
				\right)
				&
				\text{if }
				x>0,
			\end{cases}
		\end{equation}
		and
		\begin{equation}\label{Eq: Explicit shallow water discharge}
			q^{\rm sw}(t,x)
			=
			\begin{cases}
				c_-
				\left[
				\mathfrak w(x-c_-t)
				-
				{\rm R}_{\gamma}
				\mathfrak w(-x-c_-t)
				\right]
				&
				\text{if }
				x<0,
				\\
				c_+{\rm T}_{\gamma}
				\mathfrak w
				\left(
				\displaystyle
				\frac{c_-}{c_+}x-c_-t
				\right)
				&
				\text{if }
				x>0.
			\end{cases}
		\end{equation}
	\end{lemma}
	\begin{proof}
		The derivation is classical; see, for instance, \cite{Mei89}. We include the details here for the present configuration to identify the reflection and transmission coefficients explicitly. In particular, on each half-line the linear shallow-water system reads
		\begin{equation*}
			\partial_t\zeta^{\rm sw}
			+
			\partial_x q^{\rm sw}
			=
			0,
			\qquad
			\partial_tq^{\rm sw}
			+
			c_\pm^2\partial_x\zeta^{\rm sw}
			=
			0.
		\end{equation*}
		Differentiating the first equation with respect to \(t\) and the
		second one with respect to \(x\), we find that
		\begin{equation*}
			\partial_t^2\zeta^{\rm sw}
			-
			c_\pm^2\partial_x^2\zeta^{\rm sw}
			=
			0
			\qquad
			\mbox{on }
			\R_\pm.
		\end{equation*}
		Since the initial packet is right-going and supported in \(\R_-\),
		the solution is therefore of the form
		\begin{equation*}
			\zeta^{\rm sw}(t,x)
			=
			\begin{cases}
				\mathfrak w(x-c_-t)
				+
				{\rm R}_{\gamma}
				\mathfrak w(-x-c_-t),
				&
				x<0,
				\\[1mm]
				{\rm T}_{\gamma}
				\mathfrak w
				\left(
				\displaystyle
				\frac{c_-}{c_+}x-c_-t
				\right),
				&
				x>0,
			\end{cases}
		\end{equation*}
		where ${\rm R}_{\gamma}$ and ${\rm T}_{\gamma}$ remain to be determined. Notice that the reflected term is left-going with speed $c_-$, whereas the transmitted term is right-going with speed $c_+$.
		
		We next recover the discharge from
		\begin{equation*}
			\partial_xq^{\rm sw}
			=
			-\partial_t\zeta^{\rm sw},
		\end{equation*}
		and for $x<0$ we have  
		\begin{align*}
			q^{\rm sw}(t,x)
			& =
			c_-
			\int_{-\infty}^{x}
			\partial_x
			\left[
			\mathfrak w(x-c_-t)
			-
			{\rm R}_{\gamma}
			\mathfrak w(-x-c_-t)
			\right]
			{\rm d}x
			\\ 
			& 
			=
			c_-
			\left[
			\mathfrak w(x-c_-t)
			-
			{\rm R}_{\gamma}
			\mathfrak w(-x-c_-t)
			\right].
		\end{align*}
		Similarly, for $x>0$ we find that 
		\begin{equation*}
			q^{\rm sw}(t,x)
			=
			c_+{\rm T}_{\gamma}
			\mathfrak w
			\left(
			\frac{c_-}{c_+}x-c_-t
			\right).
		\end{equation*}
		Finally, continuity of $\zeta^{\rm sw}$ and $q^{\rm sw}$ at
		$x=0$ gives
		\begin{equation*}
			1+{\rm R}_{\gamma}
			=
			{\rm T}_{\gamma},
			\qquad \text{and} \qquad 
			c_-\left(1-{\rm R}_{\gamma}\right)
			=
			c_+{\rm T}_{\gamma},
		\end{equation*}
		and solving these relations concludes the proof.
	\end{proof}
	
	We first compare the qualitative evolution of the two solutions. Figure \ref{Fig: Scattering sequence} displays $\zeta$ and $q$ at times $t=8$, $16$, and $24$, as the incident wave packet interacts with the step and separates into reflected and transmitted components. These snapshots assess how accurately the shallow-water model reproduces the propagation and scattering predicted by the water-wave model.
	\begin{figure}[h!]
		\centering
		\includegraphics[height = 11cm, width=\textwidth]{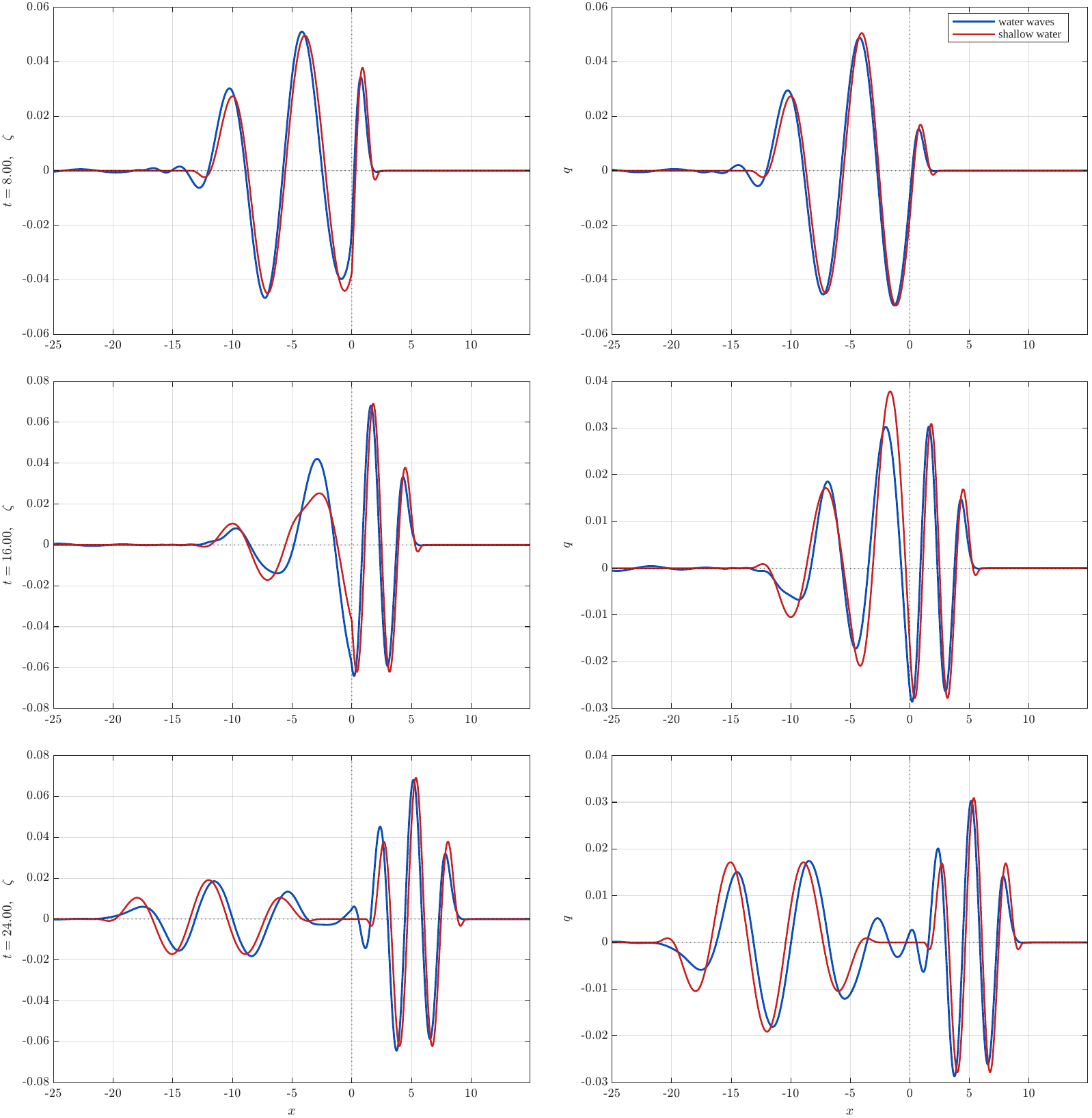}
		\caption{Evolution of the water-wave (blue) and shallow-water (red) solutions for $\mu = 0.15$, $\gamma = 0.8$, $A = 0.05$, and $\ve = 0.1$. The dotted line marks the position of the step.}
		\label{Fig: Scattering sequence}
	\end{figure}
	
	We take $\mu=0.15$ and $\gamma=0.8$ so that the dispersive effects and the influence of the depth discontinuity are clearly visible. The two models remain in close agreement at leading order, with the shallow-water solution reproducing the principal reflected and transmitted waves. As time progresses, however, the water-wave solution develops a more pronounced oscillatory tail near the step and a slight phase lag relative to the shallow-water solution. The abrupt change in depth excites shorter spatial scales whose dispersive evolution is not described by the shallow-water model in this moderately large dispersive case. For smaller values of $\mu$, additional numerical experiments show close agreement throughout the simulated time interval. In more dispersive regimes, these corrections may no longer be negligible, and a model retaining higher-order dispersion would be required; see Remark \ref{Remark: on the dispersive case for a small case}.

	We next quantify the difference between the two solutions. At each comparison time $t_n$, we define
	\begin{equation}\label{Eq: Numerical comparison errors}
		\widetilde\zeta_n
		=
		\zeta_n^N-\zeta_n^{\rm sw},
		\qquad
		\widetilde q_n
		=
		\mathcal Q_\mu^N\psi_n^N-q_n^{\rm sw}.
	\end{equation}
	For each value of $\mu$, we compute the maximum of the broken $H^1-$errors over the comparison times. Figure \ref{Fig: Convergence rates} displays these errors as functions of $\mu$ on logarithmic axes, together with $\sqrt{\mu}$ as a reference slope.
	\begin{figure}[h!]
		\centering
		\includegraphics[width=\textwidth]{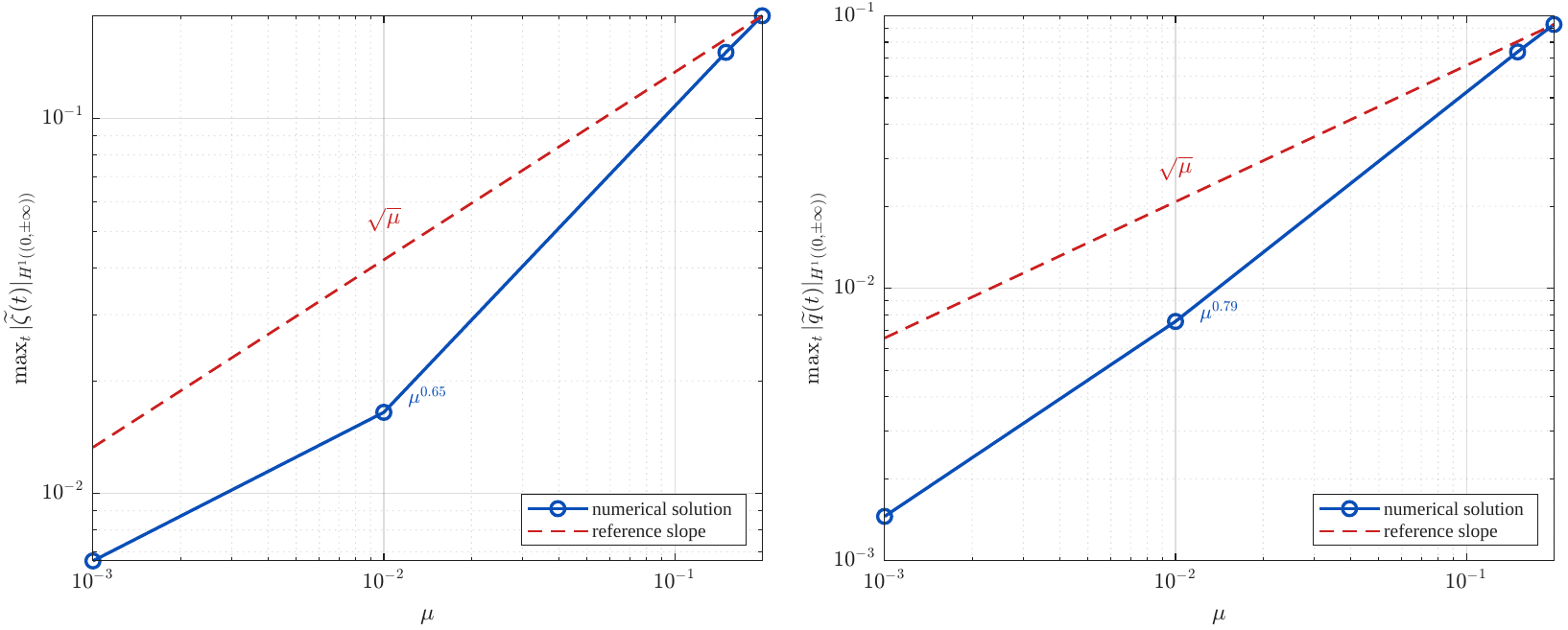}
		\caption{The blue lines are $H^1(\R_{\pm})$ errors for $\widetilde{\zeta}_n$ and $\widetilde{q}^N_n$ for $\mu \in \{ 0.2, 0.15, 0.01, 0.001\}$, $\gamma = 0.8$, $A = 0.05$, and $\ve = 0.1$. The dashed red line indicates the reference slope $\sqrt{\mu}$.}
		\label{Fig: Convergence rates}
	\end{figure}
	The observed rate depends on the parameters and the initial data. In other test cases, it was found to be close to $\sqrt{\mu}$ or faster. Taken together, these computations suggest that the $\sqrt{\mu}$ rate may be sharp for the step geometry, although this cannot be concluded from the numerical evidence alone. This rate is stronger than the $\mu^{1/4}$ rate guaranteed by our theoretical estimate for a step of order one, but slower than the $\mathcal O(\mu)$ accuracy generally obtained for sufficiently smooth bathymetries \cite{WWP}. This supports the view that the loss of regularity of the harmonic extension $\psi^{\mathfrak{h}}$ at the step affects the convergence rate.
	
	The convergence rate also depends on the topology under consideration. In particular, convergence in the broken $H^1$ norm controls the derivatives in $L^2$ on either side of the step, but does not control their one-sided traces at $x=0$. As discussed in Remark \ref{Ref: on trace of source}, one can attain this control if we have an $L^{\infty}-$bound on the residual obtained by inserting the water-wave solution into the shallow-water momentum relation:
	\begin{equation}\label{Eq: Shallow-water residual}
		R^N_n
		=
		-
		\left(
		\mathcal Q_\mu\zeta^N_n-h\partial_x\zeta^N_n
		\right).
	\end{equation}
	To examine the behavior of the residual close to the step, we introduce the stretched variable
	\begin{equation*}
		X=\frac{x}{\sqrt{\mu}},
	\end{equation*}
	that resolves the rapidly varying correction on its natural spatial scale. Since the discharge term $\mathcal Q_\mu\zeta^N$ has a continuous trace across the step, whereas $h\partial_x\zeta^N$ generally does not, the residual has distinct one-sided traces at $x=0$. The left panel in Figure \ref{Fig: disp smoothing} shows that this defect is localized near the step, it becomes concentrated on an $\mathcal{O}(\sqrt{\mu})$ scale, and is small in the broken $H^1-$norm.  The right panel shows the lack of convergence in $L^{\infty}_{t,x}$, and in some cases one can find instances when the rate is negative.
	\begin{figure}[h!]
		\centering
		\includegraphics[width=\textwidth]{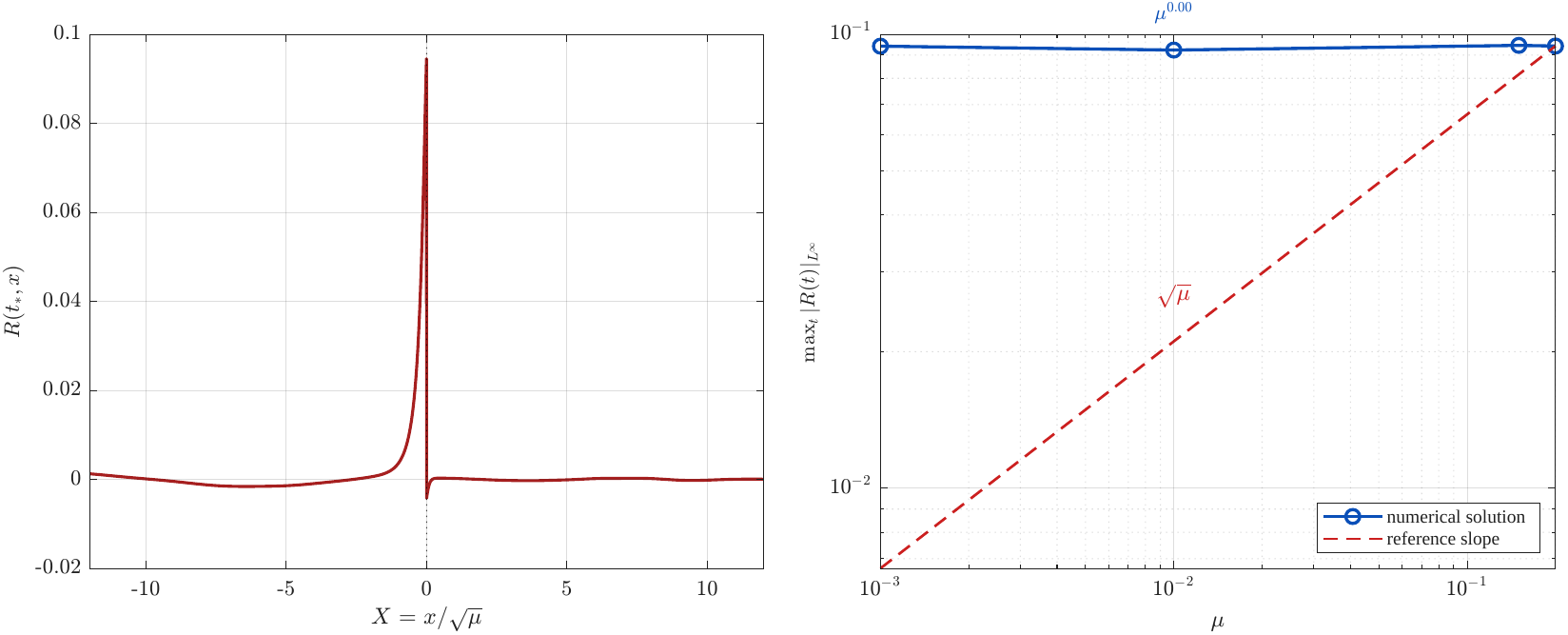}
		\caption{The behaviour of the shallow-water residual $R^N_n$ defined by \eqref{Eq: Shallow-water residual} for $\gamma = 0.8$, $A = 0.05$, and $\ve = 0.1$.  Left: profile of $R^N_n$ at one time in the stretched variable $X=x/\sqrt{\mu}$ for $\mu=0.15$. Right: The $L^\infty-$bound on $R^N_n$ for different values of $\mu \in \{ 0.2, 0.15, 0.01, 0.001\}$ (blue) with $\sqrt{\mu}$ shown as a reference slope (red).}
		\label{Fig: disp smoothing}
	\end{figure}
	The numerical profiles also reveal a marked difference in the regularity of the two models at the step. The higher-order shallow-water transmission conditions do not, in general, enforce the continuity of $\partial_x\zeta_n^{\rm sw}$ across $x=0$. By contrast, for smooth initial data, the water-wave solution remains smooth along the upper boundary. As illustrated in Figure \ref{Fig: Derivative comparison}, $\partial_x(\zeta^N_n-\zeta_n^{\rm sw})$ has a jump at $x=0$ and indicates that $\zeta^N_n- \zeta_n^{\rm sw}$ is at most in $H^{s}(\R)$ for $s<3/2$. 
	\begin{figure}[h!]
		\centering
		\includegraphics[height = 6cm, width=\textwidth]{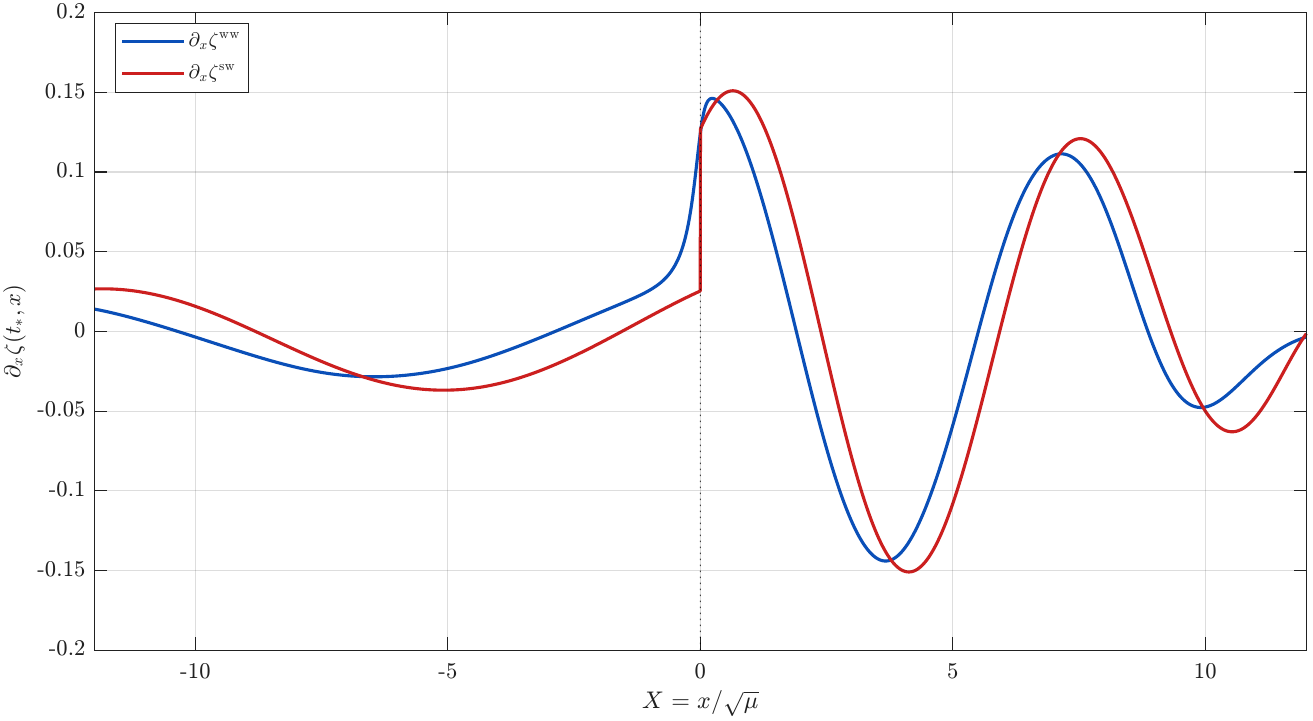}
		\caption{Comparison of $\partial_x\zeta^N_n$ and $\partial_x\zeta_n^{\rm sw}$ at the same time as in Figure \ref{Fig: disp smoothing}. 
		}
		\label{Fig: Derivative comparison}
	\end{figure}

	The difference between the models represents the dispersive part of the waves, and the simulations show that it produces a narrow, continuous transition localized near the step. This behavior is reminiscent of the dispersive boundary layers arising in dispersive perturbations of hyperbolic initial-boundary and transmission problems discussed in the introduction. We refer to this localized transition region as a step-induced dispersive boundary layer, which is the topic of Section \ref{Sec: bl}.

	\section{Analysis of dispersive boundary layers}\label{Sec: bl}

	In this section, we aim to provide a better asymptotic description of the linear water-wave equations near the transition region localized around the step. We know from the previous sections that we need to improve the asymptotic description of the flux, where we expanded the operator:
	\begin{align*}
		\partial_ t q = -  \mathcal{Q}_{\mu} \zeta.
	\end{align*}
	This is straightforward using the formulation in conformal coordinates, given by Definition \ref{Def: Projected horizontal discharge}, where we let $\alpha\in \R$ and $U_0(\alpha) = U(\alpha, 0)$ as defined in \eqref{Eq: Conf. map Step}, to define the smooth transition function $h_{\mu, \gamma}: \R \rightarrow \R$ by
	\begin{equation*}
		h_{\mu, \gamma} = U'_0\circ U_{0}^{-1}.
	\end{equation*}
	This function corresponds to the \lq\lq conformal depth\rq\rq\, appearing in Cathala’s shallow-water model for polygonal bathymetries \cite{Cathala2013Polygonal, Cathala2013Thesis}. While the corresponding expansion of the Dirichlet--Neumann operator is justified there, uniformity with respect to the shallow-water parameter in the sharp-step regime is not addressed. Using the abstract functional framework, we have the following expansion.
	\begin{lemma} 
		Let  $\mu, \gamma \in (0, \tfrac{1}{2})$, $\Omega$ and $\Gamma$ be as in Assumption \ref{Def: Assumption 3}. Also, let $p =2$ or $\infty$. Then for any $f \in \mathscr{S}(\R)$ there holds,
		\begin{equation}
			|\mathcal{Q}_{\mu} f - h_{\mu, \gamma}\partial_x f|_{L^{p}(\R)} \leq C \mu^{1/(2p) + 1/4}(1+ \gamma \mu^{-1/4}) |f|_{\dot{\mathcal{H}}^2_{\mu}(\R)}.
		\end{equation}
	\end{lemma}
	\begin{remark}
		This expansion of the flux could be used to replace the transmission problem here, where the discontinuous height function is replaced with $h_{\mu, \gamma}$. However, it is not precise enough to recover a uniform bound on the rest term for $\gamma$ of order 1 in $L^{\infty}(\R)$. On the other hand, for $\gamma$ small, it does, and we can even describe it in terms of the physical variables (see the next subsection).
	\end{remark}
	\begin{proof}
		Writing the horizontal flux operator in conformal coordinates gives us the following Fourier representation 
		\begin{align*}
			\mathcal{F}\left((\mathcal{Q}_{\mu} - \partial_{\alpha}) f_0\right)(\xi) 
			& = 
			\sqrt{\mu}\left( \frac{{\rm T}_{\mu}(\xi) - 1}{\sqrt{\mu} |\xi| }\right) |\xi| \widehat{  \partial_{\alpha}f}_0(\xi)
			\\ 
			& \eqcolon \sqrt{\mu}\, {\rm r}_{\mu}(\xi) |\xi|\widehat{  \partial_{\alpha}f}_0(\xi)	
			,
		\end{align*}
		where the symbol ${\rm r}_{\mu}(\xi) = r(\sqrt{\mu}\xi )$ is uniformly bounded and is square-integrable. In particular,  for $p=2$, we use Plancherel's identity together with Lemma \ref{lemma: alpha to G} to find that
		\begin{align*}
			|\mathcal{Q}_{\mu} f - h_{\mu, \gamma}\partial_x f|_{L^{2}(\R)} 
			&  \leq 
			C \sqrt{\mu} |\partial_{\alpha}^2f_0|_{L^2(\R)}
			\\
			& 
			\leq C \sqrt{\mu}(1+ \gamma \mu^{-1/4}) |f|_{\dot{\mathcal{H}}^2_{\mu}(\R)}.
		\end{align*}
		In the case $p = \infty$, we use the definition of the Fourier transform and the Cauchy-Schwarz inequality to see that
		\begin{align*}
			|\mathcal{Q}_{\mu} f - h_{\mu, \gamma}\partial_x f|_{L^{\infty}(\R)}
			&
			\leq C\sqrt{\mu} \left(\int_{\R}  \left(r(\sqrt{\mu}\xi )\right)^2 {\rm d}\xi \right)^{1/2}
			|\partial_{\alpha}^2f_0|_{L^2(\R)}
			\\
			& 
			\leq C \mu^{1/4}(1+ \gamma \mu^{-1/4}) |f|_{\dot{\mathcal{H}}^2_{\mu}(\R)}.
		\end{align*}
	\end{proof}
	This estimate shows competing difficulties: to control the expansion in terms of the domain of the Dirichlet--Neumann operator, we need to pay with $\gamma \mu^{-1/4}$, while the $L^{\infty}(\R)$ gives an additional $\mu^{-1/4}$ loss and therefore requires the size of the step to be small. A natural scaling is to choose $\gamma \lesssim \mu^{1/4}$ and improve the description of the dispersive layer to get a sharper $L^{\infty}-$bound.

	\subsection{Description of dispersive boundary layer in the vanishing step limit} To give an improved asymptotic description close to the step, we will refrain from using the Fourier description of the flux operator. Instead, we will use the kernel formulation of the multiplier ${\rm T}_{\mu}({\rm D})$ to see better how the operator interacts with the conformal change of variable. The  unscaled version is given in Lemma 2.2 in \cite{EhrnstromJohnsonClaassen2019}, (see also entry I.7.37 in \cite{Oberhettinger1990}) and in our case reads  
	\begin{align*}
		\mathcal{T}_{\mu}(\alpha) & \coloneq
		\mathcal{F}^{-1} \left( \tfrac{\tanh\left(\sqrt{\mu}\,  \cdot \right)}{\sqrt{\mu} \, \cdot }\right) (\alpha)
		\\ &
		= 
		\frac{1}{\sqrt{\mu}}\, 
		\mathcal{F}^{-1} \left( \tfrac{\tanh\left( \cdot \right)}{ \cdot }\right) \left(\frac{\alpha}{\sqrt{\mu}}\right)
		\\ &
		= 
		\frac{1}{ \pi \sqrt{\mu} } \log \coth\left( \tfrac{\pi |\alpha| }{4 \sqrt{\mu}}\right),
	\end{align*}
	where the factor $1/\pi$ is the normalization factor ensuring that the kernel has mass 1.  With this formula at hand, we define the dispersive boundary layer in the vanishing-step regime:
	\begin{lemma}\label{Lemma: Bl}
		Let  $\mu, \gamma \in (0, \tfrac{1}{2})$, $\Omega$ and $\Gamma$ be as in Assumption \ref{Def: Assumption 3}. Also, let $p = 2$ or $\infty$, $f \in \mathscr{S}(\R)$, and define the dispersive boundary layer operator by
		\begin{align*}
			\mathcal{B}_{\mu} f (x) 
			= -
			\frac{1}{2}
			\int_{\R} \left(  \frac{(x-y)}{2\sqrt{\mu}} 
			+
			\frac{1}{\pi}
			\log \left(\tfrac{\cosh\left(\frac{\pi}{2 \sqrt{\mu}} x\right)}{\cosh\left(\frac{\pi}{2 \sqrt{\mu}} y \right)}\right)
			\right)
			\frac{ \partial_y f(y)}{ \sqrt{\mu}\sinh\left(\frac{\pi (x-y)}{2 \sqrt{\mu}}\right)} \, {\rm d} y.
		\end{align*}
		Then there holds,
		\begin{equation}\label{Eq: final est bl}
			\left| \mathcal{Q}_{\mu} f  -\left( {\rm T}_{\mu}({\rm D}) \partial_x f + \gamma  \mathcal{B}_{\mu} f \right) \right|_{L^{p}(\R)} \leq \gamma^{2} |\log\gamma|  \, C |\partial_x f|_{L^{p}(\R)}.
		\end{equation}
	\end{lemma}
	\begin{proof}
		We start by reformulating the flux operator in terms of the kernel representation above, where we find that
		\begin{align*}
			\left(\mathcal{Q}_{\mu} f \right)\circ U_0(\alpha)
			& = 
			\frac{1}{\pi \sqrt{\mu}} 
			\int_{\R}  \log \coth\left( \tfrac{\pi |\alpha - \beta| }{4 \sqrt{\mu}}\right)  \partial_{\beta} \left(f \circ U_0(\beta)\right) \, {\rm d}\beta.
		\end{align*}
		Now, set $x = U_0(\alpha)$ and $y = U_0(\beta)$ to find that
		\begin{align}\label{Eq: Q in phys to be aprox}
			\mathcal{Q}_{\mu} f (x)
			& =  
			\int_{\R} {\mathcal{T}_{\mu}} \left( U_0^{-1}(x) - U_0^{-1}(y)\right)  \partial_y  f(y)  \, {\rm d}y.
		\end{align}
		To ease the presentation of the proof, we will consider two steps. First, we will use an asymptotic expansion of $ U_0^{-1}(x) - U_0^{-1}(y)$, provided in Lemma \ref{Lemma: relation between systems}: 
		\begin{align*}
			U_0^{-1}(x) - U_0^{-1}(y) 
			=
			(x-y)  
			+ 
			\gamma {\mathcal{A}}_{\mu}(x,y)
			+ 
			\gamma^{2} |\log\gamma|\, |x-y|r,
		\end{align*}
		for the Lipschitz continuous function $\mathcal{A}_{\mu}$ defined by \eqref{Eq: Def mathA} and $r$ uniformly bounded. Then we expand the full kernel, and we conclude with the estimate on the residual terms that appear from the expansion. \\

		\noindent
		{\bf Step 1.} \textit{Expansion of the kernel $\mathcal{T}_{\mu}$ composed with the inverse mapping.} To make an expansion of the kernel, we define the function
		$$h(x,y) = U_0^{-1}(x) - U_0^{-1}(y) -(x-y),$$ 
		where $|h|\leq  \gamma C|x-y|$. We may therefore Taylor expand $\mathcal{T}_{\mu}$ around $(x-y)$ to find that
		\begin{align*}
			{\mathcal{T}_{\mu}} \left( U_0^{-1}(x) - U_0^{-1}(y)\right)
			& =
			{\mathcal{T}_{\mu}}(x-y) + h {\mathcal{T}_{\mu}}'(x-y) + h^2 \int_{0}^1(1-t) {\mathcal{T}_{\mu}}''\left((x-y) + t h\right) \, {\rm d}t
			\\ 
			&  = 
			{\mathcal{T}_{\mu}}(x-y) + \gamma {\mathcal{A}}_{\mu}(x,y) {\mathcal{T}_{\mu}}'(x-y) 
			+ 
			\gamma^{2} |\log\gamma|
			R(x,y),
		\end{align*}
		where we have 
		\begin{equation*}
			z\, \mathcal{T}_{\mu}'(z)
			=
			-\frac{z}{2\sqrt{\mu}} \frac{1}{
				\sqrt{\mu}
				\sinh\left(
				\frac{\pi z}{2\sqrt{\mu}}
				\right)
			},
		\end{equation*}
		and
		\begin{equation*}
			z^2\, \mathcal{T}_{\mu}''(z)
			=
			\left(\frac{z}{2\sqrt{\mu}}\right)^2\frac{\pi}{\sqrt{\mu}}
			\frac{
				\cosh\left(
				\frac{\pi z}{2\sqrt{\mu}}
				\right)
			}{
				\sinh^2\left(
				\frac{\pi z}{2\sqrt{\mu}}
				\right)
			}.
		\end{equation*}
		While the rest is given by 
		\begin{align*}
			|R(x,y)| \leq C
			\left( |x-y| 
			\left|{\mathcal{T}_{\mu}}'(x-y)  \right|
			+
			|x-y|^2\left| \int_{0}^1(1-t) {\mathcal{T}_{\mu}}''\left((x-y) + t  h\right) \, {\rm d}t\right|
			\right).
		\end{align*}
		To control the last term in the rest, we note that relation \eqref{Eq: Inverse relations} and estimate \eqref{Eq: bound on U_0'} implies that $h$ has the same sign as $(x-y)$ and gives
		\begin{align*}
			|(x-y) + t h| = |x-y| + t |h| \geq |x-y|,
		\end{align*}
		for any $t \in (0,1)$. Since $\mathcal{T}_{\mu}''$ is even and $r \mapsto \mathcal{T}_{\mu}''(r)$ for $r>0$ is decreasing, we find that
		\begin{align*}
			|R(x,y)| \leq C
			\left(|x-y| 
			\left|{\mathcal{T}_{\mu}}'(x-y)  \right|
			+
			|x-y|^2 \left| {\mathcal{T}_{\mu}}''(x-y)\right|
			\right).
		\end{align*}
		This is enough to conclude this step, where we use \eqref{Eq: Q in phys to be aprox} to obtain the following expansion
		\begin{align*}
			\mathcal{Q}_{\mu} f (x)
			=  
			{\mathcal{T}}_{\mu}* \partial_x f(x) 
			+
			\gamma \int_{\R} 
			{\mathcal{A}}_{\mu}(x,y) {\mathcal{T}_{\mu}}'(x-y) \partial_y f(y)\, {\rm d}y
			+
			\gamma^2 |\log\gamma|
			\int_{\R} R(x,y) \partial_y f(y)\, {\rm d}y.
		\end{align*}
		\\
		
		\noindent
		{\bf Step 2.} \textit{Estimate on the rest.} To prove \eqref{Eq: final est bl} for $p = 2$ or $\infty$, we simply need to control the rest term. Making a change of variable gives us the following term to control
		\begin{align*}
			\left|\int_{\R} R(x,y) \partial_y f (y)\, {\rm d}y  \right| 
			\leq & 
			C 
			\int_{\R}
			\left(
			|x-y|
			\left|
			\mathcal{T}_{\mu}'(x-y)
			\right|
			+
			|x-y|^2
			\left|\mathcal{T}_{\mu}''(x-y)\right|
			\right) |\partial_y f (y)|
			\,{\rm d}y
			\\
			= & 
			C 
			\int_{\R}
			\left(
			\frac{|z|}{
				2|\sinh(\pi z/2)|
			}
			+
			\frac{\pi z^2}{4}
			\frac{
				\cosh(\pi z/2)
			}{
				\sinh^2(\pi z/2)
			}
			\right) |(\partial_y f) (x-\sqrt{\mu} z)|
			\,{\rm d}z,
		\end{align*}
		for all $x \in \R$. The kernel is $L^1(\R)$, so applying Hölder's inequality gives the first estimate:
		\begin{align*}
			\left|\int_{\R} R(\cdot,y) \partial_y f (y)\, {\rm d}y  \right|_{L^{\infty}(\R)}
			\leq  
			C|\partial_x f|_{L^{\infty}(\R)}.
		\end{align*}
		For the second estimate, we use the Minkowski integral inequality  to find
		\begin{align*}
			\left|\int_{\R} R(\cdot ,y) \partial_y f (y)\, {\rm d}y  \right|_{L^2(\R)}
			\leq  
			C \left| \partial_xf\right|_{L^2(\R)},
		\end{align*}
		which concludes the proof of this lemma.
	\end{proof} 
	Here we only required $\partial_x f \in L^{\infty}(\R) \cap L^2(\R)$, but we know that it is sufficient to establish a bound in terms of $\partial_x f \in H^1(\R)$ to close the estimate uniformly in the small parameters for $f \in \dot{\mathcal{H}}_{\mu}^2(\R)$ in the regime $\gamma \lesssim \mu^{1/4}$. Using this additional regularity allows us to simplify the boundary layer operator further:
	\begin{lemma} \label{Lemma: simplified disp bl}
		Let  $\mu, \gamma \in (0, \tfrac{1}{2})$, $s \in (0,1]$, $\Omega$ and $\Gamma$ be as in Assumption \ref{Def: Assumption 3}.  Also, let $f \in \mathscr{S}(\R)$ and define the simplified dispersive boundary layer function by
		\begin{align*}
			\mathcal{C}_{\mu}  (x) 
			= 
			-\frac{2}{\pi} \arctan \left( \exp\left(  \tfrac{\pi x}{2\sqrt{\mu}}\right)\right).
		\end{align*}
		Then there holds,
		\begin{equation*}
			\left| \mathcal{Q}_{\mu} f  -\left( {\rm T}_{\mu}({\rm D}) \partial_x f + \gamma \, \mathcal{C}_{\mu} \partial_x f \right) \right|_{L^{\infty}(\R)} \leq  (\gamma^{2} |\log\gamma| +\gamma\mu^{1/4}) C | \partial_x f|_{H^1(\R)},
		\end{equation*}
		and
		\begin{equation*}
			\left| \mathcal{Q}_{\mu} f  -\left( {\rm T}_{\mu}({\rm D}) \partial_x f + \gamma \, \mathcal{C}_{\mu} \partial_x f \right) \right|_{L^{2}(\R)} \leq  (\gamma^{2} |\log\gamma| +\gamma\mu^{s/2}) C | \partial_x f|_{H^{s}(\R)}.
		\end{equation*}
		
	\end{lemma}
	
	\begin{proof}
		Starting from  the expansion derived in Lemma \ref{Lemma: Bl}, it is enough to approximate 
		\begin{align*}
			\mathcal{B}_{\mu} f (x) 
			= 
			-
			\frac{1}{2}
			\int_{\R} \left(  \frac{(x-y)}{2\sqrt{\mu}} 
			+
			\frac{1}{\pi}
			\log \left(\tfrac{\cosh\left(\frac{\pi}{2 \sqrt{\mu}} x\right)}{\cosh\left(\frac{\pi}{2 \sqrt{\mu}} y \right)}\right)
			\right)
			\frac{ \partial_y f(y)}{ \sqrt{\mu}\sinh\left(\frac{\pi (x-y)}{2 \sqrt{\mu}}\right)} \, {\rm d} y,
		\end{align*}
		up to a rest of order $\mu^{1/4}$ and  $\mu^{s/2}$. To do so, we introduce the notation 
		\begin{equation*} 
			B_{\mu}(x)
			=
			\frac{x}{2\sqrt{\mu}}
			+
			\frac{1}{\pi}
			\log\cosh\left(\frac{\pi x}{2\sqrt{\mu}}\right).
		\end{equation*}
		Then making a change of variables gives
		\begin{align*}
			\mathcal{B}_{\mu} f (x) 
			= 
			\frac{1}{2}\Phi_{\mu} (x) \, \partial_x f(x)
			+
			R(x),
		\end{align*}
		where
		\begin{equation*}
			\Phi_{\mu} (x) 
			=
			-
			\int_{\R}
			\frac{B_{\mu}(x)-B_{\mu}(x- \sqrt{\mu}z)}{\sinh(\pi z/2)}
			\,{\rm d}z,
		\end{equation*}
		and
		\begin{align}\label{Eq: def Rest for Bmu}
			R(x)
			= 
			-
			\frac{1}{2}
			\int_{\R}
			\frac{
				B_{\mu}(x)-B_{\mu}(x- \sqrt{\mu} z)
			}{
				\sinh(\pi z/2)
			}
			\left[
			\partial_xf\left(x-\sqrt{\mu}\,z\right)-\partial_xf(x)
			\right]
			\,{\rm d}z.
		\end{align}
		We first compute the profile $\Phi_{\mu}$, where we first note that
		\begin{align*}
			B_{\mu}'(x)-B_{\mu}'(x- \sqrt{\mu} z)
			= &
			\frac{1}{2\sqrt{\mu}}
			\left(
			\tanh\left(\frac{\pi x}{2 \sqrt{\mu}}\right)
			-
			\tanh\left(\frac{\pi(x- \sqrt{\mu}z)}{2 \sqrt{\mu}}\right)
			\right)
			\\
			= &
			\frac{1}{2 \sqrt{\mu}}
			\frac{
				\sinh(\pi z/2)
			}{
				\cosh\left(\frac{\pi(x- \sqrt{\mu} z)}{2 \sqrt{\mu}}\right)
				\cosh\left(\frac{\pi x}{2 \sqrt{\mu}}\right)
			}.
		\end{align*}
		In other words, there holds,
		\begin{align*}
			\Phi'_{\mu}(x)
			=
			-
			\frac{1}{2 \sqrt{\mu}}\frac{1}{\cosh\left(\frac{\pi x}{2 \sqrt{\mu}}\right)}
			\int_{\R}
			\frac{1}{ \cosh\left(\frac{\pi(x-\sqrt{\mu}z)}{2\sqrt{\mu}}\right) }
			\,{\rm d}z
			=
			-\frac{1}{\sqrt{\mu}}
			\frac{1}{\cosh\left(\frac{\pi x}{2\sqrt{\mu}}\right)},
		\end{align*}
		which implies
		\begin{align*}
			\Phi_{\mu} (x)
			= 
			-
			\frac{4}{\pi}
			\arctan\left(
			\exp\left(\tfrac{\pi x}{2\sqrt{\mu}}\right)
			\right).
		\end{align*}
		This gives the correct decomposition of the boundary layer operator, and it only remains to estimate the remainder.  For the $L^{\infty}-$estimate, we use the Fundamental Theorem of Calculus and the Cauchy-Schwarz inequality to find that for all $x \in \R$ there holds,
		\begin{align*}
			|R(x)| 
			&  \leq C \left|\int_{\R} \frac{z}{\sinh\left( \frac{\pi z}{2}\right)} \left( \int_{x}^{x- \sqrt{\mu} z} \partial_s^2f(s) \,{\rm d} s  \right) {\rm d}z\right|
			\\
			& \leq 
			\mu^{1/4} C |\partial_x^2 f|_{L^2(\R)}.
		\end{align*}
		For the $L^2-$estimate, we use Plancherel's identity to first note that for all $s \in [0,1]$, $z \in \R$ one has,
		\begin{align*}
			\left|\partial_x f(\cdot - \sqrt{\mu}z) - \partial_x f(\cdot)\right|_{L^2(\R)}^2 
			& = \int_{\R} \left| e^{-i\sqrt{\mu}z \xi}- 1\right|^2 \, \left|\widehat{\partial_xf}(\xi)\right|^2 \, {\rm d}\xi 
			\\ 
			& 
			\lesssim \mu^s |z|^{2s} \int_{\R} \left| |\xi|^s\widehat{\partial_x f}(\xi)\right|^2 \, {\rm d}\xi,
		\end{align*}
		and implies 
		\begin{align*}
			|R|_{L^{2}(\R)} 
			& \leq 
			\mu^{s/2} C |\partial_x f|_{H^{s}(\R)}.
		\end{align*} 
	\end{proof}
	
	From this result, we may simplify the linear water waves equations so that they are independent of the conformal mapping and with a residual error that is bounded and vanishes in the small step limit. We have the following consistency result:
	\begin{prop}\label{Prop: ww + bl}
		Let  $\gamma, \mu \in (0,\tfrac{1}{2})$ such that $\gamma \lesssim \mu^{1/4}$, $\Omega$ and $\Gamma$ be as in Assumption \ref{Def: Assumption 3}. Then for any $(\zeta^{\mathrm{in}}, \psi^{\mathrm{in}}) \in \mathcal{H}^2_{\mu}(\R) \times \dot{\mathcal{H}}^{5/2}_{\mu}(\R)$ and $T>0$ there is a solution $(\zeta, \psi) \in C\left( [0, T];  \mathcal{H}^2_{\mu} \times \dot{\mathcal{H}}^{5/2}_{\mu}(\R)\right)$  to the linear water waves equations \eqref{Eq: Linear ww} and that also solves
		\begin{equation*}
			\left\{\begin{aligned}
				\partial_t \zeta + \partial_x q & = 0 
				\\ 
				\partial_t q +  \left({\rm T}_{\mu}({\rm D}) + \gamma \mathcal{C}_{\mu}\right)\partial_x \zeta & = 
				(\gamma^{2}  |\log\gamma| + \gamma\mu^{1/4}) R,
			\end{aligned}\right.
		\end{equation*}
		where $R$ satisfies
		\begin{align*}
			\sup\limits_{t  \in [0,T]} \left(
			|R(t)|_{L^{\infty}(\R)} + |R(t)|_{L^{2}(\R)} \right) \leq C \left( |\zeta^{\rm in}|_{\mathcal{H}_{\mu}^2(\R)} +  |\psi^{\rm in}|_{\dot{\mathcal{H}}_{\mu}^{5/2}(\R)}\right).
		\end{align*}
	\end{prop} 
	\begin{remark}\label{Remark: on the dispersive case for a small case}
		The result can also be stated for data in a Sobolev space if we further suppose a vanishing condition around $x=0$.  Moreover, this model recovers the water waves equations when the step vanishes and offers a better description for $\gamma$ small depending on $\mu \in (0,1)$. In particular, expanding the Fourier multiplier in frequency ${\rm T}_{\mu}(\xi) = 1 - \tfrac{\mu}{3}{\xi}^2 + \mathcal{O}(\mu^2\xi^4)$, we obtain a new version of the Boussinesq equation, or the shallow water equations. The full justification and study of these models are left for future work.
	\end{remark}
	\begin{proof}
		We argue as in the proof of Theorem \ref{Thm: Cv Step case}, where we use the solution of \eqref{Eq: Linear ww} to identify the rest term by 
		\begin{align*}
			\partial_tq 
			& = - \mathcal{Q}_{\mu} \zeta
			\\
			& \eqcolon 
			- \left( {\rm T}_{\mu}({\rm D}) \partial_x \zeta + \gamma \, \mathcal{C}_{\mu} \partial_x  \zeta \right)+ \tilde R 
		\end{align*}
		where we use Lemma \ref{Lemma: simplified disp bl}, \ref{lemma: alpha to G}, and a change of variables to find that
		\begin{align*}
			|\tilde R|_{L^{\infty}(\R)} + |\tilde R|_{L^{2}(\R)} 
			& \leq 
			 (\gamma^{2}  |\log\gamma| + \gamma\mu^{1/4})  C |\zeta|_{\mathcal{H}_{\mu}^2(\R)}.
		\end{align*}
		The conclusion now follows from Theorem \ref{thm: Wp ww}.
	\end{proof}

	\section*{Acknowledgment}
	
	\noindent
	This research was funded by the project Climath of the PEPR Math-VivEs, ANR-23-EXMA0003. I would also like to warmly thank David Lannes for many important conversations regarding this work during a memorable two-year stay at Bordeaux, and Didier Pilod for insightful discussions on  elliptic theory. Lastly, I want to thank Emmanuel Dormy, Christophe Lacave, Ludivine Oruba, and Alexis Vasseur for the invitation to Institut Henri Poincaré where part of this research took place and I am thankful for the support of the Institut Henri Poincaré (UAR 839 CNRS-Sorbonne Université), and LabEx CARMIN (ANR-10-LABX-59-01).  
	 
	\appendix
	\section{Properties on the conformal map \eqref{Eq: Conf. map Step}}\label{App: prop conf map}
	In this Section we recall the derivation of the conformal map for a step. This is classical, but is given here to carefully track the dependence in the shallow-water parameter and obtain formulas adapted to the present setting. Then we give some related estimates and identities that will be used in the main proofs of the paper.

	\subsection{Proof of Lemma \ref{lemma: conf map}}
		Here we first find the transformation from the step to the half-plane, then we transform it to the strip, and conclude by computing its Jacobian. \\ 

		\noindent
		{\bf Step 1.} \textit{From the half-plane to the step}. We let $\mathbb{H} =  \{(\alpha, \beta) \in \mathbb{C} \: : \: \beta >0 \}$ and define $H : \mathbb{H} \rightarrow \Omega_{\mu}$ using the Schwarz-Christoffel formula \cite{DriscollTrefethen02}:
		\begin{equation}\label{Eq: Conf step}
			H(\eta) -  H(\eta_0) = c_0 \int_{\eta_0}^{\eta} \frac{(\rho - a)^{1/2}}{\rho (\rho - 1)^{1/2}}\: \mathrm{d}\rho,
		\end{equation}
		for $c_0 \in \mathbb{C}$, $a >1$ to be determined.   The numerator $(\rho - a)^{1/2}$ reflects the corner with angle $3\pi/2$, $\rho^{-1}$ reflects the two points at infinity, and $(\rho-1)^{-1/2}$ models the corner with angle $\pi/2$.  
		\begin{figure}[h]
			\centering 
			\begin{minipage}{.45\textwidth}
				\centering 
				\begin{tikzpicture}
					\begin{axis}[
						x=1cm, y=1.0cm,
						axis lines=middle,
						xlabel={$x$}, ylabel={$y$},
						x label style={anchor=north},
						y label style={anchor=east},
						xmin=-2.5, xmax=3.3, ymin=0, ymax=1,
						ticks=none,
						clip=false
						] 
						\fill[aqua!20] (axis cs:-2.5,0) -- (axis cs:3,0) -- (axis cs:3,-1.5) -- (axis cs:0,-1.5) -- (axis cs:0,-3) -- (axis cs:-2.5,-3) -- cycle;
						\fill[gray!100, nearly transparent] (axis cs:-2.5,-3) -- (axis cs:0,-3) -- (axis cs:0,-1.5) -- (axis cs:3,-1.5) -- (axis cs:3,-3.5) -- (axis cs:-2.5,-3.5) -- cycle;
						
						\draw[brown, very thick] (axis cs:-2.5,-3) -- (axis cs:0,-3);
						\draw[brown, very thick] (axis cs:0,-1.5) -- (axis cs:3,-1.5);
						\draw[red, very thick] (axis cs:0,-3) -- (axis cs:0,-1.5);
						\draw[blue, very thick] (axis cs:-2.5,0) -- (axis cs:3,0);
						
						\node[anchor=south east] at (axis cs:-2.5, -0.3) {\small $y=0$};
						\node[anchor=east] at (axis cs:-2.5, -3) {\small $y=-\sqrt{\mu}$};
						\node[anchor=west] at (axis cs:-1, -1.2) {\small $y=\sqrt{\mu}(-1+\gamma)$};
						\node[anchor=west] at (axis cs:-0.6, -3.2) {\small $x=0$};
					\end{axis}
				\end{tikzpicture}
			\end{minipage}%
			\hspace{0.08cm}
			\hfill 
			\begin{minipage}{.08\textwidth}
				\centering
				\begin{tikzpicture}
					\node[label=above:{\hspace{-0.15cm} $H = h_1 + ih_2$}] at (0.,-1.2) {}; 
					\draw[<-, >=stealth, thick] (-0.4,-1.2) to [out=30, in=150] (0.4,-1.2);
				\end{tikzpicture}
			\end{minipage}%
			\hfill 
			\begin{minipage}{.45\textwidth}
				\centering
				\begin{tikzpicture}
					\begin{axis}[
						x=1.0cm, y=1.0cm, 
						axis x line=bottom,     
						axis y line=middle,    
						xlabel={$\alpha$}, ylabel={$\beta$},
						x label style={at={(axis description cs:1,0)}, anchor=north},
						y label style={anchor=east},
						xmin=-2.5, xmax=3.3, 
						ymin=-2.8, ymax=1,     
						ticks=none,
						clip=false
						] 
						\fill[aqua!20] (axis cs:-2.5,0) rectangle (axis cs:3,-2.8);
						\fill[gray!100, nearly transparent] (axis cs:-2.5,-2.8) rectangle (axis cs:3,-3.5);
						
						\draw[blue, very thick] (axis cs:-2.5,-2.8) -- (axis cs:-1.2,-2.8);
						\draw[brown, very thick] (axis cs:-1.2,-2.8) -- (axis cs:0,-2.8); 
						\draw[brown, very thick] (axis cs:1.2,-2.8) -- (axis cs:3,-2.8);

						\draw[red, very thick] (axis cs:0,-2.8) -- (axis cs:1.2,-2.8);
						\node[anchor=north, red] at (axis cs:-1.2, -2.8) {\small $0$};
						\node[anchor=north, red] at (axis cs:0, -2.8) {\small $1$};
						\node[anchor=north, red] at (axis cs:1.2, -2.85) {\small $a$};			
						
						\node[anchor=east] at (axis cs:-2.5, -2.8) {\small $\beta=0$};
						
					\end{axis}
				\end{tikzpicture}
				
			\end{minipage}
			\caption{A conformal transformation $H$, as defined in equation \eqref{Eq: Conf step}, from the upper-half space to the step. The blue line maps to the surface, the brown lines map to the steps, and the red line illustrates the height of the step.} 
			\label{Fig: Conf half to step}
		\end{figure}
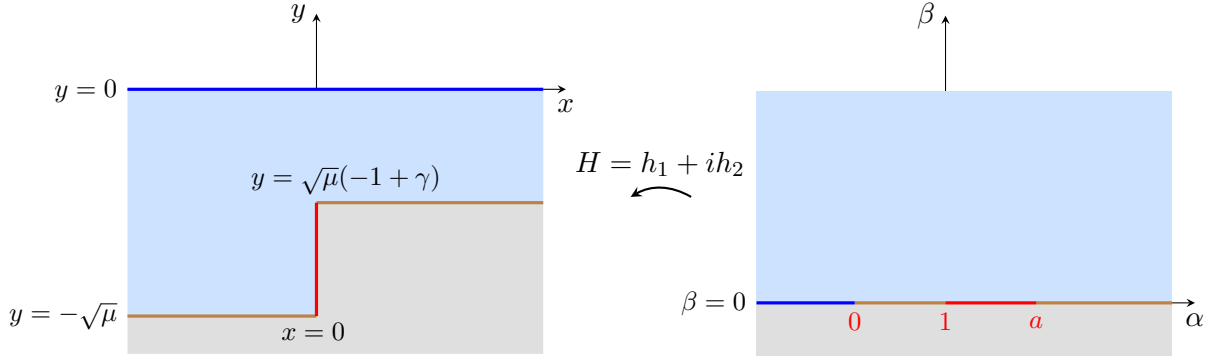
		With this in mind, we first consider $\eta \in (-\infty,0)$ and $\eta \in (0,1)$  that is connected by the path along a small ball of size $r>0$ around $\zeta = 0$ denoted $\partial B_{r}(0)\cap \mathbb{C}^+$, and polar coordinates $\rho = r e^{i\theta}$ for $\theta \in (0,\pi)$ to find that
		\begin{align*}
			H(\eta) -  H(\eta_0) 
			& =
			c_0 \int_{\partial B_{r}(0)\cap \mathbb{C}^+}\frac{(\rho - a)^{1/2}}{\rho (\rho- 1)^{1/2}} \: \mathrm{d}\rho 
			\\ 
			& = 
			i\, c_0 \int_{0}^{\pi}\frac{(re^{i \theta}  - a)^{1/2}}{ (r e^{i \theta} - 1)^{1/2}}    \: \mathrm{d}\theta.
		\end{align*}
		We note that for  $\eta \in (-\infty,0)$ we have ${\mathfrak{Im}} H(\eta) = 0$. While ${\mathfrak{Im}} H(\eta_0) = - \sqrt{\mu} $ when $\eta \in (0, 1)$.  Therefore, we find that as $r \rightarrow 0$ that there holds,
		\begin{equation*}
			\sqrt{a} = \frac{\sqrt{\mu}}{c_0\pi}.
		\end{equation*}
		Similarly, we can choose a large path $R>0$ along the ball $\partial B_R(0)\cap \mathbb{C}^+$ to find that
		\begin{align*}
			H(\eta) -  H(\eta_0) 
			& = 
			i\,  c_0 \lim \limits_{R \rightarrow \infty } \int_{0}^{\pi}\frac{(Re^{i \theta}  - a)^{1/2}}{ (R e^{i \theta} - 1)^{1/2}}    \: \mathrm{d}\theta,
		\end{align*}
		where ${\mathfrak{Im}} H(\eta_0) =\sqrt{\mu}(-1+\gamma)$ for $\eta_0 \in (1,\infty)$. This  gives us the following relations
		\begin{equation*}
			c_0 \pi = \sqrt{\mu} (1- \gamma)
			\qquad \text{and} \qquad 
			a = (1- \gamma)^{-2}.
		\end{equation*}
		We may now use that the integral determining $H$ is given by
		\begin{align*}
			H(\eta) 
			=
			\frac{2\sqrt{\mu}(1-\gamma)}{\pi}  
			\left(
			\tanh^{-1} \left(  W(\eta) \right) 
			- 
			\sqrt{a} \tanh^{-1} \left( \sqrt{a}\, W(\eta) \right) \right) 
			+
			{\rm cst.},
		\end{align*}
		with  ${\rm cst}. \in \mathbb{C}$ and
		\begin{equation*}
			W(\eta) = \left( \frac{\eta - 1}{\eta - a} \right)^{1/2}.
		\end{equation*}
		Now, since $H(1) = - i\, \sqrt{\mu}$ and $W(1) = 0$ we find that ${\rm cst.} = - i\, \sqrt{\mu}$ and the formula for $H$ reads,
		\begin{align}\label{Eq: for H}
			H(\eta) 
			=
			\frac{2\sqrt{\mu}(1-\gamma)}{\pi}  
			\left(
			\tanh^{-1} \left(  W(\eta) \right) 
			- 
			\sqrt{a}\tanh^{-1} \left( \sqrt{a} \, W(\eta) \right) \right) 
			-
			i\, \sqrt{\mu}.
		\end{align}

		\noindent
		{\bf Step 2.} \textit{From the strip to the step}. To conclude, we will  compose the conformal map for $H$ with the exponential map $f : \rm{S}_{\mu} \rightarrow  \mathbb{H}$ defined by
		$$ f(\zeta) 
		= 
		-
		\exp\left( \frac{\pi}{\sqrt{\mu}}\zeta\right),$$
		where the height of the step corresponds to $\mathfrak{Re}(\zeta) \in (0, \delta )$ for some $\delta>0$ to be determined (see Figure \ref{Fig: Conf Step}). The final transformation is then defined by 
		$$F  = H \circ f.$$	
		Here we note that the boundary line $\beta=0$ is mapped to
		$(-\infty,0)$ by the exponential map, and then to the physical boundary $y=0$. Moreover, we have that the left bottom boundary corresponds to $f(\alpha - i \sqrt{\mu})  \in (0, 1)$ for all $\alpha \in (-\infty,0)$, and the vertical step corresponds to $f(\alpha - i \sqrt{\mu})  \in (1, a)$ for all $\alpha \in (0, \delta)$ where
		\begin{equation*}
			a = e^{\frac{\pi}{\sqrt{\mu}} \delta} 
			\iff 
			\delta = -  \frac{2 \sqrt{\mu}}{\pi} \ln \left(1-\gamma\right),
		\end{equation*}
		and concludes the construction of \eqref{Eq: Conf. map Step}. \\ 
		
		\noindent
		{\bf Step 3.} \textit{Formula for the Jacobian.} First note that by construction there holds,
		\begin{equation*}
			H'(\eta)
			=
			\frac{\sqrt{\mu}(1-\gamma)}{\pi}
			\frac{1}{\eta}
			\left(
			\frac{\eta - a}{\eta - 1}
			\right)^{1/2},
		\end{equation*}
		and the chain rule gives
		\begin{align*}
			F'(\zeta)
			& =
			H'(f(\zeta))f'(\zeta)
			\\
			& =
			(1-\gamma)
			\left(
			\frac{
				e^{\frac{\pi}{\sqrt{\mu}}\zeta}
				+
				e^{\frac{\pi}{\sqrt{\mu}}\delta}
			}{
				e^{\frac{\pi}{\sqrt{\mu}}\zeta}
				+
				1
			}
			\right)^{1/2},
		\end{align*}
		where $\zeta = \alpha + i \beta$. Using  elementary trigonometric identities, we find that
		\begin{equation*}
			\frac{
				e^{\frac{\pi}{\sqrt{\mu}}\zeta}+ e^{\frac{\pi}{\sqrt{\mu}}\delta}
			}{
				e^{\frac{\pi}{\sqrt{\mu}}\zeta}+1
			}
			=
			e^{\frac{\pi}{\sqrt{\mu}}\delta}
			\frac{
				\cosh\left(\frac{\pi(\zeta-\delta)}{2\sqrt{\mu}}\right)
			}{
				\cosh\left(\frac{\pi\zeta }{2\sqrt{\mu}}\right)
			}.
		\end{equation*}
		Using again the relation $	e^{\frac{\pi}{\sqrt{\mu}}\delta}=	(1-\gamma)^{-1}$ one finds that the determinant of the Jacobian associated to the conformal map is given by formula \eqref{Eq: det Jac conf map}.
		\begin{proof}[\color{white}{box}]
		\end{proof}

		\subsection{Estimates related to $F'$ in the strip} 
		The first result concerns the asymptotics around the corner of the map from the original fluid domain to the unit strip. In what follows, we use the standard identification $\mathbb{C}^2 \simeq \R^2$. In particular, for $F: {\rm S}_{\mu} \rightarrow \Omega_{\mu}$ defined by \eqref{Eq: Conf. map Step} and introduce the scaling $\sigma(\alpha + i \beta) = 	\alpha + i \sqrt{\mu} \, \beta$ for $\alpha + i \beta \in {\rm S}$. Then we define $F_{\mu} : {\rm S} \rightarrow \Omega$ by
		\begin{equation}\label{Eq: Fmu}
			F_{\mu} = \sigma^{-1} \circ F \circ \sigma.
		\end{equation}
		Then the corner point reads ${\rm p}_{\rm c} = i(-1+\gamma)$, which in conformal coordinates corresponds to $F_{\mu}(\zeta_{\rm c}) = F_{\mu}(\delta - i)$. With this notation in mind, have the following result on the relation between the two coordinate systems around the corner.
		\begin{lemma}\label{Lemma: est on inverse}
			Let $\mu, \gamma \in (0, \tfrac{1}{2})$, ${\rm p} = x+ i y \in \Omega$, and define $H({\rm p}) = F^{-1}_{\mu}({\rm p}) - F^{-1}_{\mu}({\rm p}_{\rm c})$. Then if $|\sigma({\rm p} - {\rm p}_{\rm c})| \leq c_0 \gamma \sqrt{\mu}$ for some $c_0 >0$ there holds,
			\begin{equation*}
				\sigma \circ H({\rm p})
				= 
				\kappa
				\left( \sigma({\rm p}-{\rm p}_{\rm c}) \right)^{2/3}
				+
				\kappa^2(\sigma({\rm p}-{\rm p}_{\rm c}))^{4/3}\,
				R,
			\end{equation*}
			where the coefficient is defined by $
				\kappa
				=
				\left(
				\tfrac{
					9\sqrt{\mu}\,\gamma(2-\gamma) }{
					4\pi(1-\gamma)^2
				}
				\right)^{1/3}$ and the rest satifies
			\begin{equation*}
				\sup_{|\sigma (\zeta - \zeta_{\rm c})| < c_0 \gamma \sqrt{\mu}   }\left( 
				\left|R(\zeta)\right|
				+
				\left| R'(\zeta)\right|
				+
				\left| R''(\zeta)\right|
				\right)
				\lesssim \gamma^{-1}.
			\end{equation*}
		\end{lemma}
		\begin{proof}
			We first use the definition $\sigma({\rm p}-{\rm p}_{\rm c}) = F \circ \sigma(\zeta) - F \circ \sigma(\zeta_{\rm c})$ and the Fundamental Theorem of Calculus to find that
			\begin{equation}\label{Eq: FTC F mu}
				F \circ \sigma(\zeta) - F \circ \sigma(\zeta_{\rm c}) 
				= \sigma(\zeta - \zeta_{\rm c})
				\int_0^1 
				F'
				\left(  \sigma \left(\zeta_{\rm c} +t(\zeta - \zeta_{\rm c})\right)\right)\, {\rm d}t.
			\end{equation} 
			The aim is to simplify the integrand and solve for $\zeta - \zeta_{\rm c}$ in terms of the physical variables.  With this in mind, we use identity \eqref{Eq: det Jac conf map} for the conformal derivative to see that
			\begin{align*}
				\left[F'
				\left(  \sigma \left(\zeta_{\rm c} +t(\zeta - \zeta_{\rm c})\right)\right)\right]^2
				=
				(1-\gamma)
				\frac{
					\cosh\left(
					-\frac{i\pi}{2}
					+
					\frac{\pi t}{2\sqrt{\mu}} \sigma (\zeta - \zeta_{\rm c})
					\right)
				}{
					\cosh\left(
					-\frac{i\pi}{2}
					+
					\frac{\pi \delta}{2 \sqrt{\mu}}
					+
					\frac{\pi t}{2\sqrt{\mu}}\sigma (\zeta - \zeta_{\rm c})
					\right)
				}.
			\end{align*}
			Now, to simplify the expression further we let $z \in \mathbb{C}$, $r \in (0,1)$, and recall the following relations on hyperbolic functions:
			\begin{equation*}
				\cosh\left(z-\frac{i\pi}{2}\right)
				=
				-i\sinh(z),
			\end{equation*}
			and
			\begin{align*}
				\sinh\left(z-\log(r)\right)
				&=
				\sinh(z)\cosh\left(\log(r)\right)
				-
				\cosh(z)\sinh\left(\log(r)\right)
				\\
				&=
				\frac{1}{2r}
				\left(
				(1-r^2)\cosh(z)
				+
				(1+r^2)\sinh(z)
				\right),
			\end{align*}
			to see that 
			\begin{align*}
				\frac{\cosh\left(z-\frac{i\pi}{2}\right)}{\cosh\left(z-\log(r)-\frac{i\pi}{2}\right) }
				& = 
				\frac{2r}{1-r^2}
				\frac{\tanh(z)}{1+\frac{1+r^2}{1-r^2}\tanh(z)}
				\\
				& = 
				\frac{2r}{1-r^2}z( 1 + z r(z)),
			\end{align*}
			where we use a Taylor expansion of the hyperbolic tangent  for $|z|\lesssim 1-r^2$  to find that 
			\begin{equation*}
				\sup_{|z| < c_0(1-r^2) }\left( 
				\left|r(z)\right|
				+
				\left| r'(z)\right|
				+
				\left| r''(z)\right|
				\right)
				\lesssim \frac{1}{1-r^2},
			\end{equation*}
			for some $c_0>0$ small enough. In particular, if we let  $r=1-\gamma$ and $z=\frac{\pi t}{2\sqrt{\mu}}\,  \sigma(\zeta - \zeta_{\rm c})$ such that $|\sigma(\zeta - \zeta_{\rm c})| \lesssim \gamma \sqrt{\mu}$ then we may expand the square root to find the relation
			\begin{align*}  
				F'
				\left(  \sigma \left(\zeta_{\rm c} +t(\zeta - \zeta_{\rm c})\right)\right)
				=
				A\,
				\left(t\sigma (\zeta - \zeta_{\rm c})\right)^{1/2} \left( 1
				+ 
				\sigma(\zeta - \zeta_{\rm c}) R\left(t\zeta\right) \right),
			\end{align*}
			with $A =  \left(\frac{\pi(1-\gamma)^2}{\sqrt{\mu}\gamma(2-\gamma)} \right)^{1/2}$ and we ease notation letting $R$ be any function satisfying
			\begin{equation}\label{Eq: est R}
				\sup_{|\sigma (\zeta - \zeta_{\rm c})| < c_0 \gamma \sqrt{\mu}   }\left( 
				\left|R(\zeta)\right|
				+
				\left| R'(\zeta)\right|
				+
				\left| R''(\zeta)\right|
				\right)
				\lesssim \gamma^{-1}.
			\end{equation}
			Substituting this expansion
			into \eqref{Eq: FTC F mu} gives
			\begin{align*}
				\sigma({\rm p} - {\rm p}_{\rm c})
				& =  
				\frac{2}{3} A 
				\left( \sigma( \zeta - \zeta_{\rm c})\right)^{3/2}
				\left( 
				1
				+
				\sigma( \zeta - \zeta_{\rm c})\,
				 R \right).
			\end{align*}
			We may now solve for $\sigma( \zeta - \zeta_{\rm c} )$ where we again use a Taylor expansion to see that
			\begin{align}\label{Eq: to be iterated}
				\sigma( \zeta - \zeta_{\rm c}) 
				= 
				\left(\frac{3}{2} \frac{1}{A}\right)^{2/3}
				\left( \sigma({\rm p}-{\rm p}_{\rm c}) \right)^{2/3}
				+
				(\sigma( \zeta - \zeta_{\rm c} ))^{2}\,
				R.
			\end{align} 
			The first coefficient agrees with $\kappa$ in the statement of the Lemma. Moreover, we note that $|\sigma(\zeta - \zeta_{\rm c})| \lesssim \gamma \sqrt{\mu}$ implies $|\sigma({\rm p}-{\rm p}_{\rm c})| \lesssim \gamma \sqrt{\mu}$ and we may use \eqref{Eq: est R} to iterate \eqref{Eq: to be iterated}, where we find  
			\begin{equation*}
				\sigma( \zeta - \zeta_{\rm c}) 
				= 
				\kappa
				\left( \sigma({\rm p}-{\rm p}_{\rm c}) \right)^{2/3}
				+
				\kappa^2(\sigma({\rm p}-{\rm p}_{\rm c}))^{4/3}\,
				R.
			\end{equation*}

		\end{proof}

		The next result concerns the singularity that appears in the conformal  change of variable for an $\dot{H}^2(\Omega)$ function at 
		$$\zeta_{\rm c} = \delta -i \qquad \text{and} \qquad  \zeta_0 = -i.$$
		We recall the definition of the determinant of the Jacobian by $J = |F'(\zeta)|^2$ from Remark \ref{Remark: change of variables}, and give the following pointwise estimate:
		
		\begin{lemma}\label{Lemma: est on change of variables}
			Let $\mu, \gamma \in (0, \tfrac{1}{2})$ and $\mathcal{R} \in C^{\infty}(\Omega)$ be defined by $\mathcal{R} = R \circ F^{-1}_{\mu}$, where $R : {\rm S} \rightarrow \R$ satisfies
			\begin{equation*}
				|\nabla^{\mu} R(\zeta )| \leq \frac{c_1}{\sqrt{\mu}} |\sigma(\zeta - \zeta_{\rm c})| \qquad \text{and} \qquad 	|(\nabla^{\mu})^2 R(\zeta )| \leq c_2,
			\end{equation*}
			for some $c_1, c_2 >0$. Then for any $\zeta \in {\rm S}\backslash \{\zeta_0, \zeta_{\rm c}\}$ there holds,
			\begin{align*}
				\left|
				\left((\nabla^{\mu})^2
				\mathcal{R}  \right)\circ F_\mu(\zeta)
				\right|^2J\circ \sigma
				\leq 
				C \left(\frac{\gamma^2\mu c_1^2}{|\sigma(\zeta - \zeta_0)| \, |\sigma(\zeta - \zeta_{\rm c})|} + c_2^2\left( 1 + \frac{\sqrt{\mu}}{|\sigma (\zeta - \zeta_{\rm c})|} \right) \right).
			\end{align*}
		\end{lemma}
	
		\begin{proof}  First, we use the relations in Remark \ref{Remark: change of variables} to see that 
			\begin{align*}
				\mu^{j/2}\left(\partial_x^j \mathcal{R} \right) \circ F_{\mu}  
				& = 
				\left(
				\frac{U_{\alpha}\circ \sigma}{J\circ \sigma} \sqrt{\mu}\partial_{\alpha} 
				+
				\frac{ U_{\beta}\circ \sigma}{J\circ \sigma} \partial_{\beta}
				\right)^j R
				\\ 
				\left(\partial_y^j \mathcal{R} \right) \circ F_{\mu}  
				& = 
				\left(
				\frac{V_{\alpha}\circ \sigma}{J\circ \sigma} \sqrt{\mu}\partial_{\alpha} + \frac{V_{\beta}\circ \sigma}{J\circ \sigma} \partial_{\beta}
				\right)^j R,
			\end{align*}
			and  
			\begin{align*} 
				 \left|\nabla^{\mu}_{\alpha, \beta}
				\left(\frac{1}{J \circ \sigma }\right)\right| 
				=2
				\sqrt{\mu} \left|
				\frac{F''\circ \sigma}{(J\circ \sigma)^{3/2}}\right|.
			\end{align*}
			Consequently, we obtain the pointwise estimate
			\begin{align*}
				J\circ \sigma
				\left|
				\left((\nabla^{\mu})^2
				\mathcal{R}  \right)\circ F_\mu
				\right|^2
				\lesssim 
				&
				\frac{1}{(J\circ\sigma)}
				\left|
				(\nabla_{\alpha,\beta}^{\mu})^2
				R
				\right|^2 
				+
				\mu\,
				\frac{|F''\circ \sigma|^2}{(J\circ \sigma)^{2}}
				\left|
				\nabla_{\alpha,\beta}^{\mu}
				R
				\right|^2.
			\end{align*}
		 	For the first term, we let $\alpha + i\beta  \in \mathbb{C}$ and use the identity $|\cosh(\alpha +i \beta )|^2 = \sinh^2(\alpha) + \cos^2(\beta)$ to see that
		 	\begin{align*}	
		 		\frac{1}{(J\circ\sigma)}
		 		&=
		 		(1-\gamma)^{-1}
		 		\left(
		 		\tfrac
		 		{ \sinh^2\left(\frac{\pi\alpha}{2\sqrt{\mu}}\right)
		 			+
		 			\sin^2\left(\frac{\pi(\beta+1)}{2}\right)
		 		}
		 		{
		 			\sinh^2\left(\frac{\pi(\alpha-\delta)}{2\sqrt{\mu}}\right)
		 			+
		 			\sin^2\left(\frac{\pi(\beta+1)}{2}\right)
		 		}
		 		\right)^{1/2}
		 		\\ 
		 		& \lesssim 1+  \frac{\sqrt{\mu}}{|\sigma(\zeta - \zeta_{\rm c})|}.
		 	\end{align*}
		 	On the other hand, we have by direct computations and the previous estimate that
		 	\begin{align*}
		 		|F''\circ \sigma(\alpha, \beta)|
		 		& = 
		 		\frac{\pi \gamma (2- \gamma)}{8 \sqrt{\mu}} \left|\frac{1}{F'\circ \sigma(\alpha, \beta)} \frac{1}{\cosh^2\left(\frac{\pi (\alpha + i \sqrt{\mu}\beta)}{2\sqrt{\mu} }\right)}\right|
		 		\\ 
		 		& = 
		 		\frac{\pi \gamma (2- \gamma)}{8 \sqrt{\mu}} \frac{1}{\sqrt{J\circ \sigma(\alpha, \beta)}}\left| \frac{1}{	\sinh^2\left(\tfrac{\pi\alpha}{2\sqrt{\mu}}\right)
		 			+
		 			\sin^2\left(\frac{\pi(\beta+1)}{2}\right)}\right|.
		 	\end{align*}
	 		Now, if we let $\zeta_0 = - i$, then we have
	 		\begin{align*} 
	 			\frac{|F''\circ \sigma|}{(J\circ \sigma)}
	 			\lesssim   \frac{\gamma \sqrt{\mu}}{|\sigma(\zeta - \zeta_{0})|^{1/2}|\sigma(\zeta - \zeta_{\rm c})|^{3/2}}.
	 		\end{align*}
		\end{proof}

	\subsection{Estimates related to $F'$ at the surface} We will now study various weighted estimates related to the real part of conformal mapping \eqref{Eq: Conf. map Step} at $\beta = 0$. Throughout this section, we will let $U_0(\alpha) = U(\alpha, 0)$  for $\alpha \in \R$, and recall the definition 
	\begin{align*}
		U'_0 (\alpha)
		= 
		(1-\gamma)^{1/2}
		\left(\tfrac{	
			\cosh\left(\frac{\pi(\alpha-\delta)}{2\sqrt{\mu}}\right)
		}{
			\cosh\left(\frac{\pi \alpha}{2\sqrt{\mu}}\right)
		}\right)^{1/2},
	\end{align*}
	where 
	$$\delta = - \frac{2\sqrt{\mu}}{\pi}\ln{(} 1-\gamma {)}.$$ 
	With this definition in mind, we give the first result.

	\begin{lemma}\label{Lemma: est on U_0} 
		Let $n \in \N$, $\mu, \gamma \in (0,\tfrac{1}{2})$, $\alpha \in \R$. Then for all $\alpha \in \R$ we have that
		\begin{align}\label{Eq: bound on U_0'}
			(1- \gamma) \leq U_0'(\alpha)  \leq 1.
		\end{align}
		Moreover, there is a universal constant $C>0$ such that 
		\begin{align}\label{Eq: exp decay}
			\left| U_0^{(n+2)}(\alpha)\right| & \leq \frac{\gamma}{\mu^{(n+1)/2}}C \exp \left(- \tfrac{\pi |\alpha|}{\sqrt{\mu}}\right), 
			\\ 
			\label{Eq: der U_0 in L^2}
			\left| U_0^{(n+2)}\right|_{L^2(\R)} & \leq  \frac{\gamma \mu^{1/4}}{\mu^{(n+1)/2}} C.
		\end{align}
	\end{lemma}
	\begin{proof}
		The first estimate we argue as in Step 3 of the proof of Lemma \ref{lemma: conf map}, where we use the identity
		\begin{align*}
			U'_0 (\alpha)
			= 
			(1-\gamma) 
			\left( \tfrac{\exp\left( \frac{\pi \alpha}{\sqrt{\mu}}\right) + \exp\left( \frac{\pi \delta}{\sqrt{\mu}}\right)}{\exp\left( \frac{\pi \alpha}{\sqrt{\mu}}\right)  +1} \right)^{1/2},
		\end{align*}
		which tends to  $1 = (1-\gamma) \exp\left( \frac{\pi \delta}{2\sqrt{\mu}}\right)$ as $\alpha$ goes to $-\infty$ and is strictly decreasing and tending to $(1-\gamma)$ at infinity. We also note that the exponential decay estimates follow by computing the next derivative:   
		\begin{equation*}
			U''_0(\alpha)
			= 
			-\frac{\pi \gamma (2-\gamma)}{8 \sqrt{\mu}} \frac{1}{U_0'(\alpha)} {\rm sech}^2\left(\frac{\pi \alpha}{2 \sqrt{\mu}}\right).
		\end{equation*}
		The result then follows by induction. 
	\end{proof} 
	
	We may also relate the two coordinate systems through the following relation:

	\begin{lemma}\label{Lemma: relation between systems} Let $\gamma, \mu \in (0,\tfrac{1}{2})$ and define the function ${\mathcal{A}}_{\mu} : \R^2 \rightarrow \R$ by
		\begin{align}\label{Eq: Def mathA}
			{\mathcal{A}}_{\mu}(x,y) 
			=
			\frac{1}{2}
			\left( x-y \right)	
			+
			\frac{\sqrt{\mu}}{\pi} \log \left(\tfrac{\cosh\left(\frac{\pi}{2 \sqrt{\mu}} x\right)}{\cosh\left(\frac{\pi}{2 \sqrt{\mu}} y \right)}\right).
		\end{align}
		Then there is a function $R \in L^{\infty}(\R^2)$ uniform with respect to $\gamma$ and $\mu$ such that
		\begin{align*}
			U_0^{-1}(x) - U_0^{-1}(y) 
			=
			(x-y)  
			+
			\gamma {\mathcal{A}}_{\mu}(x,y)
			+ 
			\gamma^{2} |\log\gamma|\,
			|x-y|R.
		\end{align*} 
	\end{lemma}
	
	\begin{proof} 
		First, we will simplify the expression by making an expansion of the inverse, looking at the following identity:
		\begin{align}\label{Eq: Inverse relations}
			U_0(\alpha) - U_0(\beta) = (\alpha - \beta) \int_{0}^{1} U'_0(\alpha + t(\beta - \alpha)) \, {\rm d} t.
		\end{align}
		The aim is to solve for $(\alpha(x) - \beta(y))$ up to lower order. 
		With this in mind, we let ${\rm X}(t, \alpha, \beta) = \tfrac{\pi}{2 \sqrt{\mu}}\left( \alpha + t(\beta - \alpha)\right)$ and use the following expansion
		\begin{align*}
			U_0'(\alpha + t(\beta - \alpha)) 
			& = 
			(1-\gamma)^{1/2}
			\left(
			\tfrac{\cosh\left( {\rm X} + \ln(1-\gamma) \right)}{\cosh({\rm X})}  
			\right)^{1/2}
			\\ 
			& = 
			(1-\gamma)^{1/2}
			\left(
			\cosh\left(\ln(1-\gamma)\right)   + \sinh\left(\ln(1-\gamma)\right) \tanh({\rm X})
			\right)^{1/2}
			\\ 
			& = 
			1- \tfrac{\gamma}{2}-\tfrac{\gamma}{2}  \tanh({\rm X}) + \gamma^2g({\rm X}),
		\end{align*}
		for some function $g$ depending only on the hyperbolic tangent and is therefore uniformly bounded for any ${\rm X} \in \R$. Then by direct computations, there holds,
		\begin{align}\label{Eq: lip bound}
			(\alpha - \beta) \int_{0}^1\tanh \left( \tfrac{\pi}{2 \sqrt{\mu}}\left( \alpha + t(\beta - \alpha)\right)\right) {\rm d} t
			= 
			\frac{2 \sqrt{\mu}}{\pi} \log \left(\tfrac{\cosh\left(\frac{\pi}{2 \sqrt{\mu}} \alpha\right)}{\cosh\left(\frac{\pi}{2 \sqrt{\mu}} \beta \right)}\right).
		\end{align}
		Substituting this expansion into \eqref{Eq: Inverse relations} gives
		\begin{align*}
			x - y 
			=
			(\alpha - \beta)(1- \tfrac{\gamma}{2}) 
			-
			\frac{\gamma \sqrt{\mu}}{\pi} \log \left(\tfrac{\cosh\left(\frac{\pi}{2 \sqrt{\mu}} \alpha\right)}{\cosh\left(\frac{\pi}{2 \sqrt{\mu}} \beta \right)}\right)
			+
			\gamma^2(\alpha - \beta) R,
		\end{align*}
		where the rest $R$ is some generic error function that can change from line to line, and that satisfies 
		$$|R|_{L^{\infty}(\R^2)} \leq C.$$ 
		To isolate the inverse, meaning $\alpha$ and $\beta$, solely as a function of $x$ and $y$, we note that \eqref{Eq: lip bound} also gives us a Lipschitz bound on  
		$$
		a_{\mu}(\alpha) \coloneq  \tfrac{2\sqrt{\mu}}{\pi}\log\left(\cosh\left(\tfrac{\pi}{2 \sqrt{\mu}}\alpha\right)\right).$$ 
		 In particular, using also the expansion
		$\left( 1- \tfrac{\gamma}{2}\right)^{-1} = 1+\tfrac{\gamma}{2} + \mathcal{O}(\gamma^2)$ we find that 
		\begin{align*}
			\alpha - \beta	= &  
			\left( 1 +  \tfrac{\gamma}{2}\right)
			\left( x-y \right)
			+
			\frac{\gamma \sqrt{\mu}}{\pi} \log \left(\tfrac{\cosh\left(\frac{\pi}{2 \sqrt{\mu}} \alpha \right)}{\cosh\left(\frac{\pi}{2 \sqrt{\mu}} \beta \right)}\right) 
			+ 
			\gamma^{2} (|x-y| + |\alpha - \beta|) R.
		\end{align*}
		Lastly, we use the definition $x = U_0(\alpha)$ and $y = U_0(\beta)$, whose derivatives are uniformly bounded, together with the Fundamental Theorem of Calculus to find that
		\begin{align*}
			& \left| \left(a_{\mu}(\alpha) -a_{\mu}(U_0(\alpha))\right) -  \left(a_{\mu}(\beta) -a_{\mu}(U_0(\beta))\right) \right| 
			\\ 
			&
			\hspace{4cm}
			\lesssim 
			|\alpha -\beta | \sup \limits_{ s \in \R} 
			\left( 
			\left|1- U'_0(s)\right|
			+
			\left|\tanh(\tfrac{\pi}{2 \sqrt{\mu}}s)-\tanh(\tfrac{\pi}{2 \sqrt{\mu}}U_0(s))\right|
			\right).
		\end{align*}
		The first term is of order $\gamma$ by \eqref{Eq: bound on U_0'}. For the second term, we also note that $(1-\gamma)|s| \leq |U_0(s)-U_0(0)|\leq |s|$ from \eqref{Eq: bound on U_0'}. Then use the triangle inequality and the Mean Value Theorem to get the following bound:
		\begin{align*}
			\left|\tanh(\tfrac{\pi}{2 \sqrt{\mu}}s)-\tanh(\tfrac{\pi}{2 \sqrt{\mu}}U_0(s))\right| 
			& \lesssim 
			\left|\tanh(\tfrac{\pi}{2 \sqrt{\mu}} s)-\tanh\left(\tfrac{\pi}{2 \sqrt{\mu}} (U_0(s)-U_0(0))\right)\right| + \tfrac{\pi}{2 \sqrt{\mu}}|U_0(0)|
			\\
			&\lesssim  \tfrac{\gamma}{ \sqrt{\mu}} |s| \, {\rm sech}^2\left((1-\gamma)\tfrac{\pi}{2 \sqrt{\mu}} |s|\right) + \tfrac{\pi}{2 \sqrt{\mu}}|U_0(0)|
			\\
			&\lesssim \gamma + \tfrac{\pi}{2 \sqrt{\mu}} |U_0(0)|.
		\end{align*}
		For the last term, we note that $F(0) = U_0(0)$ is defined by elementary functions in \eqref{Eq: Conf. map Step}. By a Taylor expansion we have
		$$W(f(0)) = \left(\tfrac{2}{1+(1-\gamma)^{-2}}\right)^{1/2} = 1 - \tfrac{\gamma}{2}+\mathcal{O}(\gamma^2),$$ 
		and $\tanh^{-1}(z) = \tfrac{1}{2}\left( ({\rm Log}(1+z)-{\rm Log}(1-z))\right)$ for $z \in \mathbb{C}\setminus\left((-\infty,-1]\cup[1,\infty)\right)$ where ${\rm Log}$ denotes the principal complex logarithm. Then for $\gamma \in (0,\tfrac{1}{2})$ this gives:
		$$
		U_0(0)
			=
			\frac{2 \sqrt{\mu}}{\pi}
			\left(
			(1-\gamma) \left(\tfrac{1}{2}\log \left(\tfrac{4}{\gamma}\right) + \mathcal{O}(\gamma)\right)
			-
			\left(\tfrac{1}{2} \log \left(\tfrac{4}{\gamma}\right) -\tfrac{i \pi }{2} + \mathcal{O}(\gamma)\right)
			\right)
			- i \sqrt{\mu}
			, 
		$$
		and we deduce that $|U_0(0)| \lesssim \gamma\sqrt{\mu}(1+ |\log\gamma|)$. Adding all these estimates concludes the proof.\\
	\end{proof}

	\subsection{Technical estimates} 	We first recall a simple commutator estimate to deal with weighted estimates in $\dot{H}^{1/2}_{\mu}(\R)$:
	
	\begin{lemma}
		For any $f, g \in \mathscr{S}(\R)$ and $s>1/2$ there holds,
		\begin{align}\label{Eq: [P, f]g}
			\left|\left[\mathfrak{P}({\rm D}), f\right]g\right|_{L^2(\R)} \lesssim |\partial_x f|_{L^2(\R)} |g|_{H^{s}(\R)}.
		\end{align}
	\end{lemma}
	\begin{proof}
		Following the same lines as the proof of Lemma 4.14 in \cite{WWP}, one finds that
		\begin{align*}
			\left|\left[\mathfrak{P}({\rm D}), f\right]g\right|_{L^2(\R)} 
			& 
			\lesssim  \left| \int_{\R} |\cdot - \rho| \, |\hat{f}(\cdot - \rho)| \,|\hat{g}(\rho)| \, {\rm d} \rho \right|_{L^2(\R)}.
		\end{align*}
		Then the Minkowski integral inequality gives the following result.
		\begin{align*}
			\left|\left[\mathfrak{P}({\rm D}), f\right]g\right|_{L^2(\R)} 
			& 
			\lesssim  \left| |\cdot| \, \hat{f} \right|_{L^2(\R)} \int_{\R}  |\hat{g}(\rho)| \, {\rm d} \rho,
		\end{align*}
		where $|\xi| \hat{f}(\xi) \in L^2(\R)$ by Plancherel's identity and  $\hat{g} \in L^1(\R)$ by the Cauchy-Schwarz inequality.
	\end{proof}
	
	As an application of the previous results, we can relate Sobolev estimate at low regularity with the domain of the Dirichlet--Neumann operator:
	\begin{lemma}\label{lemma: alpha to G}
		Let $\gamma, \mu \in (0,\tfrac{1}{2})$.  Then for any $f \in \mathscr{S}(\R)$ and $f_0(\alpha) = f(U(\alpha,0))$ there is a universal constant $C>0$ such that there holds,
		\begin{align}\label{Eq: dotH1/2 equiv}
			\frac{1}{C}|f|_{\dot{H}^{1/2}_{\mu}(\R)} \leq \left| f_0 \right|_{\dot{H}_{\mu}^{1/2}(\R)} & \leq C |f|_{\dot{H}^{1/2}_{\mu}(\R)},
		\end{align}
		and for $k = 0, 1$ we have
		\begin{align}  
			\label{Eq: dalpha to G} 
			\left|\partial_\alpha \left(f(U_0)\right)\right|_{H^k(\R)} 
			& \leq 
			(1+\gamma^{k}\mu^{-k/4})C |f|_{\dot{\mathcal{H}}_{\mu}^{k+1}(\R)}.
		\end{align}

	\end{lemma} 
	\begin{proof} We split the proof into three simple steps. \\ 
		
		\noindent
		{\bf Step 1.} \textit{Proof of \eqref{Eq: dotH1/2 equiv}.} Note that $\|\nabla^{\mu}(f^{\mathfrak{h}}\circ F_{\mu})\|_{L^2({\rm S})} = \|\nabla^{\mu}f^{\mathfrak{h}}\|_{L^2(\Omega)}$ when making a conformal change of variables. Then we deduce from \eqref{Eq: Phi equiv psi} the equivalence
		\begin{equation*}
			\left| f_0 \right|_{\dot{H}_{\mu}^{1/2}(\R)}  \sim |f|_{\dot{H}^{1/2}_{\mu}(\R)}.
		\end{equation*}

		\noindent
		{\bf Step 2.} \textit{Proof of \eqref{Eq: dalpha to G}$_{k=0}$.} Here we use Remark \ref{Remark: change of variables} to find the relation $\partial_y = \tfrac{1}{U'_0} \partial_{\beta}$ on $\beta = 0$. This allows us to write the Dirichlet--Neumann operator in terms of the Fourier multiplier
		\begin{align}\label{Eq: DN fourier}
			U_0'(\alpha)\,  \left({\mathcal G}_{\mu}  f \right) \circ U_0(\alpha)
			& =
			\tfrac{1}{\sqrt{\mu}} |{\rm D}| \tanh(\sqrt{\mu}|{\rm D}|)\notag  f\left(U_0(\alpha)\right)
			\\ 
			&
			\eqcolon
			{\rm g}_{\mu}({\rm D}) f_0
			.
		\end{align}
		Moreover, by \eqref{Eq: bound on U_0'} we have that $\tfrac{1}{2}\leq U_0'(\alpha)\leq 1$ and by Plancherel's identity we find
		\begin{align*}
			\left|\partial_\alpha f_0 \right|_{L^2(\R)} 
			& \leq 
			C 
			\left(  
			\left|f_0 \right|_{\dot{H}_{\mu}^{1/2}(\R)} + \left| {\rm g}_{\mu} (\rm D) f_0\right|_{L^2(\R)}
			\right)
			\\
			& \leq  
			C 
			\left(  |f|_{\dot{H}^{1/2}_{\mu}(\R)} + \left| {\mathcal G}_{\mu} f \right|_{L^2(\R)}
			\right).
		\end{align*}

		\noindent
		{\bf Step 3.} \textit{Proof of \eqref{Eq: dalpha to G}$_{k=1}$.}  First, we recall the definition of the Fourier multiplier  $\mathfrak{P}(\xi) = |\xi|\langle\sqrt{\mu} |\xi|\rangle^{-1/2}$ and apply the previous case to $\partial_{\alpha}f_0$. This gives
		\begin{align*}
			\left|\partial_\alpha^2 f_0 \right|_{L^2(\R)}    
			& 
			\lesssim
			\left| \partial_{\alpha} \mathfrak{P}({\rm D})f_0\right|_{L^2(\R)}
			+ 
			\left| 
			\frac{1}{U_0'}{\rm g}_{\mu}({\rm D})\partial_{\alpha}f_0
			\right|_{L^2(\R)} 
			\\ 
			& \eqcolon {\rm N}_1 + {\rm N}_2.
		\end{align*}
		For the first term we apply the previous case once more to find
		\begin{align*}
			{\rm N}_1 
			\lesssim & 
			\left| {\rm g}_{\mu} (\rm D)  f_0\right|_{L^2(\R)}
			+
			\left| \mathfrak{P}({\rm D})\left( \frac{1}{U_0'}{\rm g}_{\mu} (\rm D) f_0\right)\right|_{L^2(\R)}
			+
			\left| \left[\mathfrak{P}({\rm D}),U_0'\right] \frac{1}{U_0'}{\rm g}_{\mu} (\rm D) f_0\right|_{L^2(\R)}
			\\
			\lesssim & 
			\left| {\mathcal G}_{\mu} f\right|_{L^2(\R)}
			+
			\left| {\mathcal G}_{\mu} f\right|_{\dot{H}^{1/2}_{\mu}(\R)}
			+
			\left| \left[\mathfrak{P}({\rm D}),U_0'\right] \left({\mathcal G}_{\mu}  f \right) \circ U_0(\alpha)\right|_{L^2(\R)}.
		\end{align*}
		To conclude the estimate on ${\rm N}_1$, we control the commutator by trading regularity with precision in the small parameters through \eqref{Eq: [P, f]g} and \eqref{Eq: der U_0 in L^2} to find that
		\begin{align*}
			\left| \left[\mathfrak{P}({\rm D}), U_0'\right] \left({\mathcal G}_{\mu}  f \right) \circ U_0(\alpha)\right|_{L^2(\R)}
			& \lesssim 
			\left| U_0''\right|_{L^2(\R)} \left| \left({\mathcal G}_{\mu}  f \right) \circ U_0(\alpha)\right|_{H^1(\R)} 
			\\ 
			& 
			\lesssim \gamma\mu^{-1/4}\left| \left({\mathcal G}_{\mu}  f \right) \circ U_0(\alpha)\right|_{H^1(\R)}.
		\end{align*}
		Now, we may simply use the case $k=1$ on $\partial_{\alpha}\left( \left({\mathcal G}_{\mu}  f \right) \circ U_0(\alpha)\right)$ to get the final bound
		\begin{align*}
			{\rm N}_1 
			\leq
			(1+\gamma\mu^{-1/4})C\left(
			\left| {\mathcal G}_{\mu} f\right|_{L^2(\R)}
			+
			\left| {\mathcal G}_{\mu} f\right|_{\dot{H}^{1/2}_{\mu}(\R)}
			+
			\left| {\mathcal G}_{\mu}^2 f\right|_{L^2(\R)}\right).
		\end{align*}
		The remaining term, ${\rm N}_2$, is simpler where the commutator is computed explicitly and we find:
		\begin{align*}
			{\rm N}_2
			& \leq 
			\left| \partial_{\alpha} \left(
			\frac{1}{U_0'}{\rm g}_{\mu}({\rm D})f_0\right)
			\right|_{L^2(\R)} 
			+
			\left| 
			\frac{U''}{U_0'} \frac{1}{U_0'}{\rm g}_{\mu}({\rm D})f_0
			\right|_{L^2(\R)} 
			\\ 
			& \leq (1 + \gamma\mu^{-1/4})C |f|_{\dot{\mathcal{H}}_{\mu}^2(\R)}.
		\end{align*}
	\end{proof}

	The next results will be used to trade the non-local spaces ${\dot{\mathcal H}}^{n/2}(\R)$ with usual Sobolev spaces. Since we will only need to apply such estimates for the data of the problem, we may assume local behavior of the function. In particular, to remove the singularity in $\mu$ we will suppose $f$ vanishes $\{|x|\leq c_0\}$ for some $c_0 \in (0, 1)$ independent from the small parameters. 

	\begin{lemma}\label{lemma: G to dx}
		Let $\gamma, \mu \in (0,\tfrac{1}{2})$, $n \in \N$. Also let $f \in \mathscr{S}(\R)$ be  such that it vanishes on $\{|x|\leq c_0\}$ for some $c_0 \in (0, 1)$ independent of $\gamma$ and $\mu$. Then there is a universal constant $C>0$ such that there holds,
		\begin{align}\label{Eq: G to dx} 
			\left| f\right|_{\dot{\mathcal H}_{\mu}^{n/2}(\R)}
			& \leq  
				|\partial_x f|_{H^{n-1}(\R)}. 
		\end{align}
		
	\end{lemma}
	\begin{remark}
		For $n = 1$ the result holds without the vanishing condition.
	\end{remark}

	\begin{proof}
		We first note that imposing the vanishing condition on $f$ implies $f_0 = f\circ U_0$ also satisfy the same condition for $\alpha \in (-\tilde c_0, \tilde c_0)$ for some $\tilde{c_0}\in (0,1)$ using Lemma \ref{Lemma: relation between systems}. With this in mind, we consider the first cases before a general pattern emerges. \\

			\noindent
			\underline{Case $n=1$}. We use Plancherel's identity to deduce that 
			$$|f|_{\dot{H}^{1/2}_{\mu}(\R)} = |\mathfrak{P}({\rm D}) f|_{L^2(\R)} \leq |\partial_x f|_{L^2(\R)}.$$

			\noindent
			\underline{Case $n=2$}. We again use the definition of ${\mathcal G}_{\mu} f$ in conformal coordinates and Plancherel's identity to find that
			\begin{align*}
				\left|{\mathcal G}_{\mu} f\right|_{L^2(\R)} 
				& \lesssim \left|{\rm g}_{\mu}({\rm D}) f_0\right|_{L^2(\R)}
				\\ 
				& \lesssim \left| \partial_{\alpha}^2 f_0\right|_{L^2(\R)}.
			\end{align*}
			In other words, using the chain rule, we need to control the following terms
			\begin{equation*}
				\partial_{\alpha}^2\left(f(U_0(\alpha)) \right)
				= 
				U''_0(\alpha) \left(\partial_x f \right)\circ U_0 (\alpha)
				+
				\left( U'_0(\alpha) \right)^2 \left(\partial_x^2 f \right)\circ U_0 (\alpha).
			\end{equation*}
			Then  the vanishing assumption together with the pointwise estimate in Lemma \ref{Lemma: est on U_0} removes the singular behaviour in $\mu$ where
			$$\mu^{-k}e^{-|\alpha|/\sqrt{\mu}} \lesssim 1,$$
			for any $k\in\N$ when $|\alpha|>\tilde{c}_0$. We therefore find that
			\begin{align*}
				\left|{\mathcal G}_{\mu} f\right|_{L^2(\R)} 
				& \lesssim  \left|\partial_x f \right|_{H^1(\R)}.
			\end{align*}
			\noindent
			\underline{Case $n=3$}. We argue as in the previous two cases, together with $\partial_x = \tfrac{1}{U'_0} \partial_{\alpha}$ on $\beta = 0$, to find that
			\begin{align*}
				\left|{\mathcal G}_{\mu} f\right|_{\dot{H}^{1/2}_{\mu}(\R)} 
				&  \lesssim
				\left| \partial_x {\mathcal G}_{\mu} f_0\right|_{L^2(\R)} 
				\\ 
				& \lesssim 
				\left|\partial_{\alpha}^3f_0\right|_{L^2(\R)}
				+
				\left|U''_0\,  {\rm g}_{\mu}({\rm D}) f_0\right|_{L^2(\R)}.
			\end{align*}
			The first term is trivially treated using the support of $f$ and the exponential decay away from the step.  The only difficulty is to utilize the local condition when a non-local operator acts on the function. To deal with this issue, we will divide the integral into two parts
			%
			%
			\begin{align}\label{Eq: final boss}
				\left|U''_0\,  {\rm g}_{\mu}({\rm D}) f_0\right|_{L^2(\R)} & \lesssim \mu^{-1/2} \left|  {\rm g}_{\mu}({\rm D}) f_0\right|_{L^2(\{|\alpha|<\tilde{c}_0/2\})}+ \exp\left(-\tfrac{\tilde{c_0}}{3\sqrt{\mu}}\right)\left|  {\rm g}_{\mu}({\rm D}) f_0\right|_{L^2(\{|\alpha|>\tilde{c}_0/2\})},
			\end{align}
			where we used the support and \eqref{Eq: exp decay}. The second term is trivially bounded, by expanding the Fourier multiplier and compensating any loss with the exponential function. It only remains to bound the first term, and to do so, we write
			\begin{align*}
				 {\rm g}_{\mu}({\rm D}) f_0  = 	\frac{\tanh(\sqrt{\mu} |\rm D|)}{\sqrt{\mu} |\rm D|} \partial_{\alpha}^2 f_0.
			\end{align*}
			This multiplier can be written in terms of a convolution kernel \cite{EhrnstromJohnsonClaassen2019}: 
			\begin{align*} 
				\frac{\tanh(\sqrt{\mu} |\rm D|)}{\sqrt{\mu} |\rm D|} \partial_{\alpha}^2 f_0
				= 
				\frac{1}{\pi \sqrt{\mu}} \int_{\R}
				\log \coth\left( \tfrac{\pi |\alpha-\beta| }{4 \sqrt{\mu}}\right)\partial_{\beta}^2 f_0(\beta) \, {\rm d}\beta,
			\end{align*}
			where the kernel has an exponential decay, as proved in Lemma 2.2 in \cite{EhrnstromJohnsonClaassen2019}, for $1 \lesssim |\alpha-\beta|$. Since ${\rm supp}(f_0) \subset \{|\beta|>\tilde{c}_0\}$ we have that $|\alpha- \beta| \geq \tilde{c}_0/2$ for $|\alpha|<\tilde{c}_0/2$, and by the Cauchy-Schwarz inequality we deduce the following estimate
			\begin{align}
				\left|\frac{\tanh(\sqrt{\mu} |\rm D|)}{\sqrt{\mu} |\rm D|} \partial_{\alpha}^2 f_0 (\alpha)\right|
				& \leq 
				\frac{1}{ \pi \sqrt{\mu} } 
				\left(\int_{|\alpha - \beta| \geq \tilde{c}_0/2} 
				\left(\log \coth\left( \tfrac{\pi |\alpha-\beta | }{4 \sqrt{\mu}}\right)\right)^2  \, {\rm d} \beta \right)^{1/2} \|\partial_{\alpha}^2f_0\|_{L^2(\R)}
				\\
				& \leq 
				\mu^{-1}\exp\left(-\tfrac{\tilde{c_0}}{4\sqrt{\mu}}\right) \|\partial_{\alpha}^2f_0\|_{L^2(\R)}. \label{Eq: kernel decay}
			\end{align}
			As a result, we find that
			\begin{align*}
				\left|  {\rm g}_{\mu}({\rm D}) f_0\right|_{L^2(\{|\alpha|<\tilde{c}_0/2\})}
				& \lesssim 
				 \exp\left(-\tfrac{\tilde{c_0}}{10\sqrt{\mu}}\right) \|\partial_{\alpha}^2f_0\|_{L^2(\R)}.
			\end{align*}
			The exponential decay in the small parameter and the chain rule then gives in the end
			\begin{align*}
				\left|{\mathcal G}_{\mu} f\right|_{\dot{H}^{1/2}_{\mu}(\R)} \lesssim |\partial_x f|_{H^2(\R)}.
			\end{align*} 
		
			\noindent
			\underline{Case $n \geq 4$}.  For higher order derivatives we either have terms of the form  $\partial_{\alpha}^k f_0 \in L^2(\R)$, for $0 \leq  k \leq n$, which are treated using the support. Or we have the weight times terms of the form ${\rm g}_{\mu}(\rm D) f_0 \in L^2(\R)$. Then away from the origin, we can use the decay from the weight as in \eqref{Eq: final boss}, while close to the origin we use the decay from the kernel as in \eqref{Eq: kernel decay}.
	\end{proof}

	\bibliographystyle{plain}
	\bibliography{BibREF.bib}
\end{document}